\documentclass[12pt]{amsart}
\usepackage{amsmath}
\usepackage{amsfonts}
\usepackage{amssymb}
\begin{document}
\title[Extensions of the tensor algebra and their applications]
{Extensions of the tensor algebra and \\ their applications}
\author{Minoru ITOH}
\date{}
\address{Department of Mathematics and Computer Science, 
          Faculty of Science,
          Kagoshima University, Kagoshima 890-0065, Japan}
\email{itoh@sci.kagoshima-u.ac.jp }
\thanks{This research was partially supported by 
JSPS Grant-in-Aid for Young Scientists (B) 20740020.}
\begin{abstract}
This article presents a natural extension of the tensor algebra.
In addition to ``left multiplications'' by vectors,
we can consider ``derivations'' by covectors 
as basic operators on this extended algebra.
These two types of operators satisfy 
an analogue of the canonical commutation relations.
This algebra and these operators have some applications:
(i) applications to invariant theory related to tensor products,
and
(ii) applications to immanants.
The latter one includes a new method to study
the quantum immanants 
in the universal enveloping algebras of the general linear Lie algebras
and their Capelli type identities (the higher Capelli identities).
\end{abstract}
\maketitle
\theoremstyle{plain}
   \newtheorem{theorem}{Theorem}[section]
   \newtheorem{proposition}[theorem]{Proposition}%%[section]
   \newtheorem{lemma}[theorem]{Lemma}%%[section]
   \newtheorem{corollary}[theorem]{Corollary}%%[section]
\theoremstyle{remark}
   \newtheorem*{remark}{Remark}
   \newtheorem*{remarks}{Remarks}
\numberwithin{equation}{section}
\newcommand{\mybinom}[2]{\left(\!\genfrac{}{}{0pt}{}{#1}{#2}\!\right)}
\newcommand{\bibinom}[2]{\left(\!\!\left(\!\genfrac{}{}{0pt}{}{#1}{#2}\!\right)\!\!\right)}
%
%%%%%%%%%%%%%%%%%%%%%%%%%%%%%%%%%%%%%%%%%%%%%%%%%%%%%%%%%%%%%%%%%%%%%%%%%%%%%%%%%%
%
\section*{Introduction}
%
%%%%%%%%%%%%%%%%%%%%%%%%%%%%%%%%%%%%%%%%%%%%%%%%%%%%%%%%%%%%%%%%%%%%%%%%%%%%%%%%%
%
In this article, 
we introduce some extensions of the tensor algebra.
The most basic one is constructed as a vector space as follows:
$$
   \bar{T}(V) 
   = \bigoplus_{p \geq 0}
   V^{\otimes p} \otimes_{\mathbb{C}S_p} \mathbb{C}S_{\infty}.
$$
For this $\bar{T}(V)$, we can naturally define an associative algebra structure.
The ordinary tensor algebra $T(V)$ can be regarded as a subalgebra of this algebra.
This extended algebra $\bar{T}(V)$ is remarkable,
because we can consider a natural ``derivation'' $L(v^*)$
determined from any covector $v^* \in V^*$ as an operator on $\bar{T}(V)$.
An analogue of the canonical commutation relations holds
between these derivations and 
the left multiplications $L(v)$ by vectors $v \in V$
(Theorem~\ref{thm:analogue_of_CCR}).
It is also natural to call these multiplications
and derivations ``creation operators'' and ``annihilation operators,''
respectively
(namely, we can regard this $\bar{T}(V)$ as an analogue of
the Boson and Fermion Fock spaces).
The algebra $\mathcal{L}(V)$ generated by these two types of operators
is naturally isomorphic to 
$$
   \bigoplus_{p,q \geq 0}
   V^{\otimes p} \otimes_{\mathbb{C}S_p} \mathbb{C}S_{\infty}
   \otimes_{\mathbb{C}S_q} V^{* \otimes q}
$$
as vector spaces,
and we can regard this operator algebra $\mathcal{L}(V)$ as an analogue 
of the Weyl algebra and the Clifford algebra
(actually this contains these algebras naturally as quotient algebras).

This framework has some applications to representation theory and
invariant theory related to tensor products.
For example, 
we can prove the Schur--Weyl duality and its generalization 
by a simple calculation in $\mathcal{L}(V)$ (Theorem~\ref{thm:generalization_of_the_Schur_Weyl_duality}).
We also have
an analogue of the $(GL_n(\mathbb{C}), GL_{n'}(\mathbb{C}))$ duality due to Howe \cite{Ho},
and this prodives us a natural correspondence between 
the center of the universal enveloping algebra $U(\mathfrak{gl}_n)$
and some invariant ``differential operators'' on $T(\mathbb{C}^n \otimes \mathbb{C}^{n'})$
(Theorem~\ref{thm:analogue_of_Howe_duality}).
We can describe this correspondence 
as a Capelli type identity on the tensor algebra
$T(\mathbb{C}^n \otimes \mathbb{C}^{n'})$
(Theorem~\ref{thm:Capelli_identity_in_T}).
Moreover, using this Capelli type identity,
we can determine the $SL_n(\mathbb{C})$-invariants in 
$T(\mathbb{C}^n \otimes \mathbb{C}^{n'})$ (Theorem~\ref{thm:FFT_for_SL_in_T}).

Our extensions of the tensor algebra 
are also useful to treat a matrix function called ``immanant.''
This is parallel to the fact
that the exterior algebras and the symmetric tensor algebras
are useful to treat the determinant and the permanent.
This method can be developed to study the ``quantum immanants,''
a basis of the center of the universal enveloping algebra $U(\mathfrak{gl}_n)$
introduced by Okounkov \cite{O1}.
The quantum immanants have been studied by using the $R$-matrix method,
the fusion procedure, and representation theory of the Yangian $Y(\mathfrak{gl}_n)$ 
(see \cite{O1}, \cite{O2}, \cite{OO}, \cite{M1}, \cite{N2};
Section~7.4 of \cite{M2} is helpful to look at the whole picture on this issue).
Instead of these traditional approaches, we can use our extensions of the tensor algebra.
Namely, making use of our algebras as formal variables,
we can prove various fundamental relations for the quantum immanants 
by simple calculations.
It is not to say that our approach is more powerful than established approaches.
However, we can regard this approach as an advanced version 
of the exterior calculus used to study Capelli type identities 
in \cite{IU}, \cite{I1}--\cite{I6}, \cite{U2}--\cite{U5}, \cite{Ha}, \cite{Wa} 
(and also in Section~\ref{sec:Capelli_and_FFT} of this article),
and we can manipulate the noncommutativity with similar skills.
The author thinks that this approach is one of the best ways to study the quantum immanants,
and expects further developments.

Finally, combining several ideas in this article, 
we give Capelli type identities for the quantum immanants 
on $T(\mathbb{C}^n \otimes \mathbb{C}^{n'})$
as higher generalizations of the Capelli type identity given in Section~\ref{sec:Capelli_and_FFT}
(Theorems~\ref{thm:higher_Capelli_on_T} and \ref{thm:higher_Capelli_for_preimmanants_on_T}).

This article is organized as follows.
In Section~\ref{sec:def_of_our_alg}, 
we introduce the algebra $\bar{T}(V)$ as an extension of the tensor algebra $T(V)$.
This is the base of all studies in this article.
In Section~\ref{sec:multiplications_and_derivations}, 
we define multiplication and derivation operators acting on $\bar{T}(V)$,
and give their quick applications.
The algebra $\mathcal{L}(V)$ generated 
by these operators contains the Weyl algebra and the Clifford algebra
naturally as quotient algebras.
In Section~\ref{sec:commutants}, we use these operators to study 
a generalization of the Schur--Weyl duality and an analogue of Howe duality.
In Section~\ref{sec:Capelli_and_FFT}, 
we describe the action of the Capelli element
on the tensor algebra $T(\mathbb{C}^n \otimes \mathbb{C}^{n'})$ 
using the multiplication and derivation operators.
This description can be regarded as an analogue of the Capelli identity
(we also give its ``higher'' generalization in Section~\ref{sec:higher_Capelli}).
Moreover we give the first fundamental theorem of invariant theory for 
the action of $SL_n(\mathbb{C})$ on $T(\mathbb{C}^n \otimes \mathbb{C}^{n'})$ 
as an application of this Capelli type identity.
In Section~\ref{sec:def_of_variants}, 
we give some variants of the algebras $\bar{T}(V)$ and $\mathcal{L}(V)$.
These variants work as a stage to study immanants and quantum immanants
in later sections.
In Section~\ref{sec:immanants}, 
we introduce some noncommutative immanants and see their fundamental properties.
We also introduce the notion of ``preimmanants.''
In Section~\ref{sec:immanants_and_algebras}, 
we express these noncommutative immanants using our algebras.
In Section~\ref{sec:quantum_immanants}, we develop this method to treat the quantum immanants,
and show their fundamental properties.
Finally, in Section~\ref{sec:higher_Capelli}, 
we prove the higher Capelli identity and its analogue on $T(\mathbb{C}^n \otimes \mathbb{C}^{n'})$.

Recently, the author constructed $q$-analogues of $\bar{T}(V)$ and $\mathcal{L}(V)$,
which yield
a simple proof of the $q$-Schur--Weyl duality between 
the quantum enveloping algebra $U_q(\mathfrak{gl}_n)$ 
and the Iwahori--Hecke algebra of type $A$.
This result will be written somewhere else.

The author hopes that these extensions of the tensor algebra 
will be useful to study noncommutative invariant theory,
various issues related to tensors,
and furthermore supersymmetry theory.

\medskip

The author is grateful to Professor T\^oru Umeda for fruitful discussions.
He is also grateful to Professor Alexander Molev for pointing out 
some developments in the study of the quantum immanants and the higher Capelli identities.
Finally he thanks to the referee for valuable comments which improved this paper.

%
%%%%%%%%%%%%%%%%%%%%%%%%%%%%%%%%%%%%%%%%%%%%%%%%%%%%%%%%%%%%%%%%%%%%%%%%%%%%%%%%%%
%
\section{Definition of the algebra $\bar{T}(V)$}\label{sec:def_of_our_alg}
%
%%%%%%%%%%%%%%%%%%%%%%%%%%%%%%%%%%%%%%%%%%%%%%%%%%%%%%%%%%%%%%%%%%%%%%%%%%%%%%%%%
%
First we define an algebra $\bar{T}(V)$.
This is the most fundamental algebra among the extensions of the tensor algebra
which we discuss in this article.

%%%%%%%%%%%%%%%%%%%%%%%%%%%%%%%%%%%%%%%%%%%%%%%%%%%%%%%%%%%%%%%%%%%%%%%%%%%%%%%%%
\subsection{}
%%%%%%%%%%%%%%%%%%%%%%%%%%%%%%%%%%%%%%%%%%%%%%%%%%%%%%%%%%%%%%%%%%%%%%%%%%%%%%%%%
%
Let us consider an $n$ dimensional $\mathbb{C}$-vector space $V$
and its tensor algebra $T(V) = \bigoplus_{p \geq 0}T_p(V)$.
The homogeneous part $T_p(V)$ of $T(V)$ is the $p$ fold tensor product of $V$:
$T_p(V) = V^{\otimes p}$.
We consider the natural (right) action of the symmetric group $S_p$ 
on this $V^{\otimes p}$. 
Namely $\sigma \in S_p$ acts on $v_p \cdots v_1 \in V^{\otimes p}$ by
$$
   v_p \cdots v_1 \sigma = v_{\sigma(p)} \cdots v_{\sigma(1)}.
$$
Here, we omit the symbol ``$\otimes$'' for elements of $T(V)$.
Moreover, we often employ the numbering of vectors running from right to left, 
when considering a right action of the symmetric group.
We regard $T_p(V) = V^{\otimes p}$ as a right $\mathbb{C}S_p$-module with this action.
In addition, for $q \geq 0$,
we regard $\mathbb{C}S_{p+q}$ 
as a $(\mathbb{C}S_p, \mathbb{C}S_{p+q})$-bimodule
with left and right multiplications
(we embed $\mathbb{C}S_p$ in $\mathbb{C}S_{p+q}$ according to 
the canonical inclusions $S_0 \subset S_1 \subset \cdots$ of symmetric groups).
We consider the tensor product of these right module and bimodule 
(namely an induced representation):
$$
   T^{(q)}_p(V) 
   = V^{\otimes p} \otimes_{\mathbb{C}S_p} \mathbb{C}S_{p+q}
   \simeq \operatorname{Ind}_{\mathbb{C}S_p}^{\mathbb{C}S_{p+q}}V^{\otimes p}.
$$
Since $T^{(0)}_p(V)$
is naturally isomorphic to $T_p(V) = V^{\otimes p}$,
we have the inclusions
\begin{equation}\label{eq:inclusions_of_T}
   T_p(V) = T^{(0)}_p(V) \subset T^{(1)}_p(V) \subset \cdots.
\end{equation}

Let us consider a much larger space
$$
   T^{(\infty)}_p(V) = V^{\otimes p} \otimes_{\mathbb{C}S_p} \mathbb{C}S_{\infty}
   \simeq \operatorname{Ind}_{\mathbb{C}S_p}^{\mathbb{C}S_{\infty}}V^{\otimes p}.
$$
Here $S_{\infty}$ denotes the infinite symmetric group,
namely the inductive limit of 
the sequence $S_0 \subset S_1 \subset \cdots$.
Let us denote this $T^{(\infty)}_p(V)$ by $\bar{T}_p(V)$ simply.
We can regard $\bar{T}_p(V)$
as the inductive limit of the sequence (\ref{eq:inclusions_of_T}).

Noting that $\bar{T}_0(V) \simeq \mathbb{C}S_{\infty}$,
we consider the direct sum of $\bar{T}_0(V), \bar{T}_1(V),\ldots$:
$$
   \bar{T}(V) 
   = \bigoplus_{p \geq 0} \bar{T}_p(V).
$$
For this $\bar{T}(V)$, 
we can naturally define a structure of graded algebra.
That is, for 
$$
   \varphi = v_p \cdots v_1 \sigma \in \bar{T}_p(V), \qquad
   \varphi' = v'_{p'} \cdots v'_{1} \sigma' \in \bar{T}_{p'}(V),
$$
we define the product $\varphi \varphi' \in \bar{T}_{p+p'}(V)$ by
$$
   \varphi \varphi'
   = (v_p \cdots v_1 \sigma) (v'_{p'} \cdots v'_{1} \sigma') 
   = 
   v_p \cdots v_1 
   v'_{p'} \cdots v'_{1} 
   \alpha^{p'}(\sigma) \sigma'.
$$
Here $\alpha$ is the group endomorphism of $S_{\infty}$ defined by
\begin{align*}
   \alpha(\sigma) \colon 
   k & \mapsto \sigma(k-1)+1 \quad \text{for $k \geq 2$}, \\
   1 & \mapsto 1,
\end{align*}
namely we put $\alpha(s_i) = s_{i+1}$ for the adjacent transposition $s_i = (i \,\,\, i+1)$.
We can easily see that 
this multiplication is well defined and moreover associative.

Let us see some relations for this multiplication.
First,  for $v$, $w \in V$, we have
\begin{equation}\label{eq:relation_1_in_T}
   w v = v w s_1, \qquad
   s_i v = v s_{i+1}.
\end{equation}
Moreover the defining relations of the symmetric group also hold in $\bar{T}(V)$:
\begin{equation}\label{eq:relation_2_in_T}
   s_i^2 = 1, \qquad
   s_i s_{i+1} s_i = s_{i+1} s_i s_{i+1}, \qquad
   s_i s_j = s_j s_i \text{ for $|i-j| > 1$}.
\end{equation}
Actually, these relations form the defining relations of the algebra $\bar{T}(V)$
as seen in Theorem~\ref{thm:def_by_gen_and_relation} below.

It is easily seen that
$T^{(q)}(V) = \bigoplus_{p \geq 0}T^{(q)}_p(V)$ 
is a subalgebra of $\bar{T}(V)$.
In particular, $T(V) \simeq T^{(0)}(V)$ is also a subalgebra,
and the restriction of the multiplication of $\bar{T}(V)$ is equal
to the ordinary multiplication of the tensor algebra $T(V)$.
Thus, we can regard $T^{(q)}(V)$ and $\bar{T}(V)$ as extensions of the ordinary tensor algebra.

%%%%%%%%%%%%%%%%%%%%%%%%%%%%%%%%%%%%%%%%%%%%%%%%%%%%%%%%%%%%%%%%%%%%%%%%%%%%%%%%%
\subsection{}
%%%%%%%%%%%%%%%%%%%%%%%%%%%%%%%%%%%%%%%%%%%%%%%%%%%%%%%%%%%%%%%%%%%%%%%%%%%%%%%%%
%
Let us consider a ``canonical form'' of elements in  $T^{(q)}(V)$.
We denote the set $\{ 1,\ldots,n \}$ by $[n]$,
and consider a sequence $I = (i_1,\ldots,i_p) \in [n]^p$
(we assume that $I$, $J,\ldots$ mean the sequences 
$I = (i_1,\ldots,i_p)$, $J = (j_1,\ldots,j_p), \ldots$ throughout this article).
We put $I! = m_1! \cdots m_n!$,
where $m_1,\ldots,m_n$ are the multiplicities of $1,\ldots,n$ in the sequence
$I = (i_1,\ldots,i_p) \in [n]^p$, respectively.
The symmetric group $S_p$ naturally acts on $[n]^p$ by
$\sigma(I) = (i_{\sigma(1)},\ldots,i_{\sigma(p)})$.
We denote by $(S_p)_I$
the stabilizer subgroup of the sequence $I \in [n]^p$.
Namely, we put $(S_p)_I = \{ \sigma \in S_p \,|\, \sigma(I) = I \}$.
Then the order of this group $(S_p)_I$ is equal to $I!$,
and the element
$$
   s_I = \frac{1}{I!} 
   \sum_{\sigma \in (S_p)_I} \sigma
$$
in $\mathbb{C}S_p$ is idempotent.
Moreover, for a basis $e_1,\ldots,e_n$ of $V$, we have the relation
\begin{equation}\label{eq:symmetric_expression}
   e_{i_p} \cdots e_{i_1}
   = e_{i_p} \cdots e_{i_1} s_I.
\end{equation}
From this, we see the following assertion:

\begin{proposition}\label{prop:canonical_form_of_elements_of_T}\sl
   By fixing a basis $e_1,\ldots,e_n$ of $V$,
   any element in $T^{(q)}_p(V)$ can be expressed uniquely 
   as a sum of elements in the form $e_{i_p} \cdots e_{i_1} t$.
   Here $I = (i_1,\ldots,i_p)$ is a sequence in $\bibinom{[n]}{p}$,
   and $t$ is an element of the left ideal 
   $s_I \mathbb{C}S_{p+q}$ of $\mathbb{C}S_{p+q}$.
\end{proposition}

Here we put
\begin{align*}
   \mybinom{[n]}{r} 
   & = \{ (i_1,\ldots,i_r) \in [n]^r 
   \,|\, i_1 < \cdots < i_r \}, \\
   \bibinom{[n]}{r} 
   & = \{ (i_1,\ldots,i_r) \in [n]^r 
   \,|\, i_1 \leq \cdots \leq i_r \}.
\end{align*}

\begin{proof}[Proof of Proposition~{\sl\ref{prop:canonical_form_of_elements_of_T}}]
It suffices to show the uniqueness.
Indeed we can change the order of $p$ vectors freely by applying a suitable element of $S_p$ from right,
and we can assume that the factor in $\mathbb{C}S_{p+q}$
belongs $s_I \mathbb{C}S_{p+q}$ by (\ref{eq:symmetric_expression}).

The uniqueness is seen as follows.
Let $I_1,\ldots,I_r$ be distinct elements of $\bibinom{[n]}{p}$.
Our task is to show $s_{I_i} t_i = 0$ on the assumption 
$\sum_{i=1}^r e_{I_i} t_i = \sum_{i=1}^r e_{I_i} s_{I_i} t_i = 0$
with $t_1,\ldots,t_r \in \mathbb{C}S_{p+q}$.
Here we denote $e_{i_p} \cdots e_{i_1}$ simply by $e_I$.
First of all, we see $e_{I_1} t_1 = \cdots = e_{I_r} t_r = 0$
from this assumption.
Indeed the right action of $S_p$ on $V^{\otimes p}$ only changes the order of $p$ vectors.
Thus it suffices to deduce $s_I t = 0$ from the equality $e_I t = e_I s_I t = 0$.
This is done by considering the right cosets of $(S_p)_I$ in $S_{p+q}$.
First $e_{\sigma_1(I)}, \ldots, e_{\sigma_r(I)}$ are all equal  
if $\sigma_1,\ldots,\sigma_r \in S_{p+q}$ are in a same coset.
Secondly $e_{\sigma_1(I)}, \ldots, e_{\sigma_r(I)}$ are linearly independent
if $\sigma_1,\ldots,\sigma_r$ are in different cosets.
Thus, when $e_I t = e_I s_I t = 0$
with $t = \sum_{\sigma \in S_{p+q}} c_{\sigma} \sigma \in \mathbb{C}S_{p+q}$, 
we have $\sum_{\sigma \in C} c_{\sigma} = 0$
for each coset $C  \in (S_p)_I \backslash S_{p+q}$,
namely $s_I t = 0$.
\end{proof}

%%%%%%%%%%%%%%%%%%%%%%%%%%%%%%%%%%%%%%%%%%%%%%%%%%%%%%%%%%%%%%%%%%%%%%%%%%%%%%%%%
\subsection{}
%%%%%%%%%%%%%%%%%%%%%%%%%%%%%%%%%%%%%%%%%%%%%%%%%%%%%%%%%%%%%%%%%%%%%%%%%%%%%%%%%
%
We can also define the algebra $\bar{T}(V)$ 
in terms of generators and relations:

\begin{theorem}\label{thm:def_by_gen_and_relation}\sl
   The algebra $\mathcal{A}_n$ defined by the following generators and relations 
   is isomorphic to~$\bar{T}(V)$ by regarding $e_1,\ldots,e_n$ as 
   a basis of $V$:
   \begin{align*}
      \text{generators: } & e_1,\ldots,e_n,s_1,s_2,\ldots, \\
      \text{relations: } 
      & e_b e_a =  e_a e_b s_1, \quad
      s_i e_a=  e_a s_{i+1}, \\
      & s_i^2 = 1, \quad
      s_i s_{i+1} s_i = s_{i+1} s_i s_{i+1}, \quad
      s_i s_j = s_j s_i \text{ when $|i-j| > 1$}.
   \end{align*}
\end{theorem}

\begin{proof}
Noting (\ref{eq:relation_1_in_T}) and (\ref{eq:relation_2_in_T}),
we can consider a natural homomorphism from $\mathcal{A}_n$ onto $\bar{T}(V)$.
Moreover, any element in $\mathcal{A}_n$ can be expressed in
the form as in Proposition~\ref{prop:canonical_form_of_elements_of_T}
by using the first and second relations.
Thus this homomorphism is injective.
\end{proof}

%%%%%%%%%%%%%%%%%%%%%%%%%%%%%%%%%%%%%%%%%%%%%%%%%%%%%%%%%%%%%%%%%%%%%%
%
\section{Multiplications and derivations}
\label{sec:multiplications_and_derivations}
%
%%%%%%%%%%%%%%%%%%%%%%%%%%%%%%%%%%%%%%%%%%%%%%%%%%%%%%%%%%%%%%%%%%%%%%
%
Some interesting operators naturally act on $\bar{T}(V)$.
In this section, we introduce the ``multiplication'' by $v \in V$
and the ``derivation'' by $v^* \in V^*$
($V^*$ is the linear dual of $V$),
and discuss their fundamental properties and quick applications.

%%%%%%%%%%%%%%%%%%%%%%%%%%%%%%%%%%%%%%%%%%%%%%%%%%%%%%%%%%%%%%%%%%%%%%%%%%%%%%%%%
\subsection{}
%%%%%%%%%%%%%%%%%%%%%%%%%%%%%%%%%%%%%%%%%%%%%%%%%%%%%%%%%%%%%%%%%%%%%%%%%%%%%%%%%
%
We denote by $L(\varphi)$ the left multiplication by $\varphi \in \bar{T}(V)$,
i.e., we put $L(\varphi)\psi = \varphi \psi$ for $\varphi$ and $\psi \in \bar{T}(V)$.
The cases
$\varphi = \sigma \in S_{\infty} \subset \mathbb{C}S_{\infty} = \bar{T}_0(V)$
and $\varphi = v \in V \subset \bar{T}_1(V)$ are particularly fundamental,
because the other cases are generated by these two cases.

In addition to these operators $L(\sigma)$ and $L(v)$,
let us introduce the ``derivation'' $L(v^*)$ by a covector $v^* \in V^*$.
Namely, for $v^* \in V^*$, we define the operator $L(v^*)$ on $\bar{T}(V)$ by
$$
   L(v^*) v_p \cdots v_1 t
   = \sum_{k=1}^p \langle v^*, v_k \rangle 
   v_p \cdots \hat{v}_k \cdots v_1 \cdot (p \,\, p-1 \,\cdots\, k+1 \,\, k) t.
$$
Here $v_p,\ldots,v_1$ are elements of $V$,
and $t$ is an element of $\mathbb{C}S_{\infty}$.
Moreover $\hat{v}_k$ means that we omit $v_k$.
When $p=0$, we put $L(v^*) t = 0$.

This definition seems natural, 
because each term in the right hand side is obtained by moving $v_k$ to the left end 
and taking the coupling with $v^*$. 
The permutation $(p \,\, p-1 \,\cdots\, k+1 \,\, k)$ appears following this movement of $v_k$.

By a direct calculation, we see that this operation is well defined.
Note that this well-definedness is equivalent with
the fact that $L(v^*)$ is commutative with the right multiplication by $\mathbb{C}S_{\infty}$.

The following relations are immediate from the definition of $L$:

\begin{proposition}\label{prop:L_and_R_commute}\sl
   The operators $L(\sigma)$, $L(v)$, and $L(v^*)$ 
   with $\sigma \in S_{\infty}$, $v \in V$, $v^* \in V^*$
   commute with the right multiplication by $\mathbb{C}S_{\infty}$.
\end{proposition}

\begin{proposition}\label{prop:initial_sp_and_final_sp}\sl
   For $v \in V$ and $v^* \in V^*$,
   the operators $L(v)$ and $L(v^*)$ map $T^{(q)}_p(V)$ as 
   \begin{align*}
      L(v) \colon & T^{(q)}_p(V) \to T^{(q-1)}_{p+1}(V) 
	  \quad \text{for $q \geq 1$}, \\
	  & T^{(0)}_p(V) \to T^{(0)}_{p+1}(V), \\
	  L(v^*) \colon & T^{(q)}_p(V) \to T^{(q+1)}_{p-1}(V) 
	  \quad \text{for $p \geq 1$}, \\
	  & T^{(q)}_0(V) \to \{ 0 \}.
   \end{align*}
\end{proposition}

Moreover, we have the following commutation relations. 
The proof is straightforward.

\begin{theorem}\label{thm:analogue_of_CCR}\sl
   For $v$, $w \in V$ and $v^*$, $w^* \in V^*$, we have 
   \begin{align*}
      L(w) L(v) & = L(v) L(w) L(s_1), \\
      L(w^*) L(v^*) & = L(s_1) L(v^*) L(w^*), \\
      L(w^*) L(v) & = L(v) L(s_1) L(w^*) + \langle w^*, v \rangle
   \end{align*}
   and moreover
   \begin{gather*}
      L(s_i) L(v) = L(v) L(s_{i+1}), \qquad
      L(v^*) L(s_i) = L(s_{i+1}) L(v^*), \\
      L(s_i)^2 = 1, \qquad
      L(s_i) L(s_{i+1}) L(s_i) = L(s_{i+1}) L(s_i) L(s_{i+1}), \\
      L(s_i) L(s_j) = L(s_j) L(s_i) \text{ when $|i-j| > 1$}.
   \end{gather*}
\end{theorem}

We can regard the first three relations 
as an analogue of the canonical commutation relations (CCR)
and the canonical anticommutation relations (CAR).
The positions of $L(s_1)$ in the right hand sides are interesting.
We will discuss this analogue further in Section~\ref{subsec:comparison_with_CCR}.

Moreover we note the commutation relation with the natural action $\pi$ of $GL(V)$
on $\bar{T}(V)$ as an algebra automorphism.
We have the following relations for $v \in V$, $\sigma \in S_{p+q}$, $v^* \in V^*$
as seen from the definition of $L$:
\begin{equation}\label{eq:comm_relation_between_pi_and_L}
   \pi(g) L(v) = L(gv) \pi(g), \quad
   \pi(g) L(\sigma) = L(\sigma) \pi(g), \quad
   \pi(g) L(v^*) = L({}^t\!g^{-1}v^*) \pi(g).
\end{equation}

The following formula is also fundamental:
$$
   L(e^*_{i_1}) \cdots L(e^*_{i_p}) e_{i_p} \cdots e_{i_1} 
   = I! s_I.
$$
Moreover, 
when $I \in \bibinom{[n]}{r}$ and $J \in \bibinom{[n]}{p}$,
we have
\begin{equation}\label{eq:pairing}
   L(e^*_{i_1}) \cdots L(e^*_{i_r}) e_{j_p} \cdots e_{j_1} 
   = 
   \begin{cases}
   I! s_I, & \quad I = J, \\
   0, & \quad \text{$r \geq p$ and $I \ne J$}.
   \end{cases}
\end{equation}
Indeed, when $r \geq p$ and $I \ne J$,
there exists $k \in \{ 1,\ldots,p \}$ such that 
the multiplicity of $e_k$ in $I$ is greater than that in $J$.

\begin{remark}
   For $\varphi \in T_k(V)$ and $\psi \in \bar{T}(V)$,
   we have
   \begin{align*}
      L(v^*) (\varphi \psi) 
      & = (L(v^*) \varphi) \psi +
      \varphi s_k s_{k-1} \cdots s_1 (L(v^*) \psi) \\
      & = (L(v^*) \varphi) \psi +
      \varphi \cdot (k+1 \,\,\, k \,\,\, k-1 \,\, \ldots \,\, 2 \,\,\, 1) 
	  \, (L(v^*) \psi).
   \end{align*}
   This can be regarded as an analogue of the Leibniz rule.
\end{remark}

%%%%%%%%%%%%%%%%%%%%%%%%%%%%%%%%%%%%%%%%%%%%%%%%%%%%%%%%%%%%%%%%%%%%%%%%%%%%%%%%%
\subsection{}
%%%%%%%%%%%%%%%%%%%%%%%%%%%%%%%%%%%%%%%%%%%%%%%%%%%%%%%%%%%%%%%%%%%%%%%%%%%%%%%%%
%
Let $\mathcal{L}(V)$ be the subalgebra 
of $\operatorname{End}_{\mathbb{C}}(\bar{T}(V))$
generated by $L(v)$, $L(v^*)$, and $L(\sigma)$
with $v \in V$, $v^* \in V^*$, and $\sigma \in S_{\infty}$.
This algebra $\mathcal{L}(V)$ can be identified with the algebra $\mathcal{B}_n$
defined by the following generators and relations:
\begin{align}\label{eq:gen_and_rel_of_L(V)}
   \text{generators: } &
   e_1,\ldots,e_n,e^*_1,\ldots,e^*_n,s_1,s_2,\ldots, \\
   \text{relations: } 
   &e_b e_a = e_a e_b s_1, \quad
   e^*_b e^*_a = s_1 e^*_a e^*_b, \quad
   e^*_b e_a = e_a s_1 e^*_b + \delta_{ab}, \notag \\
   &s_i e_a = e_a s_{i+1}, \quad
   e^*_a s_i = s_{i+1} e^*_a, \notag \\
   &s_i^2 = 1, \quad
   s_i s_{i+1} s_i = s_{i+1} s_i s_{i+1}, \quad
   s_i s_j = s_j s_i \text{ when $|i-j|>1$}. \notag 
\end{align}
Indeed, regarding  $e_1,\ldots,e_n$ and $e^*_1,\ldots,e^*_n$ 
as a basis of $V$ and its dual basis of $V^*$,
we have the following theorem:

\begin{theorem}\label{thm:def_of_L(V)_in_terms_of_gen_and_rel} \sl
   The following correspondence induces an isomorphism from $\mathcal{B}_n$ onto $\mathcal{L}(V)$:
   $$
      s_i \mapsto L(s_i), \qquad
	  e_a \mapsto L(e_a), \qquad 
      e^*_a \mapsto L(e^*_a).
   $$
\end{theorem}

This correspondence establishes a homomorphism $f$ from $\mathcal{B}_n$ onto $\mathcal{L}(V)$
as seen from Theorem~\ref{thm:analogue_of_CCR}.
Thus, to prove this theorem, it suffices to show that this $f$ is injective.
As a preparation for this, we note the following lemma:

\begin{lemma}\sl
   The algebra $\bar{T}(V)$ is isomorphic to the subalgebra $\mathcal{B}^{\dagger}_n$ of $\mathcal{B}_n$
   generated by $e_1,\ldots,e_n$ and $s_1,s_2,\ldots$
   through the natural correspondence $e_a \mapsto e_a$, $s_i \mapsto s_i$.
\end{lemma}

Indeed,
Theorem~\ref{thm:def_by_gen_and_relation} tells that
this correspondence determines a homomorphism from $\bar{T}(V)$ onto $\mathcal{B}^{\dagger}_n$.
This is also injective,
because the inverse map is given by $\mathcal{B}^{\dagger}_n \to \bar{T}(V)$,
$\psi \mapsto f(\psi)1$.

\begin{proof}[Proof of Theorem~{\sl \ref{thm:def_of_L(V)_in_terms_of_gen_and_rel}}]
It suffices to show that $f \colon \mathcal{B}_n \to \mathcal{L}(V)$ is injective.
Let us assume that there exists a nonzero $\Phi \in \operatorname{Ker} f$.
By using the first five relations in~(\ref{eq:gen_and_rel_of_L(V)}),
this can be expressed in the following form 
with $\psi_I \in \mathcal{B}^{\dagger}_n \simeq \bar{T}(V)$:
$$
   \Phi
   = \sum_{r=0}^m
   \sum_{I \in \bibinom{[n]}{r}}
   \psi_I e^*_{i_1} \cdots e^*_{i_r}.
$$
Actually, we can take $\psi_I$ to be an element of $\mathcal{B}^{\dagger}_n s_I$, 
because we have $e^*_{i_1} \cdots e^*_{i_r} = s_I e^*_{i_1} \cdots e^*_{i_r}$
(recall (\ref{eq:symmetric_expression})).
Let $J$ be one of the shortest sequences
among $\{ I \,|\, \psi_I \ne 0 \}$
(this set is not empty, because $\Phi$ is nonzero).
Namely we fix $J \in \bibinom{[n]}{p}$
such that $\psi_J \ne 0$ 
and $\psi_I = 0$ for any $I \in \bibinom{[n]}{r}$, $r < p$.
Then, using (\ref{eq:pairing}), 
we can calculate $f(\Phi) e_{j_p} \cdots e_{j_1}$ as
$$
   f(\Phi) e_{j_p} \cdots e_{j_1}
   = \sum_{r=0}^m \sum_{I \in \bibinom{[n]}{r}} 
   L(\psi_I) L(e^*_{i_1}) \cdots L(e^*_{i_r}) 
   e_{j_p} \cdots e_{j_1}
   = J! \psi_J.
$$
This $\psi_J$ must be equal to zero.
This is a contradiction,
so that $\operatorname{Ker}f = \{ 0 \}$.
\end{proof}

We now see another realization of the algebra $\mathcal{B}_n \simeq \mathcal{L}(V)$:

\begin{corollary}\label{cor:def_of_L(V)_in_terms_of_tensors}\sl
   The vector space
   $$
      \bigoplus_{p,q \geq 0} 
      V^{\otimes p} \otimes_{\mathbb{C}S_p} \mathbb{C}S_{\infty} 
      \otimes_{\mathbb{C}S_q} V^{* \otimes q}
      = \bigoplus_{q \geq 0} 
      \bar{T}(V) \otimes_{\mathbb{C}S_q} V^{* \otimes q}
   $$
   is identified with $\mathcal{L}(V)$ as a vector space
   through the correspondence
   $$
      v_{p} \cdots v_1 \sigma v^*_1 \cdots v^*_q
      \mapsto L(v_{p}) \cdots L(v_1) L(\sigma) L(v^*_1) \cdots L(v^*_q).
   $$
   Here $v_i$, $v^*_i$, and $\sigma$ are elements of 
   $V$, $V^*$, and $S_{\infty}$, respectively.
\end{corollary}

Indeed, Theorem~\ref{thm:analogue_of_CCR} teaches
that this map is well-defined and surjective.
Moreover we can check that this is injective
in the same way as the proof of Theorem~\ref{thm:def_of_L(V)_in_terms_of_gen_and_rel}.

%%%%%%%%%%%%%%%%%%%%%%%%%%%%%%%%%%%%%%%%%%%%%%%%%%%%%%%%%%%%%%%%%%%%%%%%%%%%%%%%%
\subsection{}\label{subsec:comparison_with_CCR}
%%%%%%%%%%%%%%%%%%%%%%%%%%%%%%%%%%%%%%%%%%%%%%%%%%%%%%%%%%%%%%%%%%%%%%%%%%%%%%%%%
%
The first three relations of Theorem~\ref{thm:analogue_of_CCR} 
are similar to the canonical commutation relations (CCR)
and the canonical anticommutation relations (CAR).
Thus we can regard $\bar{T}(V)$ as an analogue of the Boson and Fermion Fock spaces.
It is also natural to call $L(v)$ and $L(v^*)$ the ``creation operator'' by~$v$
and the ``annihilation operator'' by~$v^*$, respectively.

More precisely, $\bar{T}(V)$ contains the Boson and Fermion Fock spaces as quotient algebras:
$$
   \bar{T}(V) / (\sigma - 1) \simeq S(V), \qquad
   \bar{T}(V) / (\sigma - \operatorname{sgn}(\sigma)) \simeq \Lambda(V).
$$
Similarly, $\mathcal{L}(V)$ contains 
the Weyl algebra and the Clifford algebra as quotient algebras:
$$
   \mathcal{L}(V) / (\sigma - 1) 
   \simeq \operatorname{Weyl}(V^* \oplus V), \qquad
   \mathcal{L}(V) / (\sigma - \operatorname{sgn}(\sigma)) 
   \simeq \operatorname{Clifford}(V^* \oplus V).
$$
Here $(\sigma - 1)$ and $(\sigma - \operatorname{sgn}(\sigma))$ 
are the two-sided ideals of $\bar{T}(V)$ or $\mathcal{L}(V)$ generated 
by $\{ \sigma - 1 \,|\, \sigma \in S_{\infty} \}$ and 
$\{ \sigma - \operatorname{sgn}(\sigma) \,|\, \sigma \in S_{\infty} \}$, respectively.
The author wonders if the algebras $\bar{T}(V)$ and $\mathcal{L}(V)$ 
might be useful to deepen supersymmetry theory.

%%%%%%%%%%%%%%%%%%%%%%%%%%%%%%%%%%%%%%%%%%%%%%%%%%%%%%%%%%%%%%%%%%%%%%%%%%%%%%%%%
\subsection{}\label{subsec:polarizations}
%%%%%%%%%%%%%%%%%%%%%%%%%%%%%%%%%%%%%%%%%%%%%%%%%%%%%%%%%%%%%%%%%%%%%%%%%%%%%%%%%
%
These operators $L(v)$ and $L(v^*)$ are useful to express polarization operators on $T(V)$.

Let $V$ be a vector space.
The canonical action of $GL(V)$ on $V$ is naturally
extended to the actions on $T(V)$ and $\bar{T}(V)$.
The infinitesimal action of $\mathfrak{gl}(V)$ on $\bar{T}(V)$ is expressed as
\begin{equation}\label{eq:polarization}
   \pi(E_{ij}) = L(e_i) L(e^*_j).
\end{equation}
Here $E_{ij}$ ($1 \leq i,j \leq n$) means the canonical basis of $\mathfrak{gl}(V)$,
and $e_1,\ldots,e_n$ and $e^*_1,\ldots,e^*_n$ mean a basis of $V$
and its dual basis of $V^*$, respectively.
This relation (\ref{eq:polarization}) is easily seen by a direct calculation:
\begin{equation}\label{eq:action_of_L(e)L(e*)}
   L(e_i) L(e^*_j) v_k \cdots v_1 
   = \sum_{a=1}^k \langle e_j, v_a \rangle 
   v_k \cdots v_{a+1} e_i v_{a-1} \cdots v_1. 
\end{equation}
It is natural to regard this as the counterpart of 
the polarization operator $x_i \frac{\partial}{\partial x_j}$ on the polynomial space
$\mathbb{C}[x_1,\ldots,x_n]$ \cite{We}.
As seen from Proposition~\ref{prop:initial_sp_and_final_sp}, 
we can interpret both sides of the relation (\ref{eq:polarization}) 
as linear transformations on~$T(V)$.

%%%%%%%%%%%%%%%%%%%%%%%%%%%%%%%%%%%%%%%%%%%%%%%%%%%%%%%%%%%%%%%%%%%%%%%%%%%%%%%%%
\subsection{}
%%%%%%%%%%%%%%%%%%%%%%%%%%%%%%%%%%%%%%%%%%%%%%%%%%%%%%%%%%%%%%%%%%%%%%%%%%%%%%%%%
%
We can also consider an analogue of the Euler operator:
\begin{equation}\label{eq:Euler_op}
   A = \sum_{i=1}^n L(e_i) L(e^*_i).
\end{equation}
This satisfies $A \varphi = p \varphi$ for $\varphi \in \bar{T}_p(V)$
as seen from (\ref{eq:action_of_L(e)L(e*)}).
We will use this operator (and its variants) to calculate some commutants in Section~\ref{sec:commutants}.

The following variant is also interesting:
$$
   D 
   = \sum_{i,j=1}^n L(e_i) L(e_j) L(e^*_i) L(e^*_j)
   = \sum_{i,j=1}^n L(e_j) L(e_i) L(s_1) L(e^*_i) L(e^*_j).
$$
This operator acts as the operation
``exchange arbitrary two vectors and sum up all results.''
Namely we have
$$
   D e_{k_p} \cdots e_{k_1} 
   = \sum_{\begin{subarray}{c} 1 \leq a,b \leq p \\ a \ne b\end{subarray}}
   e_{k_p} \cdots 
   \underbrace{e_{k_a}}_{\text{$b$th}} \cdots 
   \underbrace{e_{k_b}}_{\text{$a$th}} \cdots e_{k_1}
$$
(here we count vectors from right).
Thus this operator can be rewritten as
$$
   D = \frac{1}{(p-2)!} R(\sum_{\sigma \in S_p} \sigma^{-1} s_1 \sigma).
$$
Here $R(t)$ is the right multiplication by $t \in \mathbb{C}S_{\infty}$.
More generally, for $\tau \in S_r$,
we have 
$$
   \sum_{I \in [n]^r} 
   L(e_{i_r}) \cdots L(e_{i_1}) 
   L(\tau) 
   L(e^*_{i_1}) \cdots L(e^*_{i_r})
   = \frac{1}{(p-r)!} R(\sum_{\sigma \in S_p} \sigma^{-1} \tau \sigma).
$$

%%%%%%%%%%%%%%%%%%%%%%%%%%%%%%%%%%%%%%%%%%%%%%%%%%%%%%%%%%%%%%%%%%%%%%%%%%%%%%%%%%
%
\section{Descriptions of commutants}
\label{sec:commutants}
%
%%%%%%%%%%%%%%%%%%%%%%%%%%%%%%%%%%%%%%%%%%%%%%%%%%%%%%%%%%%%%%%%%%%%%%%%%%%%%%%%%
%
Using the multiplications and derivations defined in the previous section, 
we can describe some commutants of fundamental operators on tensor spaces.

%%%%%%%%%%%%%%%%%%%%%%%%%%%%%%%%%%%%%%%%%%%%%%%%%%%%%%%%%%%%%%%%%%%%%%%%%%%%%%%%%%
\subsection{}
%%%%%%%%%%%%%%%%%%%%%%%%%%%%%%%%%%%%%%%%%%%%%%%%%%%%%%%%%%%%%%%%%%%%%%%%%%%%%%%%%
%
First we can prove the Schur--Weyl duality and its generalization
by a simple calculation in $\mathcal{L}(V)$.

\begin{theorem}\label{thm:generalization_of_the_Schur_Weyl_duality}\sl
   For $l = 0,1,2,\ldots$,
   we denote by $\mathcal{L}^{(q)}_l = \mathcal{L}^{(q)}_l(V)$
   the set of all linear combinations 
   of the following operators on $T^{(q)}_p(V)$:
   $$
	  L(v_l) \cdots L(v_1) 
	  L(\sigma)
	  L(v^*_1) \cdots L(v^*_l).
   $$
   Here $v_i$, $v^*_j$, and $\sigma$ are elements of $V$, $V^*$, and $S_{l+q}$, respectively.
   Then we have the following assertions:
   \begin{enumerate}
      \item[(i)]
      The inclusions
      $\mathcal{L}^{(q)}_0 \subset
      \mathcal{L}^{(q)}_1 \subset \cdots \subset
      \mathcal{L}^{(q)}_p$
      hold as operators on $T^{(q)}_p(V)$.
      Moreover we have $\mathcal{L}^{(q)}_l = \{ 0 \}$ for $l > p$.
      \item[(ii)]
      Let $R$ be the right multiplication by $\mathbb{C}S_{p+q}$ on $T^{(q)}_p(V)$.
      Then $R(\mathbb{C}S_{p+q})$ and $\mathcal{L}^{(q)}_p$
      are mutual commutants of each other in $\operatorname{End}_{\mathbb{C}}(T^{(q)}_p(V))$.
   \end{enumerate}
\end{theorem}

\begin{proof}
To prove (i), we use the Euler type operator $A = \sum_{i=1}^n L(e_i) L(e^*_i)$
defined in~(\ref{eq:Euler_op}).
As a consequence of Theorem~\ref{thm:analogue_of_CCR},
we have the relation 
\begin{equation}\label{eq:comm_rel_for_Euler}
   L(v) A = (A-1) L(v)
\end{equation}
for $v \in V$.
Using this relation, we have
\begin{align*}
   & \sum_{k=1}^n
   L(v_l) \cdots L(v_1) 
   L(e_k) L(\alpha(\sigma)) L(e^*_k)
   L(v^*_1) \cdots L(v^*_l) \\
   & \qquad = 
   L(v_l) \cdots L(v_1) 
   A L(\sigma)
   L(v^*_1) \cdots L(v^*_l) \\
   & \qquad = 
   (A - l) 
   L(v_l) \cdots L(v_1) 
   L(\sigma)
   L(v^*_1) \cdots L(v^*_l).
\end{align*}
As operators on $T^{(q)}_p(V)$, this is equal to
$$
   (p-l)L(v_l) \cdots L(v_1) 
   L(\sigma)
   L(v^*_1) \cdots L(v^*_l),
$$
because $A|_{T^{(q)}_p(V)} = p \operatorname{id}_{T^{(q)}_p(V)}$.
This calculation means the inclusion $\mathcal{L}^{(q)}_l \subset \mathcal{L}^{(q)}_{l+1}$
for $l < p$.
The relation $\mathcal{L}^{(q)}_l = \{ 0 \}$ for $l > p$ is obvious 
from Proposition~\ref{prop:initial_sp_and_final_sp}.
Thus (i) is proved.

To prove (ii),
it suffices to show $R(\mathbb{C}S_{p+q})' = \mathcal{L}^{(q)}_p$
by the double commutant theorem (Theorem~3.3.7 in~\cite{GW}).
Here we denote by $\mathcal{D}'$
the commutant of $\mathcal{D} \subset \operatorname{End}_{\mathbb{C}}(T^{(q)}_p(V))$.
The inclusion $R(\mathbb{C}S_{p+q})' \supset \mathcal{L}^{(q)}_p$ 
is immediate from Proposition~\ref{prop:L_and_R_commute},
so that we only have to show the reverse inclusion
$R(\mathbb{C}S_{p+q})' \subset \mathcal{L}^{(q)}_p$.
To show this,
we consider a higher analogue of the Euler type operator $A$:
\begin{equation}\label{eq:higher_Euler_op}
   A_p
   = \frac{1}{p!} \sum_{I \in [n]^p} 
   L(e_{i_p}) \cdots L(e_{i_1}) L(e^*_{i_1}) \cdots L(e^*_{i_p}).
\end{equation}
Using (\ref{eq:comm_rel_for_Euler}) repeatedly,  
we can express this $A_p$ as
$$
   A_p = \frac{1}{p!} (A-p+1) (A-p+2) \cdots (A-1) A,
$$
so that $A_p \varphi = \varphi$ for $\varphi \in T^{(q)}_p(V)$.
Thus, for any $f \in R(\mathbb{C}S_{p+q})'$, we have
\begin{align*}
   f(\varphi)
   &= f(A_p \varphi) \displaybreak[0]\\
   &= f(\frac{1}{p!} \sum_{I \in [n]^p} 
   L(e_{i_p}) \cdots L(e_{i_1}) L(e^*_{i_1}) \cdots L(e^*_{i_p}) \varphi) \displaybreak[0]\\
   &= f(\frac{1}{p!} \sum_{I \in [n]^p} 
   e_{i_p} \cdots e_{i_1} t_I) \displaybreak[0]\\
   &= \frac{1}{p!} \sum_{I \in [n]^p} 
   f(e_{i_p} \cdots e_{i_1} t_I),
\end{align*}
where we denote $L(e^*_{i_1}) \cdots L(e^*_{i_p}) \varphi \in \mathbb{C} S_{p+q}$ by $t_I$.
Since $f \in R(\mathbb{C}S_{p+q})'$, we have
$$
   f(\varphi)
   = \frac{1}{p!} \sum_{I \in [n]^p} 
   f(e_{i_p} \cdots e_{i_1}) t_I 
   = \frac{1}{p!} \sum_{I \in [n]^p} 
   L(f(e_{i_p} \cdots e_{i_1})) L(e^*_{i_1}) \cdots L(e^*_{i_p}) \varphi,
$$
namely
$$
   f = \frac{1}{p!} \sum_{I \in [n]^p} 
   L(f(e_{i_p} \cdots e_{i_1})) L(e^*_{i_1}) \cdots L(e^*_{i_p}).
$$
This means that $f$ is an element of $\mathcal{L}^{(q)}_p$.
\end{proof}

When $q=0$, Theorem~\ref{thm:generalization_of_the_Schur_Weyl_duality}
is equal to the Schur--Weyl duality,
because we have the following relation:

\begin{proposition}\sl
   We have
   $\mathcal{L}^{(0)}_p = \pi(U(\mathfrak{gl}(V)))$.
\end{proposition}

\begin{proof}
As a consequence of Theorem~\ref{thm:analogue_of_CCR},
we have
$$
   [L(e_i)L(e^*_j), L(e_a)] = \delta_{ai} L(e_i).
$$
Using this commutation relation, we have
\begin{align*}
   & L(e_{i_l}) \cdots L(e_{i_2}) L(e_{i_1}) 
   L(e^*_{j_1}) L(e^*_{j_2}) \cdots L(e^*_{j_l}) \\
   & \qquad 
   = L(e_{i_1}) L(e^*_{j_1})
     \cdot
     L(e_{i_l}) \cdots L(e_{i_2}) L(e^*_{j_2}) \cdots L(e^*_{j_l}) 
     + \text{an element of $\mathcal{L}^{(0)}_{l-1}$}.
\end{align*}
Repeating this, we see that any element of $\mathcal{L}^{(0)}_l$ can be
expressed as a polynomial in $\pi(E_{ab}) = L(e_a)L(e^*_b)$, $1 \leq a,b \leq n$.
Thus we see the inclusion $\mathcal{L}^{(0)}_p \subset \pi(U(\mathfrak{gl}(V)))$.

Conversely, using the same commutation relation,
we can express any element in $\pi(U(\mathfrak{gl}(V)))$
as a sum of elements in $\mathcal{L}^{(0)}_0, \mathcal{L}^{(0)}_1,\ldots$.
Furthermore, this can be regarded as an element of $\mathcal{L}^{(0)}_p$ 
by Theorem~\ref{thm:generalization_of_the_Schur_Weyl_duality} (i).
\end{proof}

\begin{remarks}
   (1)
   When $p=0$, Theorem~\ref{thm:generalization_of_the_Schur_Weyl_duality}
   is equal to the fact that
   $L(\mathbb{C}S_q)$ and $R(\mathbb{C}S_q)$ are mutual commutants of each other
   in $\operatorname{End}_{\mathbb{C}}(\mathbb{C}S_q)$.
   
\medskip

   \noindent
   (2)
   The results in this subsection are generalized to their $q$-analogues.
   Namely we can construct $q$-analogues of $\bar{T}(V)$ and $\mathcal{L}(V)$,
   which produce a simple proof of the $q$-Schur--Weyl duality
   between the quantum enveloping algebra $U_q(\mathfrak{gl}_n)$ and 
   the Iwahori--Hecke algebra of type $A$.
   This result will be written somewhere else.

\medskip

   \noindent
   (3)
   For any group $G$, every map $f \colon G \to G$ commuting 
   with all right translations is equal to a left translation.
   This fact is proved quickly as follows.
   Let $x$ be an element of $G$.
   Then we have $f(x) = f(ex)$ with the identity element $e$.
   Since $f$ commutes with the right multiplication by $x$,
   this $f(ex)$ is equal to $f(e)x$.
   This means that $f$ is equal to the left multiplication by $f(e)$.
   It should be noted that
   the proof of Theorem~\ref{thm:generalization_of_the_Schur_Weyl_duality}
   is based on the same principle.
\end{remarks}

%%%%%%%%%%%%%%%%%%%%%%%%%%%%%%%%%%%%%%%%%%%%%%%%%%%%%%%%%%%%%%%%%%%%%%%%%%%%%%%%%%
\subsection{}
%%%%%%%%%%%%%%%%%%%%%%%%%%%%%%%%%%%%%%%%%%%%%%%%%%%%%%%%%%%%%%%%%%%%%%%%%%%%%%%%%
%
We have a similar relation about operators on $\bar{T}(V)$:

\begin{theorem}\label{thm:duality_between_M(V)_and_S} \sl
   Let $R$ be the right multiplication by $\mathbb{C}S_{\infty}$
   on $\bar{T}(V)$.
   In addition, we denote by $\mathcal{M}_p = \mathcal{M}_p(V)$ 
   the set of all operators on $\bar{T}(V)$ in the form
   $$
      \sum_{I \in [n]^p} L(\psi_I) L(e^*_{i_1}) \cdots L(e^*_{i_p}),
   $$
   where $\psi_I$'s are elements of $\bar{T}(V)$.
   Moreover, we denote by $\mathcal{M} = \mathcal{M}(V)$
   the set of all operators in the form $\sum_{p \geq 0} D_p$ with $D_p \in \mathcal{M}_p$. 
   This is an infinite sum in general,
   but this actually acts as a finite sum, 
   when it is applied to any element of $\bar{T}(V)$.
   Then these $R(\mathbb{C}S_{\infty})$ and $\mathcal{M}$ 
   are mutual commutants of each other as operators on $\bar{T}(V)$.
\end{theorem}

The inclusions $R(\mathbb{C}S_{\infty})' \supset \mathcal{M}$ and 
$R(\mathbb{C}S_{\infty}) \subset \mathcal{M}'$ are obvious,
so that it suffices to show the following two propositions:

\begin{proposition}\label{prop:inclusion_1} \sl
   We have $\mathcal{M}' \subset R(\mathbb{C}S_{\infty})$.
\end{proposition}

\begin{proposition}\label{prop:inclusion_2} \sl
   We have $R(\mathbb{C}S_{\infty})' \subset \mathcal{M}$.
\end{proposition}

To prove Proposition~\ref{prop:inclusion_2}, we use the following lemma:

\begin{lemma}\label{lem:lemma_for_duality} \sl
   For any $f \in R(\mathbb{C}S_{\infty})'$,
   there exists $D_p \in \mathcal{M}_p$
   such that $f(\varphi) = D_p \varphi$
   for $\varphi \in \bar{T}_p(V)$.
\end{lemma}

\begin{proof}[Proof of Proposition~{\sl\ref{prop:inclusion_1}}]
Assume that $f \in \mathcal{M}'$.
Then we have
$$
   L(v^*) f(1) = f(L(v^*) 1) = f(0) = 0,
$$
so that $f(1)$ belongs to $\bar{T}_0(V) = \mathbb{C}S_{\infty}$.
Moreover, we have $f(\varphi) = f(\varphi \cdot 1) = \varphi f(1)$ for $\varphi \in \bar{T}(V)$.
From these we see $f \in R(\mathbb{C}S_{\infty})$.
\end{proof}

\begin{proof}[Proof of Lemma~{\sl\ref{lem:lemma_for_duality}}]
In a way similar to the proof of Theorem~\ref{thm:generalization_of_the_Schur_Weyl_duality} (ii),
we have
$$
   f(\varphi)
   = f(A_p \varphi)
   = \frac{1}{p!} \sum_{I \in [n]^p} 
   L(f(e_{i_p} \cdots e_{i_1})) L(e^*_{i_1}) \cdots L(e^*_{i_p}) \varphi
$$
for $\varphi \in \bar{T}_p(V)$.
Here $A_p$ is the analogue of the Euler operator defined in (\ref{eq:higher_Euler_op}).
This means the assertion.
\end{proof}

\begin{proof}[Proof of Proposition~{\sl\ref{prop:inclusion_2}}]
Fix an arbitrary $f \in R(\mathbb{C}S_{\infty})'$. 
We take $D_0 \in \mathcal{M}_0$
such that $f(\varphi_0) = D_0 \varphi_0$ for all $\varphi_0 \in \bar{T}_0(V)$.
Next, we take $D_1 \in \mathcal{M}_1$
such that $(f - D_0)(\varphi_1) = D_1 \varphi_1$ for all $\varphi_1 \in \bar{T}_1(V)$.
In this way, we take $D_k \in \mathcal{M}_k$ for $k=0,1,2,\ldots$ such that 
$$
   (f - D_0 - \cdots - D_{r-1}) (\varphi_r) = D_r \varphi_r
$$
for all $\varphi_r \in \bar{T}_r(V)$.
Here we used Lemma~\ref{lem:lemma_for_duality} and the fact 
$f - D_0 - \cdots - D_{r-1} \in R(\mathbb{C}S_{\infty})'$.
From this, we can deduce the following relation for $\varphi \in \bigoplus_{k=0}^r \bar{T}_k(V)$:
$$
   f(\varphi) = (D_0 + \cdots + D_r) \varphi.
$$
This is proved by induction on $r$ 
by noting the relation $D_k \varphi_l = 0$ for $\varphi_l \in \bar{T}_l(V)$ when $k > l$.
This means that $f= \sum_{k \geq 0} D_k \in \mathcal{M}$.
\end{proof}

%%%%%%%%%%%%%%%%%%%%%%%%%%%%%%%%%%%%%%%%%%%%%%%%%%%%%%%%%%%%%%%%%%%%%%%%%%%%%%%%%%
\subsection{}
%%%%%%%%%%%%%%%%%%%%%%%%%%%%%%%%%%%%%%%%%%%%%%%%%%%%%%%%%%%%%%%%%%%%%%%%%%%%%%%%%
%
We can also consider an analogue of the Howe duality.
The general linear group $GL_n(\mathbb{C})$ naturally acts on 
the tensor product $W = \mathbb{C}^n \otimes \mathbb{C}^{n'}$
(canonically on the first component $\mathbb{C}^n$
and trivially on the second component $\mathbb{C}^{n'}$).
This can be extended to the action 
$\nu$ on $T^{(q)}_p(W)$ naturally.
Then we have the following theorem
as an analogue of the $(GL_n(\mathbb{C}), GL_{n'}(\mathbb{C}))$ duality 
on the polynomial space $\mathcal{P}(W)$ due to Howe \cite{Ho}.

\begin{theorem}\label{thm:analogue_of_Howe_duality} \sl
   Let $\mathcal{Q}_1$ be
   the operator algebra in 
   $\operatorname{End}_{\mathbb{C}}(T^{(q)}_p(\mathbb{C}^n \otimes \mathbb{C}^{n'}))$ 
   generated by $\nu(GL_n(\mathbb{C}))$ 
   and the right multiplication by $S_{p+q}$.
   Moreover, we denote by $\mathcal{Q}_2$ the set of 
   all linear combinations of the operators in the form
   $$
      \Phi_{a_1,\ldots,a_p;b_1,\ldots,b_p}(\sigma) = 
      \sum_{I \in [n]^p}
	  L(w_{i_p a_p}) \cdots L(w_{i_1 a_1}) L(\sigma)
	  L(w^*_{i_1 b_1}) \cdots L(w^*_{i_p b_p}).
   $$
   Here $w_{ij}$ and $w^*_{ij}$ mean the canonical basis of 
   $W = \mathbb{C}^n \otimes \mathbb{C}^{n'}$ and the dual basis, respectively.
   Moreover, $\sigma$ is an element of $S_{p+q}$.
   Then $\mathcal{Q}_1$ and $\mathcal{Q}_2$ are mutual commutants of each other.
\end{theorem}

\begin{proof}
It is sufficient to show $\mathcal{Q}_1' = \mathcal{Q}_2$,
because $\mathcal{Q}_1$ is semisimple.

The proof of the inclusion $\mathcal{Q}_1' \supset \mathcal{Q}_2$ is plane. 
First, $\mathcal{Q}_2$ commutes with the right multiplication by $\mathbb{C}S_{p+q}$,
because $\mathcal{Q}_2 \subset \mathcal{L}^{(q)}_p(W)$.
Secondly, $\mathcal{Q}_2$ also commutes with $\nu(GL_n(\mathbb{C}))$
as seen the following calculation for $g \in GL_n(\mathbb{C})$:
\begin{align*}
   & \nu(g) 
   \Phi_{a_1,\ldots,a_p;b_1,\ldots,b_p}(\sigma) 
   \nu(g^{-1})\\
   &\qquad =
   \sum_{I\in [n]^p}
   L(\tilde{g} w_{i_p a_p}) \cdots L(\tilde{g} w_{i_1 a_1}) L(\sigma)
   L({}^t\tilde{g}^{-1} w^*_{i_1 b_1}) \cdots L({}^t\tilde{g}^{-1} w^*_{i_p b_p}) \\
   & \qquad = \Phi_{a_1,\ldots,a_p;b_1,\ldots,b_p}(\sigma).
\end{align*}
The first equality is a consequence of (\ref{eq:comm_relation_between_pi_and_L}).
Here we regard $g \in GL_n(\mathbb{C})$ as an element of $GL(W)$
through the action of $GL_n(\mathbb{C})$ on $W$,
and denote this image by $\tilde{g}$.

The converse inclusion $\mathcal{Q}_1' \subset \mathcal{Q}_2$ is shown as follows.
The first fundamental theorem of the invariant theory for tensor spaces (Theorem~4.3.1 in~\cite{GW})
states that
all $GL_n(\mathbb{C})$-invariants in $W^{\otimes p} \otimes W^{* \otimes p}$
are expressed as a linear combination of elements in the form
$$
   \sum_{I \in [n]^p}
   w_{i_p a_p} \otimes \cdots \otimes w_{i_1 a_1} 
   \otimes
   w^*_{i_{\sigma(1)} b_{\sigma(1)}} \otimes \cdots 
   \otimes w^*_{i_{\sigma(p)} b_{\sigma(p)}}.
$$
Thus the following elements with $\tau \in S_{p+q}$ and $\sigma \in S_p$ 
expand all $GL_n(\mathbb{C})$-invariants in 
$W^{\otimes p} \otimes \mathbb{C}S_{p+q} \otimes W^{* \otimes p}$:
$$
   \Psi_{a_1,\ldots,a_p;b_1,\ldots,b_p}(\tau,\sigma) =
   \sum_{I \in [n]^p}
   w_{i_p a_p} \otimes \cdots \otimes w_{i_1 a_1} 
   \otimes \tau \otimes
   w^*_{i_{\sigma(1)} b_{\sigma(1)}} \otimes \cdots 
   \otimes w^*_{i_{\sigma(p)} b_{\sigma(p)}}.
$$
Here we consider the trivial action of $GL_n(\mathbb{C})$ on $\mathbb{C}S_{p+q}$.
Next we consider the linear map 
$F \colon W^{\otimes p} \otimes \mathbb{C}S_{p+q} \otimes W^{* \otimes p} \to \mathcal{L}^{(q)}_p(W)$
defined by the following correspondence:
$$
   w_p \otimes \cdots \otimes w_1 \otimes
   \tau
   \otimes w^*_1 \otimes \cdots \otimes w^*_p
   \mapsto
   L(w_p) \cdots L(w_1) L(\tau) L(w^*_1) \cdots L(w^*_p).
$$
This map is surjective and $GL(W)$-equivariant (hence also $GL_n(\mathbb{C})$-equivariant)
as seen from (\ref{eq:comm_relation_between_pi_and_L}).
Here $GL(W)$ acts on $ \mathcal{L}^{(q)}_p(W)$ by conjugation.
Thus all $GL_n(\mathbb{C})$-invariants in $\mathcal{L}^{(q)}_p(W)$
comes from $GL_n(\mathbb{C})$-invariants 
in $W^{\otimes p} \otimes \mathbb{C}S_{p+q} \otimes W^{* \otimes p}$.
This inplies $\mathcal{Q}_1' \subset \mathcal{Q}_2$,
because the image of $\Psi_{a_1,\ldots,a_p;b_1,\ldots,b_p}(\tau,\sigma)$ under $F$ is equal to
$\Phi_{a_1,\ldots,a_p;b_1,\ldots,b_p}(\tau\sigma)$.
\end{proof}

The following relation is immediate from Theorem~\ref{thm:analogue_of_Howe_duality}:

\begin{corollary}\label{cor:relation_between_ZU_and_Q2}\sl
   For any central element $C$ in the universal enveloping algebra $U(\mathfrak{gl}_n)$,
   the action of $C$ on $T^{(q)}_p(\mathbb{C}^n \otimes \mathbb{C}^{n'})$ 
   is equal to an operator in $\mathcal{Q}_2$,
   and this operator is also central in $\mathcal{Q}_2$.
\end{corollary}

In the next section,
we will give an explicit description of this relation in terms of generators
as an analogue of the Capelli identity.

%%%%%%%%%%%%%%%%%%%%%%%%%%%%%%%%%%%%%%%%%%%%%%%%%%%%%%%%%%%%%%%%%%%%%%%%%%%%%%%%%%
%
\section{Capelli identity on $T(\mathbb{C}^n \otimes \mathbb{C}^{n'})$
and its application to invariant theory}\label{sec:Capelli_and_FFT}
%
%%%%%%%%%%%%%%%%%%%%%%%%%%%%%%%%%%%%%%%%%%%%%%%%%%%%%%%%%%%%%%%%%%%%%%%%%%%%%%%%%
%
In this section,
we give a Capelli type identity on $T(\mathbb{C}^n \otimes \mathbb{C}^{n'})$
using the multiplication and derivation operators
defined in Section~\ref{sec:multiplications_and_derivations}
(we will give a more general relation, namely a ``higher version''
in Section~\ref{subsec:higher_Capelli_identity_on_T}).
This can be regarded as a description of the correspondence
of invariant operators in Corollary~\ref{cor:relation_between_ZU_and_Q2}.

The original Capelli identity played an important role in invariant theory.
Our Capelli type identity also has an application to invariant theory.
That is,
we can determine the $SL_n(\mathbb{C})$-invariants in $T(\mathbb{C}^n \otimes \mathbb{C}^{n'})$
using this Capelli type identity. 
This application can be regarded as a noncommutative version of a classical result,
namely the description of $SL_n(\mathbb{C})$-invariants in the polynomial algebra
$\mathcal{P}(\mathbb{C}^n \otimes \mathbb{C}^{n'})$ 
using the original Capelli identity (Theorem 2.6.A in \cite{We}).

%%%%%%%%%%%%%%%%%%%%%%%%%%%%%%%%%%%%%%%%%%%%%%%%%%%%%%%%%%%%%%%%%%%%%%%%%%%%%%%%%%
\subsection{}
%%%%%%%%%%%%%%%%%%%%%%%%%%%%%%%%%%%%%%%%%%%%%%%%%%%%%%%%%%%%%%%%%%%%%%%%%%%%%%%%%
%
Let us recall some central elements of $U(\mathfrak{gl}_n)$
called the Capelli elements (\cite{C1}, \cite{C2}, \cite{HU}, \cite{U1}).
Let $E_{ij}$ be the standard basis of $\mathfrak{gl}_n$,
and consider the matrix $E = (E_{ij})_{1 \leq i,j \leq n}$ in 
$\operatorname{Mat}_n(U(\mathfrak{gl}_n))$.
We consider the following element in $U(\mathfrak{gl}_n)$:
$$
   C_n = \operatorname{column-det}(E + \operatorname{diag}(n-1,n-2,\ldots,0)).
$$
Here ``$\operatorname{column-det}$'' means the ``column-determinant.''
Namely, for a square matrix $X = (X_{ij})_{1 \leq i,j \leq n}$ 
whose entries are not necessarily commutative,
we put
$$
   \operatorname{column-det} X
   = \sum_{\sigma \in S_n} \operatorname{sgn}(\sigma) 
   X_{\sigma(1)1} X_{\sigma(2)2} \cdots X_{\sigma(n)n}.
$$
This $C_n$ is known to be central in $U(\mathfrak{gl}_n)$.
We can generalize this to sums of minors:
\begin{equation}\label{eq:definition_of_Capelli_el}
   C_r
   = \sum_{I \in \mybinom{[n]}{r}}
     \operatorname{column-det}(E_{II} + \operatorname{diag}(r-1,r-2,\ldots,0)).
\end{equation}
Here we put $X_{IJ} = (X_{i_a j_b})_{1 \leq a,b \leq r}$
for a matrix $X = (X_{ij})_{1 \leq i \leq n, \, 1 \leq j \leq n'}$
and sequences $I = (i_1,\ldots,i_r) \in [n]^r$, $J = (j_1,\ldots,j_r) \in [n']^r$.
It is known that $C_r$ is also central in $U(\mathfrak{gl}_n)$,
and $C_1,\ldots,C_n$ generate the center of $U(\mathfrak{gl}_n)$.
We call this $C_r$ the ``Capelli element'' of degree $r$.

We work in the same situation 
as in Theorem~\ref{thm:analogue_of_Howe_duality}.
Namely, we consider the natural action of $GL_n(\mathbb{C})$ 
on $\mathbb{C}^n \otimes \mathbb{C}^{n'}$.
This action can be extended to the actions on $T^{(q)}(\mathbb{C}^n \otimes \mathbb{C}^{n'})$ 
and moreover on $\bar{T}(\mathbb{C}^n \otimes \mathbb{C}^{n'})$.
The infinitesimal action of $\mathfrak{gl}_n$ on $\bar{T}(\mathbb{C}^n \otimes \mathbb{C}^{n'})$ 
is expressed as follows
(cf. Section~\ref{subsec:polarizations}):
\begin{equation}\label{eq:expression_of_action_of_gl}
   \nu(E_{ij}) = \sum_{a=1}^{n'} L(w_{ia}) L(w^*_{ja}).
\end{equation}
Here $w_{ij}$ and $w^*_{ij}$ mean the canonical basis of $\mathbb{C}^n \otimes \mathbb{C}^{n'}$
and its dual basis, respectively.
It is convenient to consider the matrices
$Z$ and $Z^*$ in $\operatorname{Mat}_{n,n'}(\mathcal{L}(\mathbb{C}^n \otimes \mathbb{C}^{n'}))$
defined by
$$
   Z = (L(w_{ij}))_{1 \leq i \leq n, \, 1 \leq j \leq n'},\qquad
   Z^* = (L(w^*_{ij}))_{1 \leq i \leq n, \, 1 \leq j \leq n'}.
$$
Then,
we can rewrite (\ref{eq:expression_of_action_of_gl}) 
simply as $\nu(E) = Z \, {}^t\!Z^*$.

In this situation, we have the following analogue of the Capelli identity
(cf. \cite{HU}, \cite{U1}):

\begin{theorem}\label{thm:Capelli_identity_in_T}\sl
   For $1 \leq r \leq p$, we have
   \begin{align*}
      \nu(C_r) 
	  &= \sum_{I \in \mybinom{[n]}{r}} 
	  \sum_{J \in \bibinom{[n']}{r}} 
	  \frac{1}{J!} 
	  \operatorname{column-det}Z_{I^{\circ}J^{\circ}}
	  \operatorname{column-det}Z^*_{IJ} \\
	  &=
      \frac{1}{r!^2}	  
	  \sum_{I \in [n]^r} 
	  \sum_{J \in [n']^r} 
	  \operatorname{column-det}Z_{I^{\circ}J^{\circ}}
	  \operatorname{column-det}Z^*_{IJ}.
   \end{align*}
   Here we put $I^{\circ} = (i_r,\ldots,i_1)$ and $J^{\circ} = (j_r,\ldots,j_1)$
   for $I = (i_1,\ldots,i_r)$ and $J = (j_1,\ldots,j_r)$.
\end{theorem}

The proof will be given soon in Section~\ref{subsec:proof_of_Capelli_type_id}.

Recall the relation between the centers of $U(\mathfrak{gl}_n)$ and $\mathcal{Q}_2$ 
in Corollary~\ref{cor:relation_between_ZU_and_Q2}.
We can regard Theorem~\ref{thm:Capelli_identity_in_T} as
an explicit description of this relation in terms of generators.
Indeed, the right hand side of Theorem~\ref{thm:Capelli_identity_in_T}
belongs to $\mathcal{Q}_2$
(recall the relation $\mathcal{L}^{(q)}_r(\mathbb{C}^n \otimes \mathbb{C}^{n'}) 
\subset \mathcal{L}^{(q)}_p(\mathbb{C}^n \otimes \mathbb{C}^{n'})$
in Theorem~\ref{thm:generalization_of_the_Schur_Weyl_duality}).
This is parallel to the fact that
the original Capelli identity describes
the correspondence of central elements of the universal enveloping algebras
and invariant differential operators 
in the dual pair $(GL_n(\mathbb{C}),GL_{n'}(\mathbb{C}))$
(cf. Section~\ref{subsec:higher_Capelli_identity}; see also \cite{Ho}, \cite{HU}, \cite{U1}).

In Section~\ref{subsec:higher_Capelli_identity_on_T},
we will give a more detailed description of this correspondence.
Namely we will describe the image of a basis of the center of $U(\mathfrak{gl}_n)$
as a vector space.

%%%%%%%%%%%%%%%%%%%%%%%%%%%%%%%%%%%%%%%%%%%%%%%%%%%%%%%%%%%%%%%%%%%%%%%%%%%%%%%%%%
\subsection{}
%%%%%%%%%%%%%%%%%%%%%%%%%%%%%%%%%%%%%%%%%%%%%%%%%%%%%%%%%%%%%%%%%%%%%%%%%%%%%%%%%
%
We recall the exterior calculus to deal with noncommutative determinants.
This is a key of the proof of Theorem~\ref{thm:Capelli_identity_in_T}
and serves as a prototype of
the approach to (noncommutative) immanants in 
Sections~\ref{sec:immanants_and_algebras}--\ref{sec:higher_Capelli}.
See \cite{U2}--\cite{U5}, \cite{IU}, \cite{I1}--\cite{I6}, \cite{Ha}, and \cite{Wa}
for the details and further developments of this approach to noncommutative determinants.

It is convenient to
introduce the ``symmetrized determinant'' 
(or the ``double determinant'')
as another noncommutative determinant.
For an $n \times n$ matrix $X$, we put
$$
   \operatorname{symm-det} X
   = \frac{1}{n!} \sum_{\sigma,\sigma' \in S_n} 
   \operatorname{sgn}(\sigma\sigma^{\prime -1})
   X_{\sigma(1)\sigma'(1)} \cdots X_{\sigma(n)\sigma'(n)}.
$$ 
We also consider a variant with $n$ complex parameters $a_1,\ldots,a_n$:
$$
   \operatorname{symm-det}(X \,;\, a_1,\ldots,a_n) 
   = \frac{1}{n!} \sum_{\sigma,\sigma' \in S_n} 
   \operatorname{sgn}(\sigma\sigma^{\prime -1})
   X_{\sigma(1)\sigma'(1)}(a_1) \cdots X_{\sigma(n)\sigma'(n)}(a_n).
$$
Here $X_{ij}(u)$ means $X_{ij}(u) = X_{ij} + \delta_{ij} u$.
Moreover, we put
\begin{align*}
   \operatorname{det}_r(X \,;\, a_1,\ldots,a_r) 
   &= \frac{1}{r!}\sum_{I \in [n]^r} 
   \operatorname{symm-det}(X_{II} \,;\, a_1,\ldots,a_r) \\
   &= \sum_{I \in \mybinom{[n]}{r}} 
   \operatorname{symm-det}(X_{II} \,;\, a_1,\ldots,a_r).
\end{align*}
The second equality holds because $\operatorname{symm-det}(X_{II} \,;\, a_1,\ldots,a_n) = 0$ 
for $I! \ne 1$ and 
$$
   \operatorname{symm-det}(X_{II} \,;\, a_1,\ldots,a_r) 
   = \operatorname{symm-det}(X_{\sigma(I)\sigma(I)} \,;\, a_1,\ldots,a_r)
$$ 
for any $\sigma \in S_r$.
We can also express this in terms of the column-determinant:
\begin{equation}\label{eq:expression_of_det_r_in_terms_of_column-det}
   \operatorname{det}_r(X \,;\, a_1,\ldots,a_r) 
   = \frac{1}{r!}\sum_{I \in [n]^r} 
   \operatorname{column-det}(X_{II} + 1_{II}\operatorname{diag}(a_1,\ldots,a_r)).
\end{equation}
This follows by a simple calculation.

These determinants have their own advantages.
For example, the column-determinant is easily calculated,
and the symmetrized determinant has some good invariances.

\begin{remark}
   In \cite{I2}--\cite{I6},
   the matrix functions
   ``$\operatorname{column-det}$,''
   ``$\operatorname{symm-det}$,'' and 
   ``$\operatorname{det}_r$''
   are denoted by
   ``$\operatorname{det}$,''
   ``$\operatorname{Det}$,'' and
   ``$\operatorname{Det}_r$,''
   respectively.
\end{remark}

These noncommutative determinants can be expressed in terms of the exterior calculus as follows.
Let $V$ be an $n$-dimensional $\mathbb{C}$-vector space and
consider the exterior algebra $\Lambda(V)$ on $V$.
Let $X = (X_{ij})$ be an $n \times n$ matrix whose entries are elements of 
an associative $\mathbb{C}$-algebra $\mathcal{A}$. 
From now on, we calculate in the extended  algebra $\Lambda(V) \otimes \mathcal{A}$
in which the two algebras $\Lambda(V)$ and $\mathcal{A}$ commute with each other.

We fix a basis $e_1,\ldots,e_n$ of $V$ and
put $\xi_j = \sum_{i=1}^n e_i X_{ij} \in \Lambda(V) \otimes \mathcal{A}$.
Then we have
$$
   \xi_{j_1} \cdots \xi_{j_r} 
   = \sum_{I \in [n]^r}
   e_{i_1} \cdots e_{i_r} 
   \operatorname{column-det} X_{IJ}
   = r! \sum_{I \in \mybinom{[n]}{r}} 
   e_{i_1} \cdots e_{i_r} 
   \operatorname{column-det} X_{IJ}.
$$
Indeed, the first equality is seen by a straightforward calculation.
The second equality is also immediate,
because the column-determinant is alternating 
under the permutations of rows.
This is generalized as follows:
\begin{align*}
   \xi_{j_1}(a_1) \cdots \xi_{j_r}(a_r)
   &= \sum_{I \in [n]^r}
   e_{i_1} \cdots e_{i_r} 
   \operatorname{column-det} (X_{IJ} + 1_{IJ}\operatorname{diag}(a_1,\ldots,a_r))\\
   &= r!
   \sum_{I \in \mybinom{[n]}{r}}
   e_{i_1} \cdots e_{i_r} 
   \operatorname{column-det} (X_{IJ} + 1_{IJ}\operatorname{diag}(a_1,\ldots,a_r)).
\end{align*}
Here 
$1$ means the unit matrix of size $n$,
and we put $\xi_j(u) = \sum_{i=1}^n e_i X_{ij}(u)$.

To deal with the symmetrized determinant, we double the anti-commuting generators.
Namely, in addition to the basis $e_1,\ldots,e_n$ of V, 
we consider the dual basis $e^*_1,\ldots,e^*_n$ of $V^*$
and work in $\Lambda(V \oplus V^*) \otimes \mathcal{A}$. 
We consider the following elements in  $\Lambda(V \oplus V^*) \otimes \mathcal{A}$:
$$
   \tau = \sum_{i=1}^n e_i e^*_i, \quad
   \Xi = \sum_{i,j=1}^n e_i X_{ij} e^*_j, \quad
   \Xi(u) = \sum_{i,j=1}^n e_i X_{ij}(u) e^*_j 
   = \sum_{j=1}^n \xi_j(u) e^*_j
   = \Xi + u \tau.
$$
Then, a straightforward calculation shows us
\begin{align*}
   & \Xi(a_1) \cdots \Xi(a_r) \\
   & \quad
   = \frac{1}{r!^2}
   \sum_{I,J \in [n]^r} e_{i_1} e^*_{j_1} \cdots e_{i_r} e^*_{j_r}
   \sum_{\sigma,\sigma' \in S_r} 
   \operatorname{sgn}(\sigma) \operatorname{sgn}(\sigma')
   X_{i_{\sigma(1)}j_{\sigma'(1)}}(a_1)
   \cdots
   X_{i_{\sigma(r)}j_{\sigma'(r)}}(a_r) \\
   & \quad
   = 
   \sum_{I,J \in \mybinom{[n]}{r}} 
   e_{i_1} e^*_{j_1} \cdots e_{i_r} e^*_{j_r}
   \sum_{\sigma,\sigma' \in S_r} 
   \operatorname{sgn}(\sigma) \operatorname{sgn}(\sigma')
   X_{i_{\sigma(1)}j_{\sigma'(1)}}(a_1)
   \cdots
   X_{i_{\sigma(r)}j_{\sigma'(r)}}(a_r).
\end{align*}

It is useful to consider the symmetric bilinear form 
$\langle \cdot,\cdot \rangle \colon 
\Lambda(V \oplus V^*) \times \Lambda(V \oplus V^*) \to \mathbb{C}$
such that the following forms an orthonormal basis of $\Lambda(V \oplus V^*)$:
$$
   \left\{ e_{i_1} \cdots e_{i_p} e^*_{j_1} \cdots e^*_{j_q} \, \left| \, 
   I \in \mybinom{[n]}{p}, \, J \in \mybinom{[n]}{q}, \, 0 \leq p,q \leq n \right. \right\}.
$$
Moreover, we extend this to the bilinear map 
$$
   \langle \cdot,\cdot \rangle \colon 
   \Lambda(V \oplus V^*) \times \Lambda(V \oplus V^*) \otimes \mathcal{A} \to \mathcal{A}, \qquad
   \langle \varphi, \psi \otimes a \rangle 
   = \langle \varphi, \psi \rangle a.
$$
Here, $\varphi$ and $\psi$ are elements of $\Lambda(V \oplus V^*)$,
and $a$ is an element of $\mathcal{A}$.
Using this bilinear map, 
we can express some noncommutative determinants as follows.
These can be shown by simple calculations.
\begin{align}
\label{eq:expression_of_column-det1}
   \langle e_{i_1} \cdots e_{i_r} , \xi_{j_1} \cdots \xi_{j_r} \rangle
   & = \operatorname{column-det}X_{IJ}, \\
\label{eq:expression_of_column-det2}
   \langle e_{i_1} \cdots e_{i_r} , \xi_{j_1}(a_1) \cdots \xi_{j_r}(a_r) \rangle
   & = \operatorname{column-det}(X_{IJ} + 1_{IJ}\operatorname{diag}(a_1,\ldots,a_r)),\\
\label{eq:expression_of_symm_det}
   \frac{1}{r!}\langle \tau^{(r)}, \Xi(a_1) \cdots \Xi(a_r) \rangle
   & = \operatorname{det}_r(X \,;\, a_1,\ldots,a_r).
\end{align}
Here  $x^{(r)}$ means the divided power $x^{(r)} = \frac{1}{r!}x^r$.
Note that 
$$
   \tau^{(r)} 
   = \frac{1}{r!}\sum_{I \in [n]^r} 
   e_{i_1} \cdots e_{i_r} e^*_{i_r} \cdots e^*_{i_1}
   = \sum_{I \in \mybinom{[n]}{r}} 
   e_{i_1} \cdots e_{i_r} e^*_{i_r} \cdots e^*_{i_1}.
$$

%%%%%%%%%%%%%%%%%%%%%%%%%%%%%%%%%%%%%%%%%%%%%%%%%%%%%%%%%%%%%%%%%%%%%%%%%%%%%%%%%%
\subsection{}\label{subsec:proof_of_Capelli_type_id}
%%%%%%%%%%%%%%%%%%%%%%%%%%%%%%%%%%%%%%%%%%%%%%%%%%%%%%%%%%%%%%%%%%%%%%%%%%%%%%%%%
%
As a preparation for the proof of Theorem~\ref{thm:Capelli_identity_in_T}, 
we note the following relation:

\begin{proposition}\label{prop:relation_between_column_dets}\sl
   For $I$, $J \in[n]^r$ and $\sigma \in S_r$,
   we have
   \begin{align*}
      & \operatorname{column-det}(E_{\sigma(I)\sigma'(J)} 
      + 1_{\sigma(I)\sigma'(J)} \operatorname{diag}(r-1,r-2,\ldots,0)) \\
	  & \qquad
	  = \operatorname{sgn}(\sigma)\operatorname{sgn}(\sigma')
	  \operatorname{column-det}(E_{IJ} 
      + 1_{IJ} \operatorname{diag}(r-1,r-2,\ldots,0)).
   \end{align*}
   In particular, this quantity vanishes when $I! \ne 1$ or $J! \ne 1$.
\end{proposition}

\begin{proof}
For $\xi_j(u) = \sum_{i=1}^n e_i E_{ij}(u) \in \Lambda(V) \otimes U(\mathfrak{gl}_n)$,
we have the commutation relation
\begin{equation}\label{eq:comm_rel_for_xi}
   \xi_{j_1}(u+1) \xi_{j_2}(u) = - \xi_{j_2}(u+1) \xi_{j_1}(u)
\end{equation}
by a direct calculation.
Thus we have
$$
   \xi_{j_{\sigma(1)}}(r-1) \xi_{j_{\sigma(2)}}(r-2) \cdots \xi_{j_{\sigma(r)}}(0)
   = \operatorname{sgn}(\sigma)
   \xi_{j_1}(r-1) \xi_{j_2}(r-2) \cdots \xi_{j_r}(0).
$$
Combining this with (\ref{eq:expression_of_column-det2}), 
we see the alternating property under permutations of columns.
In addition, the alternating property under permutations of rows holds in general.
\end{proof}

The following is immediate from this and (\ref{eq:expression_of_det_r_in_terms_of_column-det}):

\begin{corollary}\label{cor:relation_between_column_det_and_symm_det}\sl
   We have
   $$
      C_r = \operatorname{det}_r(E \,;\, r-1,r-2,\ldots,0)
      = \frac{1}{r!}\langle \tau^{(r)}, \Xi(r-1) \Xi(r-2) \cdots \Xi(0) \rangle.
   $$
   Here we define $\Xi(u) \in \Lambda(V \oplus V^*) \otimes U(\mathfrak{gl}_n)$ by
   $\Xi(u) = \sum_{i,j=1}^n e_i e^*_j E_{ij}(u) = \sum_{j=1}^n \xi_j(u) e^*_j$.
\end{corollary}

Noting this, let us prove Theorem~\ref{thm:Capelli_identity_in_T}. 
This proof is parallel to that of the original Capelli identity given in \cite{I3}
(see also \cite{U5}).

\begin{proof}[Proof of Theorem~{\sl \ref{thm:Capelli_identity_in_T}}]
We work in $\Lambda(V \oplus V^*) \otimes \mathcal{L}(\mathbb{C}^n \otimes \mathbb{C}^{n'})$.
We put
$$
   \Xi = \Xi_{\nu(E)} = \sum_{i,j=1}^n e_i e^*_j \nu(E_{ij}),\qquad
   \Xi(u) = \Xi_{\nu(E)}(u) = \sum_{i,j=1}^n e_i e^*_j \nu(E_{ij}(u)),
$$
so that
\begin{equation}\label{eq:relation_between_Xi_and_eta}
   \Xi(0) = \Xi = \sum_{j=1}^{n'} \eta_j \eta^*_j.
\end{equation}
Here we define $\eta_j$ and $\eta^*_j$ by
$\eta_j = \sum_{a=1}^n e_a L(w_{aj})$ and 
$\eta^*_j = \sum_{a=1}^n e^*_a L(w^*_{aj})$.
The following commutation relations are straightforward from Theorem~\ref{thm:analogue_of_CCR}:
\begin{equation}\label{eq:commutation_relations_for_eta}
   \eta_i \eta_j = - \eta_j \eta_i s_1, \qquad
   \eta^*_i \eta^*_j = - s_1 \eta^*_j \eta^*_i, \qquad
   \eta^*_i \eta_j = - \eta_j s_1 \eta^*_i - \delta_{ij}\tau.
\end{equation}
In particular, we have
\begin{equation}\label{eq:commutation_relations_for_eta2}
   \eta_{j_r} \cdots \eta_{j_1} \sigma
   = \eta_{j_{\sigma(r)}} \cdots \eta_{j_{\sigma(1)}}, \qquad
   \sigma^{-1} \eta^*_{j_1} \cdots \eta^*_{j_r}
   = \eta^*_{j_{\sigma(1)}} \cdots \eta^*_{j_{\sigma(r)}}.
\end{equation}
Moreover, combining (\ref{eq:relation_between_Xi_and_eta}) 
and (\ref{eq:commutation_relations_for_eta}), 
we have 
\begin{equation}\label{eq:comm_relation_between_Xi_and_eta}
   \Xi(u+1) \eta_j = \eta_j \Xi(u).
\end{equation}
Using the relations (\ref{eq:relation_between_Xi_and_eta}) 
and (\ref{eq:comm_relation_between_Xi_and_eta}) repeatedly, we have
\begin{align*}
   & \Xi(r-1) \Xi(r-2) \cdots \Xi(2) \Xi(1) \Xi(0) \\
   & \qquad
   = \sum_{1 \leq j_r \leq n'} 
   \Xi(r-1) \Xi(r-2) \cdots \Xi(2) \Xi(1) \eta_{j_r} \eta^*_{j_r} \\
   & \qquad
   = \sum_{1 \leq j_r \leq n'} 
   \eta_{j_r} \Xi(r-2) \Xi(r-3) \cdots \Xi(1)  \Xi(0) \eta^*_{j_r} \\
   & \qquad
   = \sum_{1 \leq j_{r-1},j_r \leq n'} 
   \eta_{j_r} 
   \Xi(r-2) \Xi(r-3) \cdots \Xi(1) 
   \eta_{j_{r-1}}\eta^*_{j_{r-1}} \eta^*_{j_r} \\
   & \qquad
   = \sum_{1 \leq j_{r-1},j_r \leq n'} 
   \eta_{j_r} \eta_{j_{r-1}} \Xi(r-3) \Xi(r-4) \cdots \Xi(0) 
   \eta^*_{j_{r-1}} \eta^*_{j_r} \\
   & \qquad\qquad \vdots \\
   & \qquad 
   = \sum_{1 \leq j_1,\ldots,j_r \leq n'} 
   \eta_{j_r} \eta_{j_{r-1}} \cdots \eta_{j_1}
   \eta^*_{j_1} \cdots \eta^*_{j_{r-1}} \eta^*_{j_r} \\
   & \qquad
   = \sum_{J \in \bibinom{[n']}{r}} \frac{r!}{J!}
   \eta_{j_r}  \eta_{j_{r-1}} \cdots \eta_{j_1}
   \eta^*_{j_1} \cdots \eta^*_{j_{r-1}} \eta^*_{j_r}.
\end{align*}
The last equality is seen from (\ref{eq:commutation_relations_for_eta2}).
Taking the coupling with $\tau^{(r)}$ and 
and using (\ref{eq:expression_of_column-det1}) and 
Corollary~\ref{cor:relation_between_column_det_and_symm_det},
we have the assertion.
\end{proof}

%%%%%%%%%%%%%%%%%%%%%%%%%%%%%%%%%%%%%%%%%%%%%%%%%%%%%%%%%%%%%%%%%%%%%%%%%%%%%%%%%%
\subsection{}
%%%%%%%%%%%%%%%%%%%%%%%%%%%%%%%%%%%%%%%%%%%%%%%%%%%%%%%%%%%%%%%%%%%%%%%%%%%%%%%%%
%
Using Theorem~\ref{thm:Capelli_identity_in_T} for the case $r = n$, 
we can determine the $SL_n(\mathbb{C})$-invariants in 
$T(\mathbb{C}^n \otimes \mathbb{C}^{n'})$
and $\bar{T}(\mathbb{C}^n \otimes \mathbb{C}^{n'})$
(namely, we have the first fundamental theorem 
of invariant theory for $SL_n(\mathbb{C})$ in the tensor algebra).

\begin{theorem}\label{thm:FFT_for_SL_in_T}\sl
   We put
   $$
      \mathcal{I} 
      = \left\{ \operatorname{column-det} Y_{I_n^{\circ} J^{\circ}} 
      \,\left|\, J \in \bibinom{[n']}{n} \right.\right\}.
   $$
   Here $Y$ denotes the matrix $Y = (w_{ij})_{1 \leq i \leq n, \, 1 \leq j \leq n'}$
   in $\operatorname{Mat}_{n,n'}(\bar{T}(\mathbb{C}^n \otimes \mathbb{C}^{n'}))$,
   and $I_n$ means the sequence $I_n = (1,2,\ldots,n)$. 
   Then we have the following assertions:
   \begin{enumerate}
      \item[(i)]
      The algebra $\bar{T}(\mathbb{C}^n \otimes \mathbb{C}^{n'})^{SL_n(\mathbb{C})}$
      is generated by $\mathcal{I}$ and $S_{\infty}$.
      \item[(ii)]
      The algebra $T(\mathbb{C}^n \otimes \mathbb{C}^{n'})^{SL_n(\mathbb{C})}$ 
      is spanned by the elements in the form $A_1 \cdots A_k \sigma$.
      Here $A_1,\ldots,A_k$ are elements of $\mathcal{I}$, 
      and $\sigma$ is an element of $S_{rn}$.
   \end{enumerate}
\end{theorem}

\begin{proof}
For $\varphi \in \bar{T}_k(\mathbb{C}^n \otimes \mathbb{C}^{n'})^{SL_n(\mathbb{C})}$,
we have $\nu(E_{ij}) \varphi = \delta_{ij} \frac{k}{n} \varphi$.
This is immediate from the invariance of $\varphi$ under the action of $SL_n(\mathbb{C})$.
Noting this and the definition of the column-determinant, we have
$$
   \nu(C_n) \varphi
   = 
   \left( \frac{k}{n} + n-1 \right) 
   \left( \frac{k}{n} + n-2 \right) \cdots 
   \left( \frac{k}{n} + 0 \right) \varphi.
$$
On the other hand, from Theorem~\ref{thm:Capelli_identity_in_T}, we see
$$
   \nu(C_n) \varphi = \sum_{J \in \bibinom{[n']}{n}} \frac{1}{J!} 
   \operatorname{column-det} X_{I_n^{\circ} J^{\circ}} 
   \operatorname{column-det} X^*_{I_nJ} \varphi.
$$
The quantity $\operatorname{column-det} X^*_{I_nJ} \varphi$ 
is $SL_n(\mathbb{C})$-invariant,
because the operator $\operatorname{column-det} X^*_{I_nJ}$ 
commutes with the action of $SL_n(\mathbb{C})$.
Thus, we can express $\varphi$ as $\varphi = \sum_{i=1}^l A_i \varphi_i$.
Here $A_i$ is an element of $\mathcal{I}$,
and $\varphi_i$ is an element of $\bar{T}_{k-n}(\mathbb{C}^n \otimes \mathbb{C}^{n'})^{SL_n(\mathbb{C})}$.
Repeating this argument, we see the assertion by induction on $k$.
\end{proof}

%%%%%%%%%%%%%%%%%%%%%%%%%%%%%%%%%%%%%%%%%%%%%%%%%%%%%%%%%%%%%%%%%%%%%%
%
\section{Other natural extensions of the tensor algebra}\label{sec:def_of_variants}
%
%%%%%%%%%%%%%%%%%%%%%%%%%%%%%%%%%%%%%%%%%%%%%%%%%%%%%%%%%%%%%%%%%%%%%%
%
Suggested by the operator algebra $\mathcal{L}(V)$ introduced in 
Section~\ref{sec:multiplications_and_derivations},
we can consider similar algebras associated to two vector spaces $V$ and $W$. 
We also consider the opposite algebra of $\bar{T}(V)$.
These algebras are useful to treat the immanants in the next section.

%%%%%%%%%%%%%%%%%%%%%%%%%%%%%%%%%%%%%%%%%%%%%%%%%%%%%%%%%%%%%%%%%%%%%%
\subsection{}
%%%%%%%%%%%%%%%%%%%%%%%%%%%%%%%%%%%%%%%%%%%%%%%%%%%%%%%%%%%%%%%%%%%%%%
%
We consider the associative algebra $\mathcal{C}'_{m,n}$
defined by the following generators and relations:
\begin{align*}
   \text{generators: } &
   e_1,\ldots,e_m,f_1,\ldots,f_n,s_1,s_2,\ldots, \\
   \text{relations: } 
   & e_b e_a = e_a e_b s_1, \quad
   f_b f_a = s_1 f_a f_b, \quad
   f_b e_a = e_a s_1 f_b + c_{ab}, \\
   & s_i e_a = e_a s_{i+1}, \quad
   f_a s_i = s_{i+1} f_a, \\
   & s_i^2 = 1, \quad
   s_i s_{i+1} s_i = s_{i+1} s_i s_{i+1}, \quad
   s_i s_j = s_j s_i \text{ when $|i-j|>1$}.
\end{align*}
Here we fix arbitrary complex numbers $c_{ab}$ 
for $1 \leq a \leq m$ and $1 \leq b \leq n$.
Note that this $\mathcal{C}'_{m,n}$ is equal to 
$\mathcal{B}_n \simeq \mathcal{L}(V)$
defined in (\ref{eq:gen_and_rel_of_L(V)}), 
when $m=n$ and $c_{ab} = \delta_{ab}$.
It is also fundamental to consider the case  $c_{ab} = 0$ for all $a$ and $b$,
and we denote $\mathcal{C}'_{m,n}$ by $\mathcal{C}_{m,n}$ in this case.
It appears natural to regard $\mathcal{C}'_{m,n}$
as a ``quantization'' of $\mathcal{C}_{m,n}$.

Let us regard $e_1,\ldots,e_m$ and $f_1,\ldots,f_n$ 
as bases of two vector spaces $V$ and $W$.
Moreover we consider a bilinear map $\langle \cdot,\cdot \rangle \colon 
W \times V \to \mathbb{C}$
determined by $\langle f_b,e_a \rangle = c_{ab}$.
Then, as seen soon, 
the algebra $\mathcal{C}'_{m,n}$ can be identified with the vector space
$$
   \bar{T}'(V,W) = 
   \bigoplus_{p,q \geq 0} \bar{T}'_{p,q}(V,W).
$$
Here we put 
$$
   \bar{T}'_{p,q}(V,W) = 
   V^{\otimes p} \otimes_{\mathbb{C}S_p} \mathbb{C}S_{\infty} 
   \otimes_{\mathbb{C}S_q} W^{\otimes q}.
$$

\begin{proposition} \sl
   The linear extension of the mapping
   $$
      \bar{T}'(V,W) \to \mathcal{C}'_{m,n}, \quad
      v_p \cdots v_1 t w_1 \cdots w_q \mapsto v_p \cdots v_1 t w_1 \cdots w_q
   $$
   defines a linear isomorphism from $\bar{T}'(V,W)$ onto $\mathcal{C}'_{m,n}$.
   Here $v_i$, $w_j$, and $t$ are elements of $V$, $W$, and $\mathbb{C}S_{\infty}$, respectively.
\end{proposition}

\begin{proof}
By the definition of $\mathcal{C}'_{m,n}$,
this correspondence induces a linear map from $\bar{T}'(V,W)$ onto $\mathcal{C}'_{m,n}$.
Thus, it suffices to show that this is injective.

We consider a sufficiently large vector space $U$ and its linear dual $U^*$.
Then we can identify $W$ and $V$ with 
subspaces $W \subset U^*$ and $V \subset U$
such that the restriction of the natural coupling of $U^*$ and $U$
to $W \times V$
is equal to the bilinear map $\langle \cdot,\cdot \rangle$.
As seen in Corollary~\ref{cor:def_of_L(V)_in_terms_of_tensors},
$\bar{T}'(U,U^*)$ is isomorphic to
the operator algebra $\mathcal{L}(U)$.
Moreover, we have a natural algebra homomorphism from $\mathcal{C}'_{m,n}$ to $\mathcal{B}_{\dim U}$
defined in (\ref{eq:gen_and_rel_of_L(V)}),
Similarly we have an injective linear map from $\bar{T}'(V,W)$ to $\mathcal{L}(U)$.
Thus we see that this correspondence
between $\bar{T}'(V,W)$ and $\mathcal{C}'_{m,n}$ is injective.
\end{proof}

We introduce an algebra structure on the vector space $\bar{T}'(V,W)$
through this identification with $\mathcal{C}'_{m,n}$.

Let us denote $\bar{T}'(V,W)$ and $\bar{T}'_{p,q}(V,W)$ 
by $\bar{T}(V,W)$ and $\bar{T}_{p,q}(V,W)$, respectively,
when $c_{ab}=0$ (namely when the bilinear map $\langle \cdot,\cdot \rangle
\colon W \times V \to \mathbb{C}$ is trivial).
Then $\bar{T}(V,W) \simeq \mathcal{C}_{m,n}$ becomes a graded algebra,
where $\bar{T}_{p,q}(V,W)$ is the homogeneous part of degree $(p,q)$.
Note that $\bar{T}'(V,W)$ has only a structure of filtered algebra in general.

The following is easily seen:

\begin{proposition}\label{prop:ef_is_central}\sl
   Let $v$ and $w$ be elements of $V$ and $W$, respectively.
   Then, in the algebra $\bar{T}(V,W)$, the element $vw$ is central,
   though $wv$ is not central.
\end{proposition}

%%%%%%%%%%%%%%%%%%%%%%%%%%%%%%%%%%%%%%%%%%%%%%%%%%%%%%%%%%%%%%%%%%%%%%
\subsection{}
%%%%%%%%%%%%%%%%%%%%%%%%%%%%%%%%%%%%%%%%%%%%%%%%%%%%%%%%%%%%%%%%%%%%%%
%
Let us denote the opposite algebra of $\bar{T}(V)$ by $\bar{T}^{\circ}(V)$.
Namely, we put
$$
   \bar{T}^{\circ}(V) 
   = \bigoplus_{p \geq 0} \bar{T}^{\circ}_p(V) \qquad
   \text{with} \quad
   \bar{T}^{\circ}_p(V) 
   = \mathbb{C}S_{\infty} \otimes_{\mathbb{C}S_p} V^{\otimes p},
$$
and define the multiplication on $\bar{T}^{\circ}(V)$ by the formula
$$
   \varphi \varphi'
   = (\sigma v_1 \cdots v_p) (\sigma' v'_1 \cdots v'_{p'}) 
   = \sigma \alpha^p(\sigma') v_1 \cdots v_p v'_1 \cdots v'_{p'}
$$
for  $\varphi = \sigma v_1 \cdots v_p \in \bar{T}^{\circ}_p(V)$ and 
$\varphi' = \sigma' v'_1 \cdots v'_{p'} \in \bar{T}^{\circ}_{p'}(V)$.
Then, the following is a subalgebra of $\bar{T}^{\circ}(V)$,
and can be regarded as the opposite algebra of $T^{(q)}(V)$:
$$
   T^{(q)\circ}(V) 
   = \bigoplus_{p \geq 0} T^{(q)\circ}_p(V) \qquad
   \text{with} \quad
   T^{(q)\circ}_p(V) 
   = \mathbb{C}S_{p+q} \otimes_{\mathbb{C}S_p} V^{\otimes p}.
$$

The algebras $\bar{T}(V)$ and $\bar{T}^{\circ}(V)$
are isomorphic to special cases of the algebra $\bar{T}(V,W)$:
$$
   \bar{T}(V) \simeq \bar{T}(V,\{0\}), \qquad
   \bar{T}^{\circ}(W) \simeq \bar{T}(\{0\}, W).
$$

%%%%%%%%%%%%%%%%%%%%%%%%%%%%%%%%%%%%%%%%%%%%%%%%%%%%%%%%%%%%%%%%%%%%%%%%%%%%%%%%%%
\subsection{}\label{subsec:def_of_couplings}
%%%%%%%%%%%%%%%%%%%%%%%%%%%%%%%%%%%%%%%%%%%%%%%%%%%%%%%%%%%%%%%%%%%%%%%%%%%%%%%%%
%
Let us consider some natural bilinear maps.
First, we define the bilinear map
$$
   \langle \cdot,\cdot \rangle \colon
   \bar{T}^{\circ}_p(V^*) \times \bar{T}_p(V) 
   \to \mathbb{C}S_{\infty}
$$
by
$$
   \langle \sigma' v^*_1 \cdots v^*_p, v_p \cdots v_1 \sigma \rangle
   = \sigma' L(v^*_1) \cdots L(v^*_p) v_p \cdots v_1 \sigma.
$$
It is easily seen that this is well defined.
Namely, for $\varphi^* \in \bar{T}^{\circ}_p(V^*)$, $\varphi \in \bar{T}_p(V)$
and $\sigma$, $\sigma' \in S_{\infty}$,
we have
$\langle \sigma' \varphi^*, \varphi \sigma \rangle
= \sigma' \langle \varphi^*, \varphi \rangle \sigma$.

Restricting this to 
$T^{(0) \circ}_p(V^*) \times T^{(0)}_p(V) = T^{\circ}_p(V^*) \times T_p(V)$,
we obtain a map
$$
   \langle \cdot,\cdot \rangle \colon
   T^{\circ}_p(V^*) \times T_p(V) 
   \to \mathbb{C}S_p. 
$$
Here $T^{(0)\circ}_p(V^*)$ is isomorphic to $T_p(V^*)$,
but we denote this by $T^{\circ}_p(V^*)$ 
to emphasize that we regard this as a subspace of $\bar{T}^{\circ}(V^*)$. 
Composing this with the irreducible character $\chi_{\lambda}$ of $S_p$ 
determined by a partition $\lambda$ of $p$, 
we have a bilinear form.
Namely, we define
$$
   \langle \cdot,\cdot \rangle_{\lambda} \colon
   T^{\circ}_p(V^*) \times T_p(V) \to \mathbb{C}
$$
by $\langle x^*,x \rangle_{\lambda} = \chi_{\lambda} (\langle x^*,x \rangle)$.

The map $\langle \cdot,\cdot \rangle$ can be generalized to a bilinear map 
$$
   \diamond \colon
   \bar{T}_{p,q}(U,V^*) \times \bar{T}_{q,r}(V,W) \to \bar{T}_{p,r}(U,W) 
$$
defined by
$$
   (u_p \cdots u_1 \sigma v^*_1 \cdots v^*_q)
   \diamond 
   (v_q \cdots v_1 \sigma' w_1 \cdots w_r)
   = u_p \cdots u_1 \sigma 
   \langle v^*_1 \cdots v^*_q, v_q \cdots v_1 \rangle 
   \sigma' w_1 \cdots w_r.
$$
It is easily seen that this multiplication $\diamond$ is well defined 
and associative.

Suggested by the relation (\ref{eq:expression_of_symm_det}),
we also define $\langle \cdot \rangle \colon \bar{T}_{p,p}(V,V^*) \to \mathbb{C}S_{\infty}$ by
$$
   \langle \Phi \rangle
   = \frac{1}{p!} \sum_{I \in [n]^p} 
   e^*_{i_1} \cdots e^*_{i_p} \diamond \Phi \diamond e_{i_p} \cdots e_{i_1}.
$$
When $\langle \Phi \rangle \in \mathbb{C}S_p$, 
we denote $\chi_{\lambda}(\langle \Phi \rangle)$ simply by $\langle \Phi \rangle_{\lambda}$.

%%%%%%%%%%%%%%%%%%%%%%%%%%%%%%%%%%%%%%%%%%%%%%%%%%%%%%%%%%%%%%%%%%%%%%%%%%%%%%%%%%
\subsection{}
%%%%%%%%%%%%%%%%%%%%%%%%%%%%%%%%%%%%%%%%%%%%%%%%%%%%%%%%%%%%%%%%%%%%%%%%%%%%%%%%%
%
We can use these bilinear maps to provide some central elements in $\mathbb{C}S_p$.
Let us give some preliminaries to the discussion in 
Sections~\ref{sec:immanants_and_algebras} and \ref{sec:quantum_immanants}.

First, the following is easily seen from the definition of $\langle \cdot \rangle$.

\begin{proposition}\label{prop:centrality_of_<Phi>}\sl
   When $\Phi \in \bar{T}_{p,p}(V,V^*)$,
   the element $\langle \Phi \rangle \in \mathbb{C}S_{\infty}$ commutes with any element in $S_p$.
\end{proposition}

Let us consider $\varphi^{}_J \in T_p(V)$ and $\varphi^*_J \in T^{\circ}_p(V^*)$ 
indexed by $J \in [n]^p$ satisfying 
\begin{equation}\label{eq:key_relation_to_make_central_elements_in_CS}
   \varphi^{}_J \sigma = \varphi^{}_{\sigma(J)}, \qquad
   \sigma^{-1} \varphi^*_J = \varphi^*_{\sigma(J)}.
\end{equation}
For example, $\varphi^{}_J = v_{j_p} \cdots v_{j_1}$ and $\varphi^*_J = v^*_{j_1} \cdots v^*_{j_p}$
satisfy these relations
for any $v_1,\ldots,v_n \in V$ and $v^*_1,\ldots,v^*_n \in V^*$.
For these, we have the following proposition:

\begin{proposition}\label{prop:relations_for_varphi_and_varphi*}\sl
   Let
   $\varphi^{}_J \in T_p(V)$ and $\varphi^*_J \in T^{\circ}_p(V^*)$ 
   satisfy (\ref{eq:key_relation_to_make_central_elements_in_CS}).
   \begin{enumerate}
   \item[(i)]
   We have
   $\sigma^{-1} \langle \varphi^*_I, \varphi^{}_J \rangle \sigma' 
   = \langle \varphi^*_{\sigma(I)}, \varphi^{}_{\sigma'(J)} \rangle$.
   \item[(ii)]
   We have
   $
      \sum_{J \in [n]^p} \langle \varphi^*_J, \varphi^{}_J \rangle
      = \sum_{J \in [n]^p}\langle \varphi^{}_J \varphi^*_J \rangle.
   $
   In particular, this quantity is central in $\mathbb{C}S_p$.
   \end{enumerate}
\end{proposition}

\begin{proof}
The proof of (i) is easy.
We can deduce (ii) from
$$
   \langle \varphi^{}_J \varphi^*_J \rangle
   = \frac{1}{p!} \sum_{\sigma \in S_p} 
   \sigma^{-1} \langle \varphi^*_J, \varphi^{}_J \rangle \sigma
   = \frac{1}{p!} \sum_{\sigma \in S_p} 
   \langle \varphi^*_{\sigma(J)}, \varphi^{}_{\sigma(J)} \rangle.
$$
Here we used (i) and the following relation 
for $\varphi \in T_p(V)$ and $\varphi^* \in T^{\circ}_p(V^*)$:
\begin{equation*}
   \langle \varphi \varphi^* \rangle
   = \frac{1}{p!} \sum_{\sigma \in S_p} \sigma^{-1} \langle \varphi^*, \varphi \rangle \sigma.
\qedhere
\end{equation*}
\end{proof}

%%%%%%%%%%%%%%%%%%%%%%%%%%%%%%%%%%%%%%%%%%%%%%%%%%%%%%%%%%%%%%%%%%%%%%%%%%%%%%%%%%
%
\section{immanants}\label{sec:immanants}
%
%%%%%%%%%%%%%%%%%%%%%%%%%%%%%%%%%%%%%%%%%%%%%%%%%%%%%%%%%%%%%%%%%%%%%%%%%%%%%%%%%
%
The algebras $\bar{T}(V)$, $\bar{T}^{\circ}(V)$, and $\bar{T}(V,V^*)$
are useful to treat a matrix function called the ``immanant.''
In this section, we recall this notion and see its basic properties.
We also introduce the notion of ``preimmanants.''

%%%%%%%%%%%%%%%%%%%%%%%%%%%%%%%%%%%%%%%%%%%%%%%%%%%%%%%%%%%%%%%%%%%%%%%%%%%%%%%%%%
\subsection{}
%%%%%%%%%%%%%%%%%%%%%%%%%%%%%%%%%%%%%%%%%%%%%%%%%%%%%%%%%%%%%%%%%%%%%%%%%%%%%%%%%
%
The immanant is a natural generalization of the determinant and the permanent.
Let $\mathcal{A}$ be an associative $\mathbb{C}$-algebra,
and consider an $n \times n$ matrix $X = (X_{ij})_{1 \leq i,j \leq n}$ 
in $\operatorname{Mat}_n(\mathcal{A})$.
When $\mathcal{A}$ is commutative,
we define the immanant by
\begin{equation}\label{eq:definition_of_immanant}
   \operatorname{imm}_{\lambda} X 
   = \sum_{\sigma \in S_n} \chi_{\lambda} (\sigma) 
   X_{\sigma(1)1} \cdots X_{\sigma(n)n}
   = \sum_{\sigma \in S_n} \chi_{\lambda} (\sigma^{-1}) 
   X_{1\sigma(1)} \cdots X_{n\sigma(n)}.
\end{equation}
Here $\lambda$ is a partition of $n$,
and $\chi_{\lambda}$ is the irreducible character of $S_n$ determined by $\lambda$.
This is equal to the determinant
and the permanent when $\lambda = (1^n)$ and $\lambda = (n)$, respectively:
$$
   \operatorname{imm}_{(1^n)} = \det, \qquad
   \operatorname{imm}_{(n)} = \operatorname{per}.
$$

When $\mathcal{A}$ is not commutative,
the two expressions in (\ref{eq:definition_of_immanant}) do not coincide in general,
so that we should consider the following two functions independently: 
\begin{align*}
   \operatorname{column-imm}_{\lambda} X 
   & = \sum_{\sigma \in S_n} \chi_{\lambda} (\sigma) 
   X_{\sigma(1)1} \cdots X_{\sigma(n)n},\\
   \operatorname{row-imm}_{\lambda} X 
   & = \sum_{\sigma \in S_n} \chi_{\lambda} (\sigma^{-1}) 
   X_{1\sigma(1)} \cdots X_{n\sigma(n)}.
\end{align*}
Additionally we consider 
\begin{align*}
   \operatorname{double-imm}_{\lambda} X 
   & = \frac{\chi_{\lambda}(1)}{n!}
   \sum_{\sigma, \sigma' \in S_n} 
   \chi_{\lambda} (\sigma) \chi_{\lambda} (\sigma^{\prime -1}) 
   X_{\sigma(1)\sigma'(1)} \cdots X_{\sigma(n)\sigma'(n)},\\
   \operatorname{symm-imm}_{\lambda} X 
   & = \frac{1}{n!}
   \sum_{\sigma, \sigma' \in S_n} 
   \chi_{\lambda} (\sigma \sigma^{\prime -1}) 
   X_{\sigma(1)\sigma'(1)} \cdots X_{\sigma(n)\sigma'(n)}.
\end{align*}
These names are parallel to those of the noncommutative determinants 
in Section~\ref{sec:Capelli_and_FFT} and \cite{I1}--\cite{I6}, \cite{IU}.
These four functions are all equal to $\operatorname{imm}_{\lambda}X$,
when $\mathcal{A}$ is commutative:

\begin{proposition}\label{prop:relation_among_imm}\sl
   If $\mathcal{A}$ is commutative, we have
   $$
      \operatorname{column-imm}_{\lambda} X
      = \operatorname{row-imm}_{\lambda} X 
      = \operatorname{double-imm}_{\lambda} X 
      = \operatorname{symm-imm}_{\lambda} X.
   $$
\end{proposition}

This coincidence is easy from the following two relations 
of the irreducible characters:
\begin{equation}\label{eq:relations_for_characters}
   |S_n| \frac{\chi_{\lambda}(\sigma)}{\chi_{\lambda}(1)}
   = \sum_{\sigma' \in S_n} \chi_{\lambda}(\sigma\sigma^{\prime -1}) 
   \chi_{\lambda}(\sigma'), \qquad
   \chi_{\lambda}(\sigma) = \chi_{\lambda}(\sigma^{-1}).
\end{equation}
The first relation holds for general finite groups,
and the second one is deduced from the fact that
the irreducible representations of $S_n$ are all over $\mathbb{R}$.

\begin{proof}[Proof of Proposition~{\ref{prop:relation_among_imm}}]
The equality between ``$\operatorname{double-imm}$'' and 
``$\operatorname{column-imm}$'' is shown as follows:
\begin{align*}
   &
   \sum_{\sigma,\sigma' \in S_n} 
   \chi_{\lambda}(\sigma) \chi_{\lambda}(\sigma^{\prime -1}) 
   X_{\sigma(1) \sigma'(1)} \cdots X_{\sigma(n) \sigma'(n)} \\ 
   & \qquad
   = \sum_{\sigma,\sigma' \in S_n} 
   \chi_{\lambda}(\sigma\sigma') \chi_{\lambda}(\sigma^{\prime -1}) 
   X_{\sigma\sigma'(1) \sigma'(1)} \cdots 
   X_{\sigma\sigma'(n) \sigma'(n)} \\
   & \qquad
   = \sum_{\sigma,\sigma' \in S_n} 
   \chi_{\lambda}(\sigma\sigma') \chi_{\lambda}(\sigma^{\prime -1}) 
   X_{\sigma(1)1} \cdots X_{\sigma(n)n} \\
   & \qquad
   = \sum_{\sigma \in S_n} 
   \frac{n!}{\chi_{\lambda}(1)} \chi_{\lambda}(\sigma) 
   X_{\sigma(1)1} \cdots X_{\sigma(n)n}.
\end{align*}
Here we replaced $\sigma$ by $\sigma\sigma'$ in the first equality.
Moreover, we changed the order of $X_{ij}$'s in the second equality.
The third equality is from the first relation of 
(\ref{eq:relations_for_characters}).

The proofs of the other equalities are similar, so that we omit them.
\end{proof}

The following relation follows
from the fact that $\chi_{\lambda}$ is a class function:

\begin{proposition}\label{prop:invariance_of_symm_imm_under_permutation}\sl
   Even if $\mathcal{A}$ is noncommutative,
   we have 
   $$
      \operatorname{symm-imm}_{\lambda} \sigma^{-1} X \sigma
      = \operatorname{symm-imm}_{\lambda} X
   $$
   for $X \in \operatorname{Mat}_n(\mathcal{A})$
   and $\sigma \in S_n$.
   Here we define the left and right actions of $\sigma \in S_n$
   on $\operatorname{Mat}_n(\mathcal{A})$ by
   $$
      \sigma (X_{ij})_{1 \leq i,j \leq n} = (X_{\sigma^{-1}(i)j})_{1 \leq i,j \leq n},\qquad
      (X_{ij})_{1 \leq i,j \leq n} \sigma = (X_{i\sigma(j)})_{1 \leq i,j \leq n}.
   $$
   In particular, we have 
   $\operatorname{imm}_{\lambda} \sigma^{-1} X \sigma
   = \operatorname{imm}_{\lambda} X$,
   when $\mathcal{A}$ is commutative.
\end{proposition}

%%%%%%%%%%%%%%%%%%%%%%%%%%%%%%%%%%%%%%%%%%%%%%%%%%%%%%%%%%%%%%%%%%%%%%%%%%%%%%%%%%
\subsection{}
%%%%%%%%%%%%%%%%%%%%%%%%%%%%%%%%%%%%%%%%%%%%%%%%%%%%%%%%%%%%%%%%%%%%%%%%%%%%%%%%%
%
The Cauchy--Binet identity is known for determinants.
Namely, when the entries of $n \times n'$ matrix $X$ and 
$n' \times n^{\prime\prime}$ matrix $Y$ are commutative,
we have
$$
   \det(XY)_{IK} = \sum_{J \in \mybinom{[n']}{p}} \det X_{IJ} \det Y_{JK}.
$$
We have a similar relation for immanants:

\begin{proposition}\label{prop:Cauchy_Binet_for_immanants}
   \sl
   Consider two matrices $X \in \operatorname{Mat}_{n,n'}(\mathcal{A})$ 
   and $Y \in \operatorname{Mat}_{n',n^{\prime\prime}}(\mathcal{A})$.
   When $\mathcal{A}$ is commutative,
   we have the following relation for $I \in [n]^p$ and $K \in [n^{\prime\prime}]^p$:
   $$
      \operatorname{imm}_{\lambda}(XY)_{IK} 
	  = \frac{\chi_{\lambda}(1)}{p!} \sum_{J \in [n']^p}
      \operatorname{imm}_{\lambda}X_{IJ} \operatorname{imm}_{\lambda}Y_{JK}.
   $$
\end{proposition}

The higher Capelli identity given in \cite{O1} can be regarded as
a noncommutative analogue of this relation
(we will discuss this in Section~\ref{sec:higher_Capelli}).

\begin{proof}[Proof of Proposition~{\sl \ref{prop:Cauchy_Binet_for_immanants}}]
This is seen from the following calculation:
\begin{align*}
   & \operatorname{double-imm}_{\lambda}(XY)_{IK} \\
   & \qquad
   = \frac{\chi_{\lambda}(1)}{p!} \sum_{\sigma,\sigma' \in S_p} 
   \chi_{\lambda}(\sigma) \chi_{\lambda}(\sigma^{\prime -1})
   (XY)_{i_{\sigma(1)} k_{\sigma'(1)}} \cdots 
   (XY)_{i_{\sigma(p)} k_{\sigma'(p)}} \\
   & \qquad
   = \frac{\chi_{\lambda}(1)}{p!} \sum_{\sigma,\sigma' \in S_p} 
   \sum_{J \in [n']^p}
   \chi_{\lambda}(\sigma) \chi_{\lambda}(\sigma^{\prime -1})
   X_{i_{\sigma(1)} j_1} Y_{j_1 k_{\sigma'(1)}} 
   \cdots 
   X_{i_{\sigma(p)} j_p} Y_{j_p k_{\sigma'(p)}} \\
   & \qquad
   = \frac{\chi_{\lambda}(1)}{p!}
   \sum_{J \in [n']^p} 
   \sum_{\sigma \in S_p} 
   \chi_{\lambda}(\sigma) 
   X_{i_{\sigma(1)} j_1} \cdots X_{i_{\sigma(p)} j_p}
   \sum_{\sigma' \in S_p} 
   \chi_{\lambda}(\sigma^{\prime -1})
   Y_{j_1 k_{\sigma'(1)}} \cdots Y_{j_p k_{\sigma'(p)}} \\
   & \qquad
   = \frac{\chi_{\lambda}(1)}{p!} 
   \sum_{J \in [n']^p}
   \operatorname{column-imm}_{\lambda}X_{IJ} 
   \operatorname{row-imm}_{\lambda}Y_{JK}. 
   \qedhere
\end{align*}
\end{proof}

%%%%%%%%%%%%%%%%%%%%%%%%%%%%%%%%%%%%%%%%%%%%%%%%%%%%%%%%%%%%%%%%%%%%%%%%%%%%%%%%%%
\subsection{}
%%%%%%%%%%%%%%%%%%%%%%%%%%%%%%%%%%%%%%%%%%%%%%%%%%%%%%%%%%%%%%%%%%%%%%%%%%%%%%%%%
%
Next we introduce a fundamental matrix function expressed as a sum of immanants.
This is invariant under the conjugation by $GL_n(\mathbb{C})$,
and this invariance follows from Proposition~\ref{prop:Cauchy_Binet_for_immanants}
when the matrix entries are commutative.

We assume that $\mathcal{A}$ is not necessarily commutative.
A simple calculation shows 
\begin{align*}
   \sum_{I \in [n]^p} \operatorname{column-imm}_{\lambda} X_{II}
   = \sum_{I \in [n]^p} \operatorname{row-imm}_{\lambda} X_{II}
   = \sum_{I \in [n]^p} \operatorname{symm-imm}_{\lambda} X_{II}.
\end{align*}
Let us denote by ``$\operatorname{imm}_{\lambda,p} X$'' 
this quantity divided by $p!$.
We can also express this as
$$
   \operatorname{imm}_{\lambda,p} X
   =
   \sum_{I \in \bibinom{[n]}{p}} \frac{1}{I!} 
   \operatorname{symm-imm}_{\lambda} X_{II}
$$
as seen from Proposition~\ref{prop:invariance_of_symm_imm_under_permutation}.
We can regard this function ``$\operatorname{imm}_{\lambda,p}$'' 
as the counterpart of ``$\operatorname{det}_p$'' given in Section~\ref{sec:Capelli_and_FFT}.
This function is invariant under the conjugation by $GL_n(\mathbb{C})$:

\begin{proposition}\label{prop:invariance_of_imm_p}\sl
   For any $g \in GL_n(\mathbb{C})$, we have
   $\operatorname{imm}_{\lambda,p} X
   = \operatorname{imm}_{\lambda,p} gXg^{-1}$.   
\end{proposition}

This proposition is immediate from Proposition~\ref{prop:Cauchy_Binet_for_immanants},
when $\mathcal{A}$ is commutative.
We will prove the noncommutative case later
as a corollary of Proposition~\ref{prop:preimm_p} (iii).

For $X \in \operatorname{Mat}_n(\mathbb{C})$,
this quantity $\operatorname{imm}_{\lambda,p} X$ is equal to the Schur polynomial in eigenvalues of $X$.
We will discuss this issue in Section~\ref{subsec:various_expressions_of_preimm_p}.

%%%%%%%%%%%%%%%%%%%%%%%%%%%%%%%%%%%%%%%%%%%%%%%%%%%%%%%%%%%%%%%%%%%%%%
\subsection{}\label{subsec:preimmanants}
%%%%%%%%%%%%%%%%%%%%%%%%%%%%%%%%%%%%%%%%%%%%%%%%%%%%%%%%%%%%%%%%%%%%%%
%
Let us consider the following quantities 
determined by $X \in\operatorname{Mat}_{n}(\mathcal{A})$:
\begin{align*}
   \operatorname{column-preimm}X 
   & = \sum_{\sigma \in S_n} \sigma X_{\sigma(1)1} \cdots X_{\sigma(n)n}, \\
   \operatorname{row-preimm}X
   & = \sum_{\sigma \in S_n} X_{1\sigma(1)} \cdots X_{n\sigma(n)} \sigma^{-1}, \\
   \operatorname{symm-preimm}X
   & = \frac{1}{n!} \sum_{\sigma,\sigma' \in S_n} 
   \sigma X_{\sigma(1)\sigma'(1)} \cdots X_{\sigma(n)\sigma'(n)} 
   \sigma^{\prime -1}.
\end{align*}
These are elements of the algebra $\mathbb{C}S_n \otimes \mathcal{A}$,
in which the two algebras $\mathcal{A}$ and $\mathbb{C}S_n$ 
commute with each other.
In this article,
we call these $\mathbb{C}S_n \otimes \mathcal{A}$-valued matrix functions 
the ``preimmanants'' or the ``predeterminants,''
because the immanants and the determinants can be obtained from these 
by applying the character $\chi_{\lambda}$ and the signature function:
\begin{align*}
   \operatorname{column-imm}_{\lambda} X
   &= \chi_{\lambda}(\operatorname{column-preimm}X), \\ 
   \operatorname{row-imm}_{\lambda} X
   &= \chi_{\lambda}(\operatorname{row-preimm}X), \\
   \operatorname{symm-imm}_{\lambda} X
   &= \chi_{\lambda}(\operatorname{symm-preimm}X). 
\end{align*}
Here we define the linear map $\chi_{\lambda} \colon \mathbb{C}S_n \otimes \mathcal{A} \to \mathcal{A}$
by $\chi_{\lambda}(t \otimes a) = \chi_{\lambda}(t) a$.
We also put
\begin{align*}
   \operatorname{column-imm}^{\circ} X
   &= (\operatorname{column-imm}X)^{\circ}, \\
   \operatorname{row-imm}^{\circ} X
   &= (\operatorname{row-imm}X)^{\circ}, \\
   \operatorname{symm-imm}^{\circ} X
   &= (\operatorname{symm-imm}X)^{\circ}.
\end{align*}
Here we define the antiautomorphism $t \mapsto t^{\circ}$ on $\mathbb{C}S_n$
by $\sigma^{\circ} = \sigma^{-1}$ for $\sigma \in S_n$.
Moreover we extend this to a linear transformation on $\mathbb{C}S_n \otimes \mathcal{A}$ 
by $(t \otimes a)^{\circ} = t^{\circ} \otimes a$.

When $\mathcal{A}$ is commutative, we have 
\begin{align*}
   \operatorname{column-preimm}X
   & = \operatorname{row-preimm}X 
   = \operatorname{symm-preimm}X, \\ 
   \operatorname{column-preimm}^{\circ}X 
   & = \operatorname{row-preimm}^{\circ}X 
   = \operatorname{symm-preimm}^{\circ}X
\end{align*}
by a simple calculation.
Let us denote these two quantities by
$$
   \operatorname{preimm}X, \qquad
   \operatorname{preimm}^{\circ}X,
$$
respectively.
These two are connected by the transposition:
$\operatorname{preimm}{}^t\!X = \operatorname{preimm}^{\circ} X$.

For the actions of $S_n$ on $\operatorname{Mat}_n(\mathcal{A})$
given in Proposition~\ref{prop:invariance_of_symm_imm_under_permutation},
the following relations hold:

\begin{proposition}\label{prop:preimm_and_multiplications_by_permutations} \sl
   We have
   \begin{align*}
      \operatorname{column-preimm}(\sigma X) 
	  &= \sigma (\operatorname{column-preimm}X), \\
      \operatorname{row-preimm}(X\sigma) 
	  &= (\operatorname{row-preimm}X) \sigma, \\
      \operatorname{symm-preimm}(\sigma X \sigma') 
	  &= \sigma (\operatorname{symm-preimm}X) \sigma', 
	  \displaybreak[0]\\
      \operatorname{column-preimm}^{\circ}(\sigma X) 
	  &= (\operatorname{column-preimm}^{\circ}X) \sigma, \\
      \operatorname{row-preimm}^{\circ}(X\sigma) 
	  &= \sigma(\operatorname{row-preimm}^{\circ}X), \\
      \operatorname{symm-preimm}^{\circ}(\sigma X \sigma') 
	  &= \sigma'(\operatorname{symm-preimm}^{\circ}X)\sigma.
   \end{align*}
\end{proposition}

We can regard Proposition~\ref{prop:invariance_of_symm_imm_under_permutation} 
as a corollary of this proposition.

%%%%%%%%%%%%%%%%%%%%%%%%%%%%%%%%%%%%%%%%%%%%%%%%%%%%%%%%%%%%%%%%%%%%%%
\subsection{}
%%%%%%%%%%%%%%%%%%%%%%%%%%%%%%%%%%%%%%%%%%%%%%%%%%%%%%%%%%%%%%%%%%%%%%
%
Under the assumption of commutativity,
Cauchy--Binet type identities still hold for the preimmanants:

\begin{proposition}\label{prop:Cauchy_Binet_for_preimmanants}\sl
   When $\mathcal{A}$ is commutative,
   we have the following relations
   for $X \in \operatorname{Mat}_{n,n'}(\mathcal{A})$,
   $Y \in \operatorname{Mat}_{n',n^{\prime\prime}}(\mathcal{A})$
   and $I \in [n]^p$, $K \in [n^{\prime\prime}]^p$:
   \begin{align*}
      \operatorname{preimm}(XY)_{IK}
      & = \frac{1}{p!} \sum_{J \in [n']^p} 
      \operatorname{preimm}X_{IJ} \operatorname{preimm}Y_{JK} \\
      & = \sum_{J \in \bibinom{[n']}{p}} 
      \frac{1}{J!} 
	  \operatorname{preimm}X_{IJ} 
	  \operatorname{preimm}Y_{JK}, \\
      \operatorname{preimm}^{\circ}(XY)_{IK}
      & = \frac{1}{p!} \sum_{J \in [n']^p} 
      \operatorname{preimm}^{\circ}Y_{JK} \operatorname{preimm}^{\circ}X_{IJ} \\
      & = \sum_{J \in \bibinom{[n']}{p}} 
      \frac{1}{J!} 
	  \operatorname{preimm}^{\circ}Y_{JK} 
	  \operatorname{preimm}^{\circ}X_{IJ} .
   \end{align*}
\end{proposition}

The proof is almost the same as that of 
Proposition~\ref{prop:Cauchy_Binet_for_immanants}, so that we omit it.

In Section~\ref{sec:higher_Capelli}, we will obtain
noncommutative analogues of this proposition,
which are regarded as translations of the higher Capelli identities.

Proposition~\ref{prop:Cauchy_Binet_for_preimmanants} 
does not hold in general if $\mathcal{A}$ is noncommutative,
but we can consider an analogue of $\operatorname{imm}_{\lambda,p}$
without the assumption of commutativity.

\begin{proposition}\label{prop:preimm_p}\sl
   Even if $\mathcal{A}$ is noncommutative, the following assertions hold: 
   \begin{enumerate}
   \item[(i)]
   We have
   \begin{align*}
      & \sum_{I \in [n]^p} 
	  \operatorname{column-preimm}X_{II}
	  = \sum_{I \in [n]^p} 
	  \operatorname{row-preimm}X_{II}
	  = \sum_{I \in [n]^p} 
	  \operatorname{symm-preimm}X_{II} \\
      & \quad
      = \sum_{I \in [n]^p} 
	  \operatorname{column-preimm}^{\circ}X_{II}
      = \sum_{I \in [n]^p} 
	  \operatorname{row-preimm}^{\circ}X_{II}
      = \sum_{I \in [n]^p} 
	  \operatorname{symm-preimm}^{\circ}X_{II}.
   \end{align*}
   We denote by ``$\operatorname{preimm}_p X$'' this quantity divided by $p!$.
   \item[(ii)]
   The quantity $\operatorname{preimm}_p X$ is commutative 
   with any element in $\mathbb{C}S_p$. 
   \item[(iii)] 
   The quantity $\operatorname{preimm}_p X$ is invariant 
   under the conjugation by $GL_n(\mathbb{C})$.
   \end{enumerate}
\end{proposition}

\begin{proof}[Proof of Proposition~{\sl \ref{prop:preimm_p}} (i) and (ii)]
A simple calculation shows us
$$
   \sum_{I \in [n]^p} \operatorname{column-preimm}X_{II}
   = \sum_{I \in [n]^p} \operatorname{row-preimm}X_{II}
   = \sum_{I \in [n]^p} \operatorname{symm-preimm}X_{II}.
$$
This quantity commutes with any element in $\mathbb{C}S_p$.
Indeed, we have
$$
   \sigma^{-1} (\operatorname{symm-preimm}X_{II}) \sigma
   = \operatorname{symm-preimm} (\sigma^{-1} X_{II} \sigma) 
   = \operatorname{symm-preimm}X_{\sigma(I)\sigma(I)}
$$
using Proposition~\ref{prop:preimm_and_multiplications_by_permutations}.
Since we have $t = t^{\circ}$ for any central element $t$ in $\mathbb{C}S_p$,
we conclude (i) and (ii).
\end{proof}

We will show Proposition~\ref{prop:preimm_p} (iii) using the algebra $\bar{T}(V,V^*)$ 
in Section~\ref{sec:immanants_and_algebras}.

%%%%%%%%%%%%%%%%%%%%%%%%%%%%%%%%%%%%%%%%%%%%%%%%%%%%%%%%%%%%%%%%%%%%%%
\subsection{}\label{subsec:various_expressions_of_preimm_p}
%%%%%%%%%%%%%%%%%%%%%%%%%%%%%%%%%%%%%%%%%%%%%%%%%%%%%%%%%%%%%%%%%%%%%%
%
The immanant and the preimmanant have interesting relations with symmetric functions.
Let $X$ be a matrix in $\operatorname{Mat}_n(\mathbb{C})$.
Then, as mentioned in Section~1.4 of~\cite{O1} or Section~2 of \cite{OO},
the following relation holds:

\begin{proposition}\sl
   We have
   $$
      \operatorname{imm}_{\lambda,p}X = s_{\lambda}(a_1,\ldots,a_n).
   $$
   Here $s_{\lambda}$ is the Schur polynomial, and
   $a_1,\ldots,a_n$ are the eigenvalues of $X$.
\end{proposition}

\begin{proof}
It suffices to consider the case when $X$ is triangular,
because both sides are invariant under conjugations.
In this case, the assertion follows from the relations
$$
   s_{\lambda} = \sum_{\mu \vdash p} K_{\lambda\mu} m_{\lambda}, \qquad
   \sum_{\sigma \in S_{\mu}} \chi^{}_{\lambda}(\sigma) = |S_{\mu}| K_{\lambda\mu},
$$
where the first relation is well known formula (for example (7.35) in \cite{S}) 
and the second relation is equivalent to 
$\langle \operatorname{Ind}_{S_{\mu}}^{S_p} 1,\chi_{\lambda} \rangle
= K_{\lambda\mu}$
(Proposition~7.18.7 in \cite{S}).
Here $m_{\lambda}$ is the monomial symmetric polynomial, 
$K_{\lambda\mu}$ is the Kostka number,
and $S_{\mu}$ is the Young subgroup $S_{\mu_1} \times \cdots \times S_{\mu_l}$ of $S_p$
determined by $\mu = (\mu_1,\ldots,\mu_l) \vdash p$.
\end{proof}

Next, let us interpret the preimmanant in the context of symmetric functions.
We denote by $\operatorname{Class}(S_p)$ the set of all class functions on $S_p$.
This can be identified with the center $Z\mathbb{C}S_p$ of $\mathbb{C}S_p$
through the canonical inner product on $\operatorname{Class}(S_p)$.
Moreover, through the Frobenius characteristic map,
we can identify $\operatorname{Class}(S_p)$ with $\operatorname{Symm}_p$, 
the set of all homogeneous symmetric functions in $x_1,x_2,\ldots$ of degree $p$
(Section~7.18 of \cite{S}).
For example,
the following three elements correspond to each other through this identification:
$$
   \tilde{s}_{\lambda} \in Z\mathbb{C}S_p, \qquad
   \chi_{\lambda} \in \operatorname{Class}(S_p), \qquad
   s_{\lambda} \in \operatorname{Symm}_p.
$$
Here we put $\tilde{s}_{\lambda} = \frac{1}{p!} \sum_{\sigma \in S_p} \chi_{\lambda}(\sigma)\sigma$
for $\lambda \vdash p$.
Then the following is easy:

\begin{lemma}\label{lem:tilde_s_lambda}\sl
   We have $t = \sum_{\lambda \vdash p} \chi_{\lambda}(t)\tilde{s}_{\lambda}$
   for any central element $t$ in $\mathbb{C}S_p$.
\end{lemma}

Applying this, we have
\begin{equation}\label{eq:counterpart_of_Cauchy}
   \operatorname{preimm}_p X
   = \sum_{\lambda \vdash p} \operatorname{imm}_{\lambda,p} X \cdot \tilde{s}_{\lambda}
   = \sum_{\lambda \vdash p} s_{\lambda}(a_1,\ldots,a_n) \tilde{s}_{\lambda}.
\end{equation}
We can regard this relation 
as the counterpart of the Cauchy identity (Theorem~7.12.1 in \cite{S})
or the irreducible decomposition of the regular representation of $S_p \times S_p$ on $\mathbb{C}S_p$
through the identification $Z \mathbb{C}S_p \simeq \operatorname{Class}(S_p) \simeq \operatorname{Symm}_p$.

We can generalize (\ref{eq:counterpart_of_Cauchy}) as follows:

\begin{theorem}\sl
   We have 
   $$
      \operatorname{preimm}_p X
      = \sum_{\lambda \vdash p} s_{\lambda}(a_1,\ldots,a_n) \tilde{s}_{\lambda}
      = \sum_{\lambda \vdash p} m_{\lambda}(a_1,\ldots,a_n) \tilde{h}_{\lambda}
      = \sum_{\lambda \vdash p} \frac{p!}{z_{\lambda}}
      p_{\lambda}(a_1,\ldots,a_n) \tilde{p}_{\lambda}. 
   $$
\end{theorem}
Here, $m_{\lambda}$ and $p_{\lambda}$ are
the monomial symmetric polynomial and the power sum symmetric polynomial, respectively.
Moreover, $\tilde{h}_{\lambda}$ and $\tilde{p}_{\lambda}$ in $Z \mathbb{C}S_p$
denote the counterparts of the complete homogeneous symmetric function $h_{\lambda}$
and the power sum symmetric function $p_{\lambda}$ in $\operatorname{Symm}_p$.
Namely we put
$$
   \tilde{h}_{\lambda}
   = \frac{1}{|S_p||S_{\lambda}|} 
   \sum_{\sigma \in S_p} 
   \sum_{\tau \in S_{\lambda}} 
   \sigma^{-1} \tau \sigma, \qquad
   \tilde{p}_{\lambda}
   = \frac{z_{\lambda}}{p!}\sum_{\sigma \in \operatorname{Conj}_{\lambda}} \sigma.
$$
Here $\operatorname{Conj}_{\lambda} \subset S_p$ is the conjugacy class corresponding to $\lambda$,
and $z_{\lambda}$ means $z_{\lambda} = p!/|\operatorname{Conj}_{\lambda}|$.
These $\tilde{h}_{\lambda}$ and $\tilde{p}_{\lambda}$ correspond to 
$\operatorname{Ind}_{S_{\lambda}}^{S_p} 1$
and $1_{\operatorname{Conj}_{\lambda}}$ in $\operatorname{Class}(S_p)$,
respectively.
Here $1_{\operatorname{Conj}_{\lambda}}$ is 
the characteristic function of $\operatorname{Conj}_{\lambda}$:
$$
   1_{\operatorname{Conj}_{\lambda}}(\sigma) = 
   \begin{cases}
      1, & \sigma \in \operatorname{Conj}_{\lambda}, \\
	  0,  & \sigma \not\in \operatorname{Conj}_{\lambda}.
   \end{cases}
$$

%%%%%%%%%%%%%%%%%%%%%%%%%%%%%%%%%%%%%%%%%%%%%%%%%%%%%%%%%%%%%%%%%%%%%%%%%%%%%%%%%%
%
\section{immanants and extensions of the tensor algebra}
\label{sec:immanants_and_algebras}
%
%%%%%%%%%%%%%%%%%%%%%%%%%%%%%%%%%%%%%%%%%%%%%%%%%%%%%%%%%%%%%%%%%%%%%%%%%%%%%%%%%
%
In Section~\ref{sec:Capelli_and_FFT}, we used the exterior algebra
in order to treat noncommutative determinants conveniently.
Replacing the exterior algebra by the algebras 
$\bar{T}(V)$, $\bar{T}^{\circ}(V)$, and $\bar{T}(V,V^*)$,
we can treat the noncommutative immanants similarly. 
We can regard this approach as a kind of generating function method. 
This method plays an important role to treat the quantum immanant in Section~\ref{sec:quantum_immanants}.

%%%%%%%%%%%%%%%%%%%%%%%%%%%%%%%%%%%%%%%%%%%%%%%%%%%%%%%%%%%%%%%%%%%%%%%%%%%%%%%%%%
\subsection{}
%%%%%%%%%%%%%%%%%%%%%%%%%%%%%%%%%%%%%%%%%%%%%%%%%%%%%%%%%%%%%%%%%%%%%%%%%%%%%%%%%
%
Before going to the main subject,
we note a natural construction of extensions of bilinear maps.
Let $U_1$, $U_2$, and $U_3$ be three $\mathbb{C}$-vector spaces,
and consider  a bilinear map $\langle \cdot, \cdot \rangle \colon U_1 \times U_2 \to U_3$.
We consider another $\mathbb{C}$-vector space $\mathcal{A}$, 
and define bilinear maps
$$
   \langle \cdot, \cdot \rangle_R \colon 
   U_1 \times U_2 \otimes \mathcal{A} \to U_3 \otimes \mathcal{A}, \qquad
   \langle \cdot, \cdot \rangle_L \colon 
   U_1 \otimes \mathcal{A} \times U_2 \to U_3 \otimes \mathcal{A}
$$
by
$\langle u_1, u_2 \otimes a \rangle_R = \langle u_1, u_2 \rangle \otimes a$
and 
$\langle u_1 \otimes a, u_2 \rangle_L = \langle u_1, u_2 \rangle \otimes a$.
Let us denote these $\langle \cdot,\cdot \rangle_R$ and $\langle \cdot,\cdot \rangle_L$
by the same symbol $\langle \cdot,\cdot \rangle$,
when no confusion arises.
All bilinear maps which we use in this section
are defined from the bilinear maps given in Section~\ref{subsec:def_of_couplings}
in this way.

%%%%%%%%%%%%%%%%%%%%%%%%%%%%%%%%%%%%%%%%%%%%%%%%%%%%%%%%%%%%%%%%%%%%%%%%%%%%%%%%%%
\subsection{}
%%%%%%%%%%%%%%%%%%%%%%%%%%%%%%%%%%%%%%%%%%%%%%%%%%%%%%%%%%%%%%%%%%%%%%%%%%%%%%%%%
%
First, let us see how to treat the column-immanant and the row-immanant.

Let $\mathcal{A}$ be an associative $\mathbb{C}$-algebra,
and fix $X = (X_{ij})_{1 \leq i,j \leq n} \in \operatorname{Mat}_n(\mathcal{A})$.
We consider an $n$-dimensional $\mathbb{C}$-vector space $V$
and work in the extended algebra $\bar{T}(V) \otimes \mathcal{A}$
in which the two algebras $\bar{T}(V)$ and $\mathcal{A}$ commute with each other.

Fix a basis $e_1,\ldots,e_n$ of $V$ and its dual basis $e^*_1,\ldots,e^*_n$.
We consider 
$\xi_j = \sum_{i=1}^n e_i X_{ij} \in \bar{T}(V) \otimes \mathcal{A}$.
Then we have
\begin{align*}
   \langle e^*_{i_1} \cdots e^*_{i_p},
      \xi_{j_p} \cdots \xi_{j_1} \rangle
   & = \langle
   e^*_{i_1} \cdots e^*_{i_p}, \sum_{K \in [n]^p} e_{k_p} \cdots e_{k_1} 
   X_{k_p j_p} \cdots X_{k_1 j_1} 
   \rangle \\
   & =  \sum_{\sigma \in S_p} \frac{1}{I!}
   \langle
   e^*_{i_1} \cdots e^*_{i_p},
   e_{i_{\sigma(p)}} \cdots e_{i_{\sigma(1)}}
   \rangle   
   X_{i_{\sigma(p)} j_p} \cdots X_{i_{\sigma(1)} j_1} \\
   & = 
   \sum_{\sigma \in S_p} s_I \sigma 
   X_{i_{\sigma(p)} j_p} \cdots X_{i_{\sigma(1)} j_1} \\
   & = 
   \sum_{\sigma \in S_p} \sigma 
   X_{i_{\sigma(p)} j_p} \cdots X_{i_{\sigma(1)} j_1} \\
   & = \operatorname{column-preimm} X_{I^{\circ}J^{\circ}}.
\end{align*}
Here, we see the third equality, because (\ref{eq:pairing}) implies
$$
   \langle
   e^*_{i_1} \cdots e^*_{i_p}, 
   e_{i_{\sigma(p)}} \cdots e_{i_{\sigma(1)}}
   \rangle   
   = L(e^*_{i_1}) \cdots L(e^*_{i_p}) e_{i_{\sigma(p)}} \cdots e_{i_{\sigma(1)}} 
   = I! s_I \sigma.
$$
Moreover, we see the fourth equality,
because $s_I = \frac{1}{I!} \sum_{\sigma' \in (S_p)_I} \sigma'$ and 
$$
   \sum_{\sigma \in S_p} \sigma' \sigma X_{i_{\sigma(p)}j_p} \cdots X_{i_{\sigma(1)}j_1}
   = \sum_{\sigma \in S_p} \sigma X_{i_{\sigma(p)}j_p} \cdots X_{i_{\sigma(1)}j_1}
$$
for $\sigma' \in (S_p)_I$.

We have a similar calculation for 
$\xi^*_i = \sum_{j=1}^n X_{ij} e^*_j 
\in \bar{T}^{\circ}(V^*) \otimes \mathcal{A}$.
Thus, the following lemma holds:

\begin{lemma}\label{lem:expression_of_column-imm}
   \sl
   For $I,J \in [n]^p$, we have
   \begin{align*}
      \operatorname{column-preimm} X_{I^{\circ}J^{\circ}}
      &= \langle e^*_{i_1} \cdots e^*_{i_p},
      \xi_{j_p} \cdots \xi_{j_1} \rangle, \\
      \operatorname{column-imm}_{\lambda} X_{I^{\circ}J^{\circ}}
	  &= \langle e^*_{i_1} \cdots e^*_{i_p},
      \xi_{j_p} \cdots \xi_{j_1} \rangle_{\lambda}, \\
      \operatorname{row-preimm} X_{IJ}
      &= 
	  \langle \xi^*_{i_1} \cdots \xi^*_{i_p}, 
	  e_{j_p} \cdots e_{j_1} \rangle, \\
      \operatorname{row-imm}_{\lambda} X_{IJ}
	  &= 
	  \langle \xi^*_{i_1} \cdots \xi^*_{i_p}, 
	  e_{j_p} \cdots e_{j_1} \rangle_{\lambda}.
   \end{align*}
\end{lemma}

Similarly,
for $\gamma^*_j = \sum_{i=1}^n X_{ij} e^*_i 
\in \bar{T}^{\circ}(V^*) \otimes \mathcal{A}$
and 
$\gamma_i = \sum_{j=1}^n e_j X_{ij} 
\in \bar{T}(V) \otimes \mathcal{A}$,
we have the following lemma:

\begin{lemma} \sl
   For $I,J \in [n]^p$, we have
   \begin{align*}
      \operatorname{column-preimm}^{\circ} X_{IJ}
      &= 
	  \langle 
      \gamma^*_{j_p} \cdots \gamma^*_{j_1}, 
	  e_{i_1} \cdots e_{i_p}
	  \rangle, \\
      \operatorname{column-imm}_{\lambda} X_{IJ}
	  &= 
	  \langle 
      \gamma^*_{j_p} \cdots \gamma^*_{j_1}, 
	  e_{i_1} \cdots e_{i_p}
	  \rangle_{\lambda}, \\
      \operatorname{row-preimm}^{\circ} X_{I^{\circ}J^{\circ}}
      &=
	  \langle  
	  e^*_{j_1} \cdots e^*_{j_p}, 
	  \gamma_{i_p} \cdots \gamma_{i_1}
	  \rangle, \\
      \operatorname{row-imm}_{\lambda} X_{I^{\circ}J^{\circ}}
	  &= \langle  
	  e^*_{j_1} \cdots e^*_{j_p}, 
	  \gamma_{i_p} \cdots \gamma_{i_1}
	  \rangle_{\lambda}.
   \end{align*}
\end{lemma}

%%%%%%%%%%%%%%%%%%%%%%%%%%%%%%%%%%%%%%%%%%%%%%%%%%%%%%%%%%%%%%%%%%%%%%
\subsection{}
%%%%%%%%%%%%%%%%%%%%%%%%%%%%%%%%%%%%%%%%%%%%%%%%%%%%%%%%%%%%%%%%%%%%%%
%
To treat the symmetrized immanant, the algebra $\bar{T}(V,V^*)$ is useful.
We consider the following element of $\bar{T}(V,V^*) \otimes \mathcal{A}$:
$$
   \Xi 
   = \sum_{i,j=1}^n e_i X_{ij} e^*_j 
   = \sum_{j=1}^n \xi_j e^*_j
   = \sum_{i=1}^n e_i \xi^*_i.
$$
Then, we have
\begin{align*}
   &  e^*_{i_1} \cdots e^*_{i_p}
   \diamond
   \Xi^p 
   \diamond
   e_{j_p} \cdots e_{j_1} \\
   & \quad
   = \sum_{K,L \in [n]^p} 
   e^*_{i_1} \cdots e^*_{i_p}
   \diamond
   e_{k_p} \cdots e_{k_1} 
   X_{k_1 l_1} \cdots X_{k_p l_p} 
   e^*_{l_1} \cdots e^*_{l_p} 
   \diamond
   e_{j_p} \cdots e_{j_1} \\
   & \quad
   = \sum_{\sigma, \sigma' \in S_p} \frac{1}{I!J!}
   e^*_{i_1} \cdots e^*_{i_p}
   \diamond
   e_{i_{\sigma(p)}} \cdots e_{i_{\sigma(1)}} 
   X_{i_{\sigma(1)} j_{\sigma'(1)}} \cdots X_{i_{\sigma(p)} j_{\sigma'(p)}} 
   e^*_{j_{\sigma'(1)}} \cdots e^*_{j_{\sigma'(p)}}
   \diamond
   e_{j_p} \cdots e_{j_1} \\
   & \quad
   = \sum_{\sigma, \sigma' \in S_p} 
   s_I \sigma
   X_{i_{\sigma(1)} j_{\sigma'(1)}} \cdots X_{i_{\sigma(p)} j_{\sigma'(p)}} 
   \sigma^{\prime -1} s_J \\
   & \quad
   = \sum_{\sigma, \sigma' \in S_p} 
   \sigma
   X_{i_{\sigma(1)} j_{\sigma'(1)}} \cdots X_{i_{\sigma(p)} j_{\sigma'(p)}} 
   \sigma^{\prime -1} \\
   & \quad
   = p! \operatorname{symm-preimm}X_{IJ}.
\end{align*}
Here, we used Proposition~\ref{prop:ef_is_central} to show the first equality.
Thus the following lemma holds:

\begin{lemma}\sl
   For $I, J \in [n]^p$, we have
   \begin{align*} 
      \operatorname{symm-preimm} X_{IJ}
      &=
	  e^*_{i_1} \cdots e^*_{i_p}
      \diamond
      \Xi^{(p)} 
      \diamond
      e_{j_p} \cdots e_{j_1}, \\
      \operatorname{symm-imm}_{\lambda} X_{IJ}
	  &= \chi_{\lambda}
      (e^*_{i_1} \cdots e^*_{i_p}
      \diamond
      \Xi^{(p)} 
      \diamond
      e_{j_p} \cdots e_{j_1}). 
   \end{align*}
   Here, $x^{(k)}$ denotes the divided power: $x^{(k)} = \frac{1}{k!}x^k$.
\end{lemma}

Moreover, we can express $\operatorname{preimm}_p$ and $\operatorname{imm}_{\lambda,p}$ as follows:

\begin{lemma}\label{lem:expression_of_preimm}
   \sl
   We have
   $$
   	  \operatorname{preimm}_p X
	  = \langle \Xi^{(p)} \rangle, \quad
      \operatorname{imm}_{\lambda,p} X
	  = \frac{1}{p!}\sum_{I \in [n]^p} 
      \operatorname{symm-imm}_{\lambda} X_{II}
	  = \langle \Xi^{(p)} \rangle_{\lambda}.
   $$
\end{lemma}

We can also express ``$\operatorname{symm-preimm}^{\circ}$''
using the following element of $\bar{T}(V,V^*) \otimes \mathcal{A}$:
$$
   \Gamma 
   = \sum_{i,j=1}^n e_j X_{ij} e^*_i
   = \sum_{i=1}^n \gamma_i e^*_i
   = \sum_{j=1}^n e_j \gamma^*_j.
$$

\begin{lemma} \sl
   For $I, J \in [n]^p$, we have
   \begin{align*}
      \operatorname{symm-preimm}^{\circ} X_{IJ}
      & =
      e^*_{j_1} \cdots e^*_{j_p} 
      \diamond
      \Gamma^{(p)} 
      \diamond
      e_{i_p} \cdots e_{i_1}, \\ 
      \operatorname{symm-imm}^{\circ}_{\lambda} X_{IJ}
      & =
      \chi_{\lambda}(
      e^*_{j_1} \cdots e^*_{j_p} 
      \diamond
      \Gamma^{(p)} 
      \diamond
      e_{i_p} \cdots e_{i_1}
      ).
   \end{align*}
   Moreover, we have
   $$
      \operatorname{preimm}_p X
      = \langle \Gamma^{(p)} \rangle, \quad
      \operatorname{imm}_{\lambda,p} X
      = \frac{1}{p!} \sum_{I \in [n]^p} 
      \operatorname{symm-imm}^{\circ}_{\lambda} X_{II}
      = \langle \Gamma^{(p)} \rangle_{\lambda}.
   $$
\end{lemma}

Using Lemma~\ref{lem:expression_of_preimm}, 
we can show Proposition~\ref{prop:preimm_p} (iii),
namely the invariance of $\operatorname{preimm}_p$ under the conjugation by $GL_n(\mathbb{C})$.
This also concludes Proposition~\ref{prop:invariance_of_imm_p}.

\begin{proof}[Proof of Proposition~{\sl \ref{prop:preimm_p}} (iii)]
Fix $g = (g_{ij}) \in GL_n(\mathbb{C})$,
and put $\tilde{e}_j = \sum_{i=1}^n e_i g_{ij}$,
so that $\tilde{e}_1,\ldots,\tilde{e}_n$ form a basis of $V$.
We can express the dual basis $\tilde{e}^*_1,\ldots,\tilde{e}^*_n$ 
as $\tilde{e}^*_i = \sum_{j=1}^n e^*_j g^{ij}$,
where $g^{ij}$ is the $(i,j)$th entries of $g^{-1}$.
We put
$$
   \Xi_{gXg^{-1}} = \sum_{i,j=1}^n e_i(gXg^{-1})_{ij}e^*_j, \qquad
   \tilde{\Xi}_X = \sum_{i,j=1}^n \tilde{e}_i X_{ij} \tilde{e}^*_j.
$$
Then, we can easily see that these are equal:
$\Xi_{gXg^{-1}} = \tilde{\Xi}_X$.
Thus we have
$$
   \operatorname{preimm}_p gXg^{-1}
   = \langle \Xi_{gXg^{-1}}^{(p)} \rangle 
   = \langle \tilde{\Xi}^{(p)} \rangle
   = \langle \Xi^{(p)} \rangle
   = \operatorname{preimm}_p X.
$$
Indeed, $\langle \,\cdot\, \rangle$ does not depend on basis chosen.
This means the assertion.
\end{proof}

Besides Proposition~\ref{prop:preimm_p} (iii),
we can easily see
various relations in Section~\ref{sec:immanants}
from the lemmas in this section.
For example, Proposition~\ref{prop:preimm_and_multiplications_by_permutations} 
is immediate from these lemmas and the properties of $\langle \cdot , \cdot \rangle$
and  $\langle \cdot \rangle$.
Proposition~\ref{prop:preimm_p} (ii) is also easy
from Proposition~\ref{prop:centrality_of_<Phi>}.

This method is more effective against the quantum immanants
as seen in the next section.

%%%%%%%%%%%%%%%%%%%%%%%%%%%%%%%%%%%%%%%%%%%%%%%%%%%%%%%%%%%%%%%%%%%%%%%%%%%%%%%%%%
%
\section{Quantum immanants}
\label{sec:quantum_immanants}
%
%%%%%%%%%%%%%%%%%%%%%%%%%%%%%%%%%%%%%%%%%%%%%%%%%%%%%%%%%%%%%%%%%%%%%%%%%%%%%%%%%
%
The quantum immanants are a generalization of the Capelli elements,
and form a basis of the center of the universal enveloping algebra
$U(\mathfrak{gl}_n)$ as a vector space.
These were first introduced in~\cite{O1},
and have been studied by using the $R$-matrix method, fusion procedure, and representation theory
of the Yangian $Y(\mathfrak{gl}_n)$
(\cite{O1}, \cite{O2}, \cite{OO}, \cite{M1}, \cite{M2}, \cite{N2}).
In this section, 
we employ the expressions of immanants given in the previous section
instead of these standard methods,
and show various relations for the quantum immanants.
Most of these relations were already given in~\cite{O1}
(and the others can also be shown by the same method). 
Thus it is not that our approach is stronger than established ones,
but the author thinks that our method has its own advantage 
as a natural analogue of the usual generating function method.
In other words, our approach is regarded as an advanced version 
of the exterior calculus in the study of the Capelli type identities in 
\cite{IU}, \cite{I1}--\cite{I6}, \cite{U2}--\cite{U5}, \cite{Ha}, \cite{Wa}
(and also in Section~\ref{sec:Capelli_and_FFT} of this article).

It is also interesting that
the notion of ``quantum preimmanant'' naturally appears.

%%%%%%%%%%%%%%%%%%%%%%%%%%%%%%%%%%%%%%%%%%%%%%%%%%%%%%%%%%%%%%%%%%%%%%%%%%%%%%%%%
\subsection{}
%%%%%%%%%%%%%%%%%%%%%%%%%%%%%%%%%%%%%%%%%%%%%%%%%%%%%%%%%%%%%%%%%%%%%%%%%%%%%%%%%
%
First, we recall some notions and facts of from the representation theory of the symmetric group.
See \cite{JK} and \cite{O1} for the details.

We identify a partition with the corresponding Young diagram.
For a cell $(i,j)$ of a Young diagram,
we call the number $j-i$ the content of the cell.
For example, for the Young diagram corresponding to $(4,3,1)$,
each cell has the following value as its content:
$$
   \begin{matrix}
   0  & 1 & 2 & 3 \\
   -1 & 0 & 1 &   \\
   -2 &   &   &
   \end{matrix}
$$
Let $\operatorname{STab}(\lambda)$ be the set of all 
standard tableaux of shape $\lambda$.
For $T \in \operatorname{STab}(\lambda)$,
we put $c_T(k) = j-i$, when the $k$th cell is $(i,j)$.
For example, we have $c_T(1) = 0$, $c_T(2) = -1$, and $c_T(3) = 1$ for the tableau
$$
   T = \,\,
   \begin{matrix}
   1  & 3 \\
   2 & 
   \end{matrix}
   \,\,
   \in \operatorname{STab}(2,1).
$$

We consider the Young orthogonal form $\rho_\lambda$
of the irreducible representation of $S_p$
determined by a partition $\lambda \vdash p$ (\cite{JK}, \cite{O1}).
We denote by $v_T$ the vector in the Young orthogonal basis
corresponding to $T \in \operatorname{STab}(\lambda)$.
Moreover, let $\rho_{\lambda}(\sigma)_{TT'}$ be
the $(T,T')$th matrix entry of $\rho_{\lambda}(\sigma)$.
Namely, we define this by
$$
   \rho_{\lambda}(\sigma) v_T 
   = \sum_{T' \in \operatorname{STab}(\lambda)} \rho_{\lambda}(\sigma)_{TT'} v_{T'}. 
$$
We can describe the matrix entries of $\rho_{\lambda}(s_i)$ as follows.
We put $r_T(i) = c_T(i+1) -c_T(i)$.
Then, for $T \in \operatorname{STab}(\lambda)$, we have 
$$
   \rho_{\lambda}(s_i) v_T = r_T(i) v_T,
$$
when $s_i T$ is non-standard
(in this case, $r_T(i)$ is equal to $\pm 1$).
Here we consider the natural action of permutation in $S_p$ on a tableau.
When $T' = s_i T$ is standard, we have
$$
   \begin{pmatrix}
   \rho_{\lambda}(s_i) v_T & 
   \rho_{\lambda}(s_i) v_{T'} 
   \end{pmatrix}   
   = 
   \begin{pmatrix}
   v_T & 
   v_{T'}
   \end{pmatrix}  
   \begin{pmatrix}
   r^{-1} & \sqrt{1 - r^{-2}} \\
   \sqrt{1 - r^{-2}} & -r^{-1} 
   \end{pmatrix},
$$
where $r = r_T(i)$.
Using these and the recurrence formula (\ref{prop:entries_of_JM_elements}), 
we can prove Proposition~\ref{prop:entries_of_JM_elements} by induction on $i$.

Finally
we consider the Jucys--Murphy elements $x_k \in \mathbb{C}S_p$ 
(\cite{J}, \cite{O1}):
$$
   x_k = \sum_{i=1}^{k-1} (i \,\, k)
   = (1 \,\, k) + (2 \,\, k) + \cdots + (k-1 \,\, k).
$$
For $k=1$, we put $x_1 = 0$.
Then $\rho_{\lambda}(x_i)$ is diagonal,
and its entries are expressed in terms of contents:

\begin{proposition}\label{prop:entries_of_JM_elements}\sl
   We have $\rho_{\lambda}(x_i) v_T = c_T(i) v_T$,
   namely $\rho_{\lambda}(x_i)_{TT'} = \delta_{TT'} c_T(i)$.
\end{proposition}

This can be proved by using the recurrence formula
\begin{equation}\label{eq:recurrence_formula}
   x_{i+1} = s_i x_i s_i + s_i
\end{equation}
and the description of the matrix entries of $\rho_{\lambda}(s_i)$ above.

%%%%%%%%%%%%%%%%%%%%%%%%%%%%%%%%%%%%%%%%%%%%%%%%%%%%%%%%%%%%%%%%%%%%%%%%%%%%%%%%%
\subsection{}
%%%%%%%%%%%%%%%%%%%%%%%%%%%%%%%%%%%%%%%%%%%%%%%%%%%%%%%%%%%%%%%%%%%%%%%%%%%%%%%%%
%
The main objects of this section are 
the following elements in the universal enveloping algebra $U(\mathfrak{gl}_n)$:
\begin{align*}
   G_{\lambda}
   & = 
   \frac{\chi_{\lambda}(1)}{p!}
   \sum_{J \in [n]^p}   
   \sum_{\sigma \in S_p}
   \rho_{\lambda}(\sigma)_{TT}
   E_{j_{\sigma(1)} j_1} (c_T(1)) 
   E_{j_{\sigma(2)} j_2} (c_T(2)) 
   \cdots 
   E_{j_{\sigma(p)} j_p} (c_T(p)) \\
   & = 
   \frac{\chi_{\lambda}(1)}{p!}
   \sum_{J \in [n]^p}   
   \sum_{\sigma \in S_p}
   \rho_{\lambda}(\sigma^{-1})_{TT}
   E_{j_p j_{\sigma(p)}} (c_T(p)) 
   \cdots
   E_{j_2 j_{\sigma(2)}} (c_T(2)) 
   E_{j_1 j_{\sigma(1)}} (c_T(1)), \\
   G^{\circ}_{\lambda}
   & = 
   \frac{\chi_{\lambda}(1)}{p!}
   \sum_{J \in [n]^p}   
   \sum_{\sigma \in S_p}
   \rho_{\lambda}(\sigma^{-1})_{TT}
   E_{j_{\sigma(p)} j_p} (-c_T(p)) 
   \cdots
   E_{j_{\sigma(2)} j_2} (-c_T(2)) 
   E_{j_{\sigma(1)} j_1} (-c_T(1)) \\
   & = 
   \frac{\chi_{\lambda}(1)}{p!}
   \sum_{J \in [n]^p}   
   \sum_{\sigma \in S_p}
   \rho_{\lambda}(\sigma)_{TT}
   E_{j_1 j_{\sigma(1)}} (-c_T(1)) 
   E_{j_2 j_{\sigma(2)}} (-c_T(2)) 
   \cdots 
   E_{j_p j_{\sigma(p)}} (-c_T(p)).
\end{align*}
Here $\lambda$ is a partition of $p$,
and $T$ is a standard tableau of shape $\lambda$.
We will see that these $G_{\lambda}$ and $G^{\circ}_{\lambda}$ are central in $U(\mathfrak{gl}_n)$,
and these expressions do not depend on $T \in \operatorname{STab}(\lambda)$.
These can also be expressed as 
$$
   G_{\lambda} = \chi_{\lambda}(G_p), \qquad
   G^{\circ}_{\lambda} = \chi_{\lambda}(G^{\circ}_p),
$$
where we define $G_p$ and $G^{\circ}_p \in \mathbb{C}S_p \otimes U(\mathfrak{gl}_n)$ 
as follows (these are also central):
\begin{align*}
   G_p
   & = 
   \frac{1}{p!}
   \sum_{J \in [n]^p}   
   \sum_{\sigma \in S_p}
   \sigma
   E_{j_{\sigma(1)} j_1} (x_1) 
   E_{j_{\sigma(2)} j_2} (x_2) 
   \cdots 
   E_{j_{\sigma(p)} j_p} (x_p) \\
   & = 
   \frac{1}{p!}
   \sum_{J \in [n]^p}   
   \sum_{\sigma \in S_p}
   E_{j_p j_{\sigma(p)}} (x_p) 
   \cdots
   E_{j_2 j_{\sigma(2)}} (x_2) 
   E_{j_1 j_{\sigma(1)}} (x_1) 
   \sigma^{-1}, \\
   G^{\circ}_p
   & = 
   \frac{1}{p!}
   \sum_{J \in [n]^p}   
   \sum_{\sigma \in S_p}
   E_{j_{\sigma(p)} j_p} (-x_p) 
   \cdots
   E_{j_{\sigma(2)} j_2} (-x_2) 
   E_{j_{\sigma(1)} j_1} (-x_1) 
   \sigma^{-1} \\
   & = 
   \frac{1}{p!}
   \sum_{J \in [n]^p}   
   \sum_{\sigma \in S_p}
   \sigma
   E_{j_1 j_{\sigma(1)}} (-x_1) 
   E_{j_2 j_{\sigma(2)}} (-x_2) 
   \cdots 
   E_{j_p j_{\sigma(p)}} (-x_p). 
\end{align*}
We will show the equivalence of these various expressions later.

We can regard $G_{\lambda}$ and $G^{\circ}_{\lambda}$ (resp.\ $G_p$ and $G^{\circ}_p$) 
as modifications of $\operatorname{imm}_{\lambda,p} X$ (resp.\ $\operatorname{preimm}_p X$).
Indeed, the highest terms of these elements
with respect to the canonical filtration of $U(\mathfrak{gl}_n)$
are equal to  $\operatorname{imm}_{\lambda,p} E$ and $\operatorname{preimm}_p E$,
respectively.
To be more specific, 
$G_p$ and $G^{\circ}_p$ are obtained by modifying
$\frac{1}{p!}\sum_I \operatorname{preimm}X_{II}$ and 
$\frac{1}{p!}\sum_I \operatorname{preimm}^{\circ}X_{II}$, respectively.
It is interesting that two different modifications naturally appear.

The second central element $G^{\circ}_{\lambda}$ of $U(\mathfrak{gl}_n)$ 
is known as the ``quantum immanant'' \cite{O1}
(this is denoted by $\mathbb{S}_{\lambda}$ in \cite{O1}).
Thus, it is natural to call $G^{\circ}_p$ the ``quantum preimmanant.''
As seen in Section~\ref{sec:higher_Capelli}, 
we have beautiful Capelli type identities for the quantum immanant $G^{\circ}_{\lambda}$ 
(the higher Capelli identities).
This fact seems to indicate that $G^{\circ}_p$ is more fundamental than $G_p$,
though ``$\operatorname{preimm}$'' is simpler than ``$\operatorname{preimm}^{\circ}$''
in Proposition~\ref{prop:Cauchy_Binet_for_preimmanants}.

It is easily seen that 
these two are related by the automorphism of $U(\mathfrak{gl}_n)$ defined by $E_{ij} \mapsto -E_{ji}$
as follows: $G_p \mapsto (-)^p G^{\circ}_p$.

\begin{remark}
   When $\lambda = (1^p)$, 
   we can express $G_{\lambda}$ and $G^{\circ}_{\lambda}$ as follows:
   \begin{align*}
      G_{(1^p)} 
      & = 
      \frac{1}{p!}
      \sum_{J \in [n]^p}
      \operatorname{column-det}(E_{JJ} - \operatorname{diag}(0,1,\ldots,n-1)), \\
      G^{\circ}_{(1^p)}
      & = 
      \frac{1}{p!}
      \sum_{J \in [n]^p}
      \operatorname{column-det}(E_{JJ} + \operatorname{diag}(n-1,n-2,\ldots,0)).
   \end{align*}
   In particular, $G^{\circ}_{(1^p)}$ is equal to the Capelli element $C_p$ 
   seen in Section~\ref{sec:Capelli_and_FFT}.
   Similarly, we can express $G_{(p)}$ and $G^{\circ}_{(p)}$ in terms of column-permanent
   ($G^{\circ}_{(p)}$ was first given and studied 
   in~\cite{N1}).
\end{remark}

Let us express these elements
using the algebras $\bar{T}(V)$, $\bar{T}^{\circ}(V)$, and $\bar{T}(V,V^*)$.
Under these expressions, we can easily show various relations for these elements.

%%%%%%%%%%%%%%%%%%%%%%%%%%%%%%%%%%%%%%%%%%%%%%%%%%%%%%%%%%%%%%%%%%%%%%%%%%%%%%%%%
\subsection{}\label{subsec:first_expression_of_G}
%%%%%%%%%%%%%%%%%%%%%%%%%%%%%%%%%%%%%%%%%%%%%%%%%%%%%%%%%%%%%%%%%%%%%%%%%%%%%%%%%
%
Let $V$ be an $n$-dimensional $\mathbb{C}$-vector space
with basis $e_1,\ldots,e_n$,
and consider
$\xi_j = \sum_{i=1}^n e_i E_{ij} \in \bar{T}(V) \otimes U(\mathfrak{gl}_n)$.
Then we have the commutation relation
\begin{align}\label{eq:comm_relation_of_xi_1}
   \xi_j \xi_l 
   & = \sum_{i,k=1}^n e_i e_k E_{ij} E_{kl} \\
   & = \sum_{i,k=1}^n e_i e_k
   (E_{kl} E_{ij} + E_{il} \delta_{kj} - E_{kj} \delta_{il}) 
   \notag \\
   & =  \xi_l \xi_j \cdot (1 \,\, 2) + \xi_l e_j - e_l \xi_j, 
   \notag
\end{align}
namely
$$
   \xi_j (\xi_l +  e_l \cdot (1 \,\, 2)) 
   = \xi_l  (\xi_j + e_j \cdot (1 \,\, 2)) \cdot (1 \,\, 2).
$$
This can be generalized as follows:

\begin{lemma}\label{lem:comm_relation_of_xi_3}\sl 
   We put $\xi_j(u) = \sum_{i=1}^n e_i E_{ij}(u) = \xi_j + e_j u$.
   Then we have
   $$
      \xi_i(y_l) \xi_j(y_{l+1}) 
	  = \xi_j(y_l) \xi_i(y_{l+1}) s_1.
   $$
   Here we define $y_k \in \mathbb{C}S_p$ as follows (for $k=1$, we put $y_1 = 0$):
   $$
      y_k 
      = \sum_{i=1}^{k-1} (1 \,\,\, i+1) 
      = (1 \,\, 2) + (1 \,\, 3) + \cdots + (1 \,\, k).
   $$
\end{lemma}

This relation can be regarded as an analogue of (\ref{eq:comm_rel_for_xi}) 
in the exterior algebra.

Before proving Lemma~{\ref{lem:comm_relation_of_xi_3}}, 
we consider $x^{\circ}_k \in \mathbb{C}S_p$ defined by
$$
   x^{\circ}_k = \sum_{i=1}^{k-1} (p-i+1 \,\,\, p-k+1).
$$
For $k=1$, we put $x^{\circ}_1 = 0$.
The elements $x_k$, $x^{\circ}_k$, and  $y_k$ are connected by the relations
\begin{equation}\label{eq:relation_between_x_and_y}
   \varepsilon^{-1} x_i \varepsilon = x^{\circ}_i, \qquad\qquad
   \alpha^{p-i}(y_i) = x^{\circ}_i,
\end{equation}
where $\varepsilon$ is the following element in $S_p$:
$$
   \varepsilon = 
   \begin{pmatrix}
   1 & 2 & \hdots & p-1 & p \\
   p & p-1 & \hdots & 2 & 1 
   \end{pmatrix}.
$$

\begin{proof}[Proof of Lemma~{\sl\ref{lem:comm_relation_of_xi_3}}]
By a simple calculation, we have
\begin{align}\label{eq:expansion_of_LHS}
   & (\xi_i + e_i y_l) (\xi_j + e_j y_{l+1}) \\
   & \qquad
   = \xi_i \xi_j 
      + e_i e_j \alpha(y_l) y_{l+1}
	  + e_i \xi_j \alpha(y_l) 
	  + \xi_i e_j y_{l+1} \notag \\
   & \qquad
   = \{ \xi_j \xi_i s_1 + \xi_j e_i - e_j \xi_i \} 
      + e_j e_i s_1 \alpha(y_l) y_{l+1}
	  + \xi_j e_i s_1 \alpha(y_l) 
	  +  e_j \xi_i s_1 y_{l+1} \notag \\
   & \qquad
   = \xi_j \xi_i s_1
      + e_j e_i s_1 \alpha(y_l) y_{l+1}
	  + \xi_j e_i s_1 (\alpha(y_l) + s_1) 
	  +  e_j \xi_i s_1 (y_{l+1} - s_1). \notag
\end{align}
Here, in the second equality, we used (\ref{eq:comm_relation_of_xi_1}).
Since we have
\begin{gather*}
   s_1 \alpha(y_l) s_1 = y_{l+1} - s_1, \qquad
   s_1 y_{l+1} s_1 = \alpha(y_l) + s_1, \qquad
   \alpha(y_l) y_{l+1} = y_{l+1 }\alpha(y_l), \\
   s_1 \alpha(y_l) y_{l+1} 
   = (y_{l+1} - s_1) s_1 y_{l+1}
   = y_{l+1} s_1 (y_{l+1} - s_1)
   = y_{l+1} \alpha(y_l) s_1,
\end{gather*}
the last line of (\ref{eq:expansion_of_LHS}) is equal to
$$
   \xi_j \xi_i s_1
   + e_j e_i \alpha(y_l)y_{l+1} s_1
   + \xi_j e_i y_{l+1}  s_1 
   + e_j \xi_i \alpha(y_l) s_1 
   = (\xi_j + e_j y_l) (\xi_i + e_i y_{l+1}) s_1.
$$
This means the assertion.
\end{proof}

We see the following from Lemma~\ref{lem:comm_relation_of_xi_3},
because $\sigma \xi_{i}(y_{k+1}) = \xi_{i}(y_{k+1}) \alpha(\sigma)$ for $\sigma \in S_k$:

\begin{corollary}\sl 
   For $\sigma \in S_p$, we have
   $$ 
      \xi_{i_p} (y_1) 
	  \xi_{i_{p-1}} (y_2) 
	  \cdots 
      \xi_{i_1} (y_p) \sigma \\
      = 
      \xi_{i_{\sigma(p)}} (y_1) 
      \xi_{i_{\sigma(p-1)}} (y_2) 
      \cdots 
      \xi_{i_{\sigma(1)}} (y_p).
   $$
   Namely, the element 
   $\varphi^{}_J = \xi_{j_p} (y_1) \xi_{j_{p-1}} (y_2) \cdots \xi_{j_1} (y_p)$
   satisfies the relation (\ref{eq:key_relation_to_make_central_elements_in_CS}).
\end{corollary}

Recalling Proposition~\ref{prop:relations_for_varphi_and_varphi*} (ii),
we put
\begin{align*}
   G^{IJ} 
   &= 
   \langle
   e^*_{i_1} \cdots e^*_{i_p},
   \xi_{j_p} (y_1) 
   \xi_{j_{p-1}} (y_2) 
   \cdots 
   \xi_{j_1} (y_p)
   \rangle, \\
   G_p 
   &= \frac{1}{p!} \sum_{J \in [n]^p} G^{JJ} \\
   &= \frac{1}{p!} \sum_{J \in [n]^p} 
   \langle
   e^*_{j_1} \cdots e^*_{j_p},
   \xi_{j_p} (y_1) 
   \xi_{j_{p-1}} (y_2) 
   \cdots 
   \xi_{j_1} (y_p)
   \rangle \\
   &= \frac{1}{p!} \sum_{J \in [n]^p} 
   \langle
   \xi_{j_p} (y_1) 
   \xi_{j_{p-1}} (y_2) 
   \cdots 
   \xi_{j_1} (y_p)
   e^*_{j_1} \cdots e^*_{j_p}
   \rangle \\
   & = 
   \frac{1}{p!}
   \langle 
   \sum_{I,J \in [n]^p} 
   e_{i_p} \cdots e_{i_1} 
   E_{i_p j_p}(x^{\circ}_1) \cdots E_{i_1 j_1}(x^{\circ}_p)
   e^*_{j_1} \cdots e^*_{j_p}
   \rangle 
\end{align*}
The last equality is seen from the second relation of (\ref{eq:relation_between_x_and_y}).
Moreover we put
$$
   G_{\lambda} = \chi_{\lambda}(G_p)
   = \frac{1}{p!}
   \sum_{J \in [n]^p}   
   \chi_{\lambda}(G^{JJ}).
$$
We aim to prove the following theorem:

\begin{theorem}\label{thm:centrality_of_G_p}\sl
   The element $G_p$ is central in $\mathbb{C}S_p \otimes U(\mathfrak{gl_n})$.
\end{theorem}

\begin{corollary}\sl
   The element $G_{\lambda}$ is central in $U(\mathfrak{gl}_n)$.
\end{corollary}

First, the following is a consequence of 
Proposition~\ref{prop:relations_for_varphi_and_varphi*} (ii):

\begin{proposition}\label{prop:invariance_of_G_p_under_CS}\sl
   The element $G_p$ commutes with any element in $\mathbb{C}S_p$.
\end{proposition}

Thus it suffices to show the commutativity with elements of $U(\mathfrak{gl}_n)$.
For this, we note the following relation.
This is seen by a straightforward calculation
(see \cite{O1} or \cite{IU} for the details of the proof).

\begin{lemma}\label{lem:adjoint_action_on_E}\sl 
   The matrix $E$ satisfies the following relation for any $g \in GL_n(\mathbb{C})$:
   $$
      \operatorname{Ad}(g) E = {}^t\!g \cdot E \cdot {}^t\!g^{-1}.
   $$
   Here $\operatorname{Ad}(g) E$ means the matrix
   $(\operatorname{Ad}(g) E_{ij})_{1 \leq i,j \leq n}$.
\end{lemma}

Combining this lemma and Proposition~\ref{prop:expressions_of_G^IJ_and_G_p},
we can show the invariance of $G_p$ 
under the conjugation by $GL_n(\mathbb{C})$
in a way similar to that of Proposition~\ref{prop:preimm_p} (iii). 

\begin{theorem}\label{thm:invariance_of_G_p_under_GL}\sl 
   We have $\operatorname{Ad}(g) G_p = G_p$ for any $g \in GL_n(\mathbb{C})$.
   In particular, $G_p$ commutes with any element of $\mathfrak{gl}_n$.
\end{theorem}

Theorem~\ref{thm:centrality_of_G_p} follows
from this and Proposition~\ref{prop:invariance_of_G_p_under_CS}.

\begin{proof}[Proof of Theorem~{\sl \ref{thm:invariance_of_G_p_under_GL}}]
Using Lemma~\ref{lem:adjoint_action_on_E}, we have
\begin{align*}
   & \operatorname{Ad}(g)
   \sum_{I,J \in [n]^p} 
   e_{i_p} \cdots e_{i_1} 
   E_{i_p j_p}(x^{\circ}_1) \cdots E_{i_1 j_1}(x^{\circ}_p)
   e^*_{j_1} \cdots e^*_{j_p} \\
   & \qquad
   = \sum_{I,J \in [n]^p} 
   e_{i_p} \cdots e_{i_1} 
   ({}^t\!g \cdot E \cdot {}^t\!g^{-1})_{i_p j_p}(x^{\circ}_1) 
   \cdots ({}^t\!g \cdot E \cdot {}^t\!g^{-1})_{i_1 j_1}(x^{\circ}_p)
   e^*_{j_1} \cdots e^*_{j_p} \\
   & \qquad
   = \sum_{I,J \in [n]^p} 
   \tilde{e}_{i_p} \cdots \tilde{e}_{i_1} 
   E_{i_p j_p}(x^{\circ}_1) \cdots E_{i_1 j_1}(x^{\circ}_p)
   \tilde{e}^*_{j_1} \cdots \tilde{e}^*_{j_p}.
\end{align*}
Here we put $\tilde{e}_i = \sum_{j=1}^n e_j g_{ij}$ 
and $\tilde{e}^*_i = \sum_{i=1}^n e^*_i g^{ij}$,
so that these are dual bases of each other. 
Applying $\langle \cdot \rangle$ to this 
and using Proposition~\ref{prop:expressions_of_G^IJ_and_G_p}, 
we obtain the assertion,
because $\langle \cdot \rangle$ does not depend on basis chosen.
\end{proof}

Let us find other expressions of $G_p$.
From Proposition~\ref{prop:relations_for_varphi_and_varphi*} (i),
we see
\begin{equation}\label{eq:actions_to_G^IJ}
   G^{\sigma(I) \sigma'(J)} = \sigma^{-1} G^{IJ} \sigma'.
\end{equation}
Noting this, we can express $G^{IJ}$ and $G_p$ as follows:

\begin{proposition}\label{prop:expressions_of_G^IJ_and_G_p}\sl
   We have
   \begin{align*}
      G^{IJ}
      &= \sum_{\sigma \in S_p}
      \sigma
      E_{i_{\sigma(p)} j_p} (x^{\circ}_1) 
      E_{i_{\sigma(p-1)} j_{p-1}} (x^{\circ}_2) 
      \cdots 
      E_{i_{\sigma(1)} j_1} (x^{\circ}_p) \\
      &= 
      \sum_{\sigma \in S_p}
      \sigma
      E_{i_{\sigma(1)} j_1} (x_1) 
      E_{i_{\sigma(2)} j_2} (x_2) 
      \cdots 
      E_{i_{\sigma(p)} j_p} (x_p), 
      \displaybreak[0]\\
      G_p
	  & = \frac{1}{p!^2}
      \sum_{I \in [n]^p} \sum_{\sigma, \sigma' \in S_p} 
      \sigma
      E_{i_{\sigma(p)} i_{\sigma'(p)}}(x^{\circ}_1) 
      \cdots 
      E_{i_{\sigma(1)} i_{\sigma'(1)}}(x^{\circ}_p) 
      \sigma^{\prime-1} \\
      & = \frac{1}{p!^2}
      \sum_{I \in [n]^p} \sum_{\sigma, \sigma' \in S_p} 
      \sigma
      E_{i_{\sigma(1)} i_{\sigma'(1)}}(x_1) 
      \cdots 
      E_{i_{\sigma(p)} i_{\sigma'(p)}}(x_p) 
      \sigma^{\prime-1}.
   \end{align*}
\end{proposition}

\begin{proof}
Using the second relation of (\ref{eq:relation_between_x_and_y}),
we have
$$
   \varphi^{}_J 
   = \sum_{I \in [n]^p} e_{i_p} e_{i_{p-1}} \cdots e_{i_1}
   E_{i_p j_p} (x^{\circ}_1) 
   E_{i_{p-1} j_{p-1}} (x^{\circ}_2) 
   \cdots 
   E_{i_1 j_1} (x^{\circ}_p). 
$$
The first expression of $G^{IJ}$ follows from this
in a way similar to the proof of Lemma~\ref{lem:expression_of_column-imm}.
Moreover, using (\ref{eq:actions_to_G^IJ}) and the first relation of (\ref{eq:relation_between_x_and_y}),
we have
\begin{align*}
   G^{IJ}
   &= \varepsilon G^{\varepsilon(I)\varepsilon(J)} \varepsilon^{-1} \\
   &= 
   \varepsilon \sum_{\sigma \in S_p}
   \sigma
   E_{i_{\varepsilon \sigma(p)} j_{\varepsilon(p)}} (x^{\circ}_1) 
   E_{i_{\varepsilon \sigma(p-1)} j_{\varepsilon(p-1)}} (x^{\circ}_2) 
   \cdots 
   E_{i_{\varepsilon \sigma(1)} j_{\varepsilon(1)}} (x^{\circ}_p) \varepsilon^{-1} \\
   &= 
   \sum_{\sigma \in S_p}
   \varepsilon \sigma \varepsilon^{-1}
   E_{i_{\varepsilon \sigma(p)} j_{\varepsilon(p)}} (x_1) 
   E_{i_{\varepsilon \sigma(p-1)} j_{\varepsilon(p-1)}} (x_2) 
   \cdots 
   E_{i_{\varepsilon \sigma(1)} j_{\varepsilon(1)}} (x_p).
\end{align*}
Replacing $\sigma$ with $\varepsilon^{-1} \sigma \varepsilon$, we have
\begin{align*}
      G^{IJ}
      & = 
      \sum_{\sigma \in S_p}
      \sigma
      E_{i_{\sigma \varepsilon(p)} j_{\varepsilon(p)}} (x_1) 
      E_{i_{\sigma \varepsilon(p-1)} j_{\varepsilon(p-1)}} (x_2) 
      \cdots 
      E_{i_{\sigma \varepsilon(1)} j_{\varepsilon(1)}} (x_p) \\
      & = 
      \sum_{\sigma \in S_p}
      \sigma
      E_{i_{\sigma(1)} j_1} (x_1) 
      E_{i_{\sigma(2)} j_2} (x_2) 
      \cdots 
      E_{i_{\sigma(p)} j_p} (x_p). 
\end{align*}
This means the second expression of $G^{IJ}$.
The expressions of $G_p$ are shown similarly.
\end{proof}

We can express $G_{\lambda}$ simply in terms of contents 
of the Young diagram $\lambda$ as follows:

\begin{theorem}\label{thm:expression_of_G_lambda_corresponding_to_the_symmetrized_immanant}\sl
   For any $T \in \operatorname{STab}(\lambda)$,
   we have
   \begin{align*}
      G_{\lambda} 
      & = \frac{\chi_{\lambda}(1)}{p!}
      \sum_{I \in [n]^p} \sum_{\sigma \in S_p}
      \rho_{\lambda}(\sigma)_{TT}
      E_{i_{\sigma(1)} i_1}(c_T(1))
      \cdots
      E_{i_{\sigma(p)} i_p}(c_T(p)) \\
      & = \frac{\chi_{\lambda}(1)}{p!}
      \left<
      \Xi(c_T(1))
      \cdots
      \Xi(c_T(p))
      \right>_T \\
      & = \frac{1}{p!}
      \left<
      \Xi(c_T(1))
      \cdots
      \Xi(c_T(p))
      \right>_{\lambda}.
   \end{align*}
   In particular, these expressions do not depend on $T$.
   Here we define $\Xi(u)$ by $\Xi(u) = \sum_{i,j=1}^n e_i E_{ij}(u) e^*_j = \Xi + u \tau$,
   and we put
   $\langle \Phi \rangle_T = \rho_{\lambda}(\langle \Phi \rangle)_{TT}$
   for $\Phi \in T_{p,p}(V,V^*)$.
\end{theorem}

\begin{proof}
Note that $\rho_{\lambda}(A)$ is a ``scalar matrix'' 
(namely $\rho_{\lambda}(A)_{TT'} = \delta_{TT'} \frac{1}{\chi_{\lambda}(1)}\chi_{\lambda}(A)$)
by Schur's lemma,
when $A \in \mathbb{C}S_p \otimes \mathcal{A}$ commutes with any element in $\mathbb{C} S_p$.
Thus we see the last equality of the assertion,
because $\left< \Xi(c_T(1)) \cdots \Xi(c_T(p)) \right>$ commutes with any element in $\mathbb{C}S_p$.
Similarly we have $G_{\lambda} = \chi_{\lambda}(1) \rho_{\lambda}(G_p)_{TT}$.
Noting this and the relation $\rho_{\lambda}(x_i)_{TT} = c_T(i)$ 
in Proposition~\ref{prop:entries_of_JM_elements},
we can express $G_{\lambda}$ as
\begin{align*}
   G_{\lambda} 
   & = \frac{\chi_{\lambda}(1)}{p!}
   \sum_{I \in [n]^p} \sum_{\sigma \in S_p}
   \rho_{\lambda}(
   \sigma
   E_{i_{\sigma(1)} i_1}(x_1)
   \cdots
   E_{i_{\sigma(1)} i_p}(x_p) 
   )_{TT}\\
   & = \frac{\chi_{\lambda}(1)}{p!}
   \sum_{I \in [n]^p} \sum_{\sigma \in S_p}
   \rho_{\lambda}(\sigma)_{TT}
   E_{i_{\sigma(1)} i_1}(c_T(1))
   \cdots
   E_{i_{\sigma(p)} i_p}(c_T(p)).
\end{align*}
This can be rewritten as
\begin{align*}
   G_{\lambda} 
   & = \frac{\chi_{\lambda}(1)}{p!}
   \sum_{J \in [n]^p}   
   \left<
   e^*_{j_p} \cdots e^*_{j_1},
   \xi_{j_1}(c_T(1))
   \cdots
   \xi_{j_p}(c_T(p))
   \right>_T \\
   & = \frac{\chi_{\lambda}(1)}{p!}
   \sum_{J \in [n]^p} 
   \left<
   \xi_{j_1}(c_T(1))
   \cdots
   \xi_{j_p}(c_T(p))
   e^*_{j_p} \cdots e^*_{j_1}
   \right>_T \\
   & = \frac{\chi_{\lambda}(1)}{p!}
   \left<
   \Xi(c_T(1))
   \cdots
   \Xi(c_T(p))
   \right>_T.
\end{align*}
Here we put 
$\langle \varphi^*,\varphi \rangle_T = \rho_{\lambda}(\langle \varphi^*,\varphi \rangle)_{TT}$
for $\varphi^* \in T^{\circ}_p(V^*)$ and $\varphi \in T_p(V)$.
\end{proof}

%%%%%%%%%%%%%%%%%%%%%%%%%%%%%%%%%%%%%%%%%%%%%%%%%%%%%%%%%%%%%%%%%%%%%%%%%%%%%%%%%
\subsection{}\label{subsec:second_expression_of_G}
%%%%%%%%%%%%%%%%%%%%%%%%%%%%%%%%%%%%%%%%%%%%%%%%%%%%%%%%%%%%%%%%%%%%%%%%%%%%%%%%%
%
By a similar calculation, we can obtain expressions of $G_{\lambda}$
corresponding to the row-immanant as follows.
We consider $\xi^*_i(u) = \sum_{j=1}^n E_{ij}(u) e^*_j$ 
in $\bar{T}^{\circ}(V^*) \otimes U(\mathfrak{gl}_n)$.
Then the following relation holds:

\begin{lemma} \sl 
   We have
   $$
      \xi^*_i (y_{l+1}) 
	  \xi^*_j (y_l) 
      = s_1
      \xi^*_j (y_{l+1}) 
	  \xi^*_i (y_l)
   $$
   and in particular
   $$
      \sigma^{-1}
      \xi^*_{i_1} (y_p) 
	  \xi^*_{i_2} (y_{p-1})
	  \cdots 
      \xi^*_{i_p} (y_1) 
      = 
      \xi^*_{i_{\sigma(1)}} (y_p)
      \xi^*_{i_{\sigma(2)}} (y_{p-1})
      \cdots 
      \xi^*_{i_{\sigma(p)}} (y_1).
   $$
\end{lemma}

Noting this, we put
\begin{align*}
   G^{\prime IJ}
   & = 
   \langle
   \xi^*_{i_1} (y_p)
   \xi^*_{i_2} (y_{p-1})
   \cdots 
   \xi^*_{i_p} (y_1), 
   e_{j_p} \cdots e_{j_1}
   \rangle \\
   & =
   \sum_{\sigma \in S_p} 
   E_{i_1 j_{\sigma(1)}} (x^{\circ}_p) 
   E_{i_2 j_{\sigma(2)}} (x^{\circ}_{p-1}) 
   \cdots 
   E_{i_p j_{\sigma(p)}} (x^{\circ}_1) 
   \sigma^{-1} \\
   & =
   \sum_{\sigma \in S_p} 
   E_{i_p j_{\sigma(p)}} (x_p) 
   E_{i_{p-1} j_{\sigma(p-1)}} (x_{p-1}) 
   \cdots 
   E_{i_1 j_{\sigma(1)}} (x_1) 
   \sigma^{-1}
\end{align*}
and 
$$
   G'_p
   = \frac{1}{p!}
   \sum_{I \in [n]^p}
   G^{\prime II}, \qquad
   G'_{\lambda} = \chi_{\lambda}(G'_p).
$$
These $G'_p$ and $G'_{\lambda}$ are central 
in $\mathbb{C}S_p \otimes U(\mathfrak{gl}_n)$ and $U(\mathfrak{gl}_n)$, respectively.
We can also express these as
\begin{align}
\label{eq:expression_of_G^prime_p_corresponding_to_symm_imm}
   G'_p 
   &= \frac{1}{p!^2}
   \sum_{I \in [n]^p} \sum_{\sigma, \sigma' \in S_p} 
   \sigma
   E_{i_{\sigma(p)} i_{\sigma'(p)}}(x_p) 
   \cdots 
   E_{i_{\sigma(1)} i_{\sigma'(1)}}(x_1) 
   \sigma^{\prime-1}, 
   \displaybreak[0]\\
\label{eq:expression_of_G^prime_lambda_corresponding_to_symm_imm}
   G'_{\lambda} 
   & = \frac{\chi_{\lambda}(1)}{p!}
   \sum_{I \in [n]^p} \sum_{\sigma \in S_p} 
   E_{i_p i_{\sigma(p)}}(c_T(p))
   \cdots
   E_{i_1 i_{\sigma(1)}}(c_T(1))
   \rho_{\lambda}(\sigma^{-1})_{TT} \\
\notag
   & = \frac{\chi_{\lambda}(1)}{p!}
   \left<
   \Xi(c_T(p)) \cdots \Xi(c_T(1))
   \right>_T \\
\notag
   & = \frac{1}{p!}
   \left<
   \Xi(c_T(p)) \cdots \Xi(c_T(1))
   \right>_{\lambda}.
\end{align}
Hence we have $G_{\lambda} = G'_{\lambda}$,
because $\Xi(c_T(p)) \cdots \Xi(c_T(1)) = \Xi(c_T(1)) \cdots \Xi(c_T(p))$.
Noting Lemma~\ref{lem:tilde_s_lambda}, we also have $G_p = G'_p$.

\begin{remark}
   It seems not easy to see this equality $G_p = G'_p$
   without using the expressions of $G_{\lambda}$ and $G'_{\lambda}$ in terms of contents.
\end{remark}

%%%%%%%%%%%%%%%%%%%%%%%%%%%%%%%%%%%%%%%%%%%%%%%%%%%%%%%%%%%%%%%%%%%%%%%%%%%%%%%%%
\subsection{}\label{subsec:expression_of_G^circ}
%%%%%%%%%%%%%%%%%%%%%%%%%%%%%%%%%%%%%%%%%%%%%%%%%%%%%%%%%%%%%%%%%%%%%%%%%%%%%%%%%
%
We can deal with $G^{\circ}_{\lambda}$ and $G^{\circ\prime}_{\lambda}$ similarly.
We define
$\gamma^*_j(u) \in \bar{T}^{\circ}(V^*) \otimes U(\mathfrak{gl}_n)$ and
$\gamma_j(u) \in \bar{T}(V) \otimes U(\mathfrak{gl}_n)$ by
$$
   \gamma^*_j(u) = \sum_{i=1}^n E_{ij}(u) e^*_i, \qquad
   \gamma_i(u) = \sum_{j=1}^n e_j E_{ij}(u).
$$
Then we have the following commutation relations:

\begin{lemma}\sl
   We have 
   $$
      \gamma^*_i(-y_{l+1}) \gamma^*_j(-y_l) 
      = s_1 \gamma^*_j(-y_{l+1}) \gamma^*_i(-y_l), \quad
      \gamma_i(-y_l) \gamma_j(-y_{l+1}) 
      = \gamma_j(-y_l) \gamma_i(-y_{l+1}) s_1
   $$
   and in particular
   \begin{align*}
      \sigma^{-1}
      \gamma^*_{j_1}(-y_p) \gamma^*_{j_2}(-y_{p-1}) \cdots \gamma^*_{j_p}(-y_1) 
      & = \gamma^*_{j_{\sigma(1)}}(-y_p) \gamma^*_{j_{\sigma(2)}}(-y_{p-1}) \cdots 
      \gamma^*_{j_{\sigma(p)}}(-y_1), \\
      \gamma_{i_p}(-y_1) \gamma_{i_{p-1}}(-y_2) \cdots \gamma_{i_1}(-y_p) \sigma
      & = \gamma_{i_{\sigma(p)}}(-y_1) \gamma_{i_{\sigma(p-1)}}(-y_2) \cdots 
      \gamma_{i_{\sigma(1)}}(-y_p).
   \end{align*}
\end{lemma}

Noting this, we put
\begin{align*}
   G^{\circ IJ}
   & = \langle 
   \gamma^*_{j_1}(-y_p) \gamma^*_{j_2}(-y_{p-1}) \cdots \gamma^*_{j_p}(-y_1), 
   e_{i_p} \cdots e_{i_1} 
   \rangle \\
   & = \sum_{\sigma \in S_p}
   E_{i_{\sigma(1)}j_1}(-x^{\circ}_p)
   E_{i_{\sigma(2)}j_2}(-x^{\circ}_{p-1})
   \cdots
   E_{i_{\sigma(p)}j_p}(-x^{\circ}_1) \sigma^{-1} \\
   & = \sum_{\sigma \in S_p}
   E_{i_{\sigma(p)}j_p}(-x_p)
   E_{i_{\sigma(p-1)}j_{p-1}}(-x_{p-1})
   \cdots
   E_{i_{\sigma(1)}j_1}(-x_1) \sigma^{-1}, \\
   G^{\circ \prime IJ}
   & = \langle 
   e^*_{j_1} \cdots e^*_{j_p},
   \gamma_{i_p}(-y_1) \gamma_{i_{p-1}}(-y_2) \cdots \gamma_{i_1}(-y_p)
   \rangle \\
   & = \sum_{\sigma \in S_p}
   \sigma
   E_{i_p j_{\sigma(p-1)}}(-x^{\circ}_1)
   E_{i_{p-1} j_{\sigma(p-2)}}(-x^{\circ}_2)
   \cdots
   E_{i_1 j_{\sigma(1)}}(-x^{\circ}_p)\\
   & = \sum_{\sigma \in S_p}
   \sigma
   E_{i_1 j_{\sigma(1)}}(-x_1)
   E_{i_2 j_{\sigma(2)}}(-x_2)
   \cdots
   E_{i_p j_{\sigma(p)}}(-x_p)
\end{align*}
and moreover
$$
   G^{\circ}_p 
   = \frac{1}{p!} \sum_{I \in [n]^p} G^{\circ II}, \qquad
   G^{\circ \prime}_p 
   = \frac{1}{p!} \sum_{I \in [n]^p} G^{\circ \prime II}, \qquad
   G^{\circ}_{\lambda}
   = \chi_{\lambda}(G^{\circ}_p), \qquad
   G^{\circ \prime}_{\lambda}
   = \chi_{\lambda}(G^{\circ \prime}_p).
$$
Then we have
\begin{align}
\label{eq:expression_of_G^circ_p_corresponding_to_symm_imm}
   G^{\circ}_p 
   &= \frac{1}{p!^2}
   \sum_{J \in [n]^p} \sum_{\sigma, \sigma' \in S_p} 
   \sigma'
   E_{j_{\sigma(p)} j_{\sigma'(p)}}(-x_p) 
   \cdots 
   E_{j_{\sigma(1)} j_{\sigma'(1)}}(-x_1) 
   \sigma^{-1}, 
   \displaybreak[0] \\
\label{eq:expression_of_G^circprime_p_corresponding_to_symm_imm}
   G^{\circ \prime}_p 
   &= \frac{1}{p!^2}
   \sum_{J \in [n]^p} \sum_{\sigma, \sigma' \in S_p} 
   \sigma'
   E_{j_{\sigma(1)} j_{\sigma'(1)}}(-x_1) 
   \cdots 
   E_{j_{\sigma(p)} j_{\sigma'(p)}}(-x_p) 
   \sigma^{-1}, 
   \displaybreak[0] \\
\label{eq:expression_of_G^circ_lambda_corresponding_to_symm_imm}
   G^{\circ}_{\lambda}
   & = \frac{\chi_{\lambda}(1)}{p!} 
   \sum_{J \in [n]^p} \sum_{\sigma \in S_p}
   E_{j_{\sigma(p)} j_p}(-c_T(p))
   \cdots
   E_{j_{\sigma(1)} j_1}(-c_T(1))
   \rho_{\lambda}(\sigma^{-1})_{TT} \\
\notag
   & = \frac{\chi_{\lambda}(1)}{p!} 
   \langle 
   \Xi(-c_T(p)) \Xi(-c_T(p-1)) \cdots \Xi(-c_T(1))
   \rangle_T \\
\notag
   & = \frac{1}{p!} 
   \langle 
   \Xi(-c_T(p)) \Xi(-c_T(p-1)) \cdots \Xi(-c_T(1))
   \rangle_{\lambda}, 
   \displaybreak[0] \\
\label{eq:expression_of_G^circprime_lambda_corresponding_to_symm_imm}
   G^{\circ \prime}_{\lambda}
   & = \frac{\chi_{\lambda}(1)}{p!} 
   \sum_{J \in [n]^p} \sum_{\sigma \in S_p}
   \rho_{\lambda}(\sigma)_{TT}
   E_{j_1 j_{\sigma(1)}}(-c_T(1))
   \cdots
   E_{j_p j_{\sigma(p)}}(-c_T(p)) \\
\notag
   & = \frac{\chi_{\lambda}(1)}{p!} 
   \langle 
   \Xi(-c_T(1)) \Xi(-c_T(2)) \cdots \Xi(-c_T(p))
   \rangle_T \\
\notag
   & = \frac{1}{p!} 
   \langle 
   \Xi(-c_T(1)) \Xi(-c_T(2)) \cdots \Xi(-c_T(p))
   \rangle_{\lambda}. 
\end{align}

Combining these, we see that $G^{\circ}_{\lambda} = G^{\circ\prime}_{\lambda}$
and in particular $G^{\circ}_p = G^{\circ\prime}_p$.
The quantity $G^{\circ}_{\lambda} = G^{\circ\prime}_{\lambda}$ 
is known as the ``quantum immanants'' (\cite{O1}, \cite{OO}).
Thus it is natural to call $G^{\circ}_p = G^{\circ\prime}_p$ the ``quantum preimmanant.''

We can also deduce these relations from the results 
in Sections~\ref{subsec:first_expression_of_G} and \ref{subsec:second_expression_of_G}
by applying the automorphism of $U(\mathfrak{gl}_n)$ 
defined by $E_{ij} \mapsto -E_{ji}$.

Most of various relations in this Section~\ref{subsec:expression_of_G^circ}
were already given in \cite{O1}.
However, it is interesting that we can deduce these relations
by simple calculations using our algebras.

%%%%%%%%%%%%%%%%%%%%%%%%%%%%%%%%%%%%%%%%%%%%%%%%%%%%%%%%%%%%%%%%%%%%%%%%%%%%%%%%%
\subsection{}
%%%%%%%%%%%%%%%%%%%%%%%%%%%%%%%%%%%%%%%%%%%%%%%%%%%%%%%%%%%%%%%%%%%%%%%%%%%%%%%%%
%
Combining these relations, 
we can also express $G^{\circ}_{\lambda}$ as
\begin{equation}\label{eq:expression_of_G^circ_lambda_corresponding_to_the_column_immanant}
   G^{\circ}_{\lambda} 
   = \sum_{I \in \bibinom{[n]}{p}^{\circ}} \frac{1}{I!} 
   \sum_{\sigma \in S_p} 
   \rho_{\lambda}(\sigma E_{i_1 i_{\sigma(1)}}(-x_1) \cdots E_{i_p i_{\sigma(p)}}(-x_p)).
\end{equation}
Here we put
$$
   \bibinom{[n]}{p}^{\circ}
   = \{ (i_1,\ldots,i_p) \in [n]^p \,|\, i_1 \geq \cdots \geq i_n \}.
$$
This expression was given in \cite{O1}.
Using this expression,
we can calculate the eigenvalue of $G^{\circ}_{\lambda}$ 
on the irreducible representation $\pi_{\mu}$ of $\mathfrak{gl}_n$
determined by a partition $\mu$.
The result is as follows
(see Section~3.7 of \cite{O1} for the details of this calculation):
$$
   \pi_{\mu}(G_{\lambda}) = \sum_{T} \prod_{\alpha} (\lambda_{T(\alpha)} - c_{\alpha}).
$$
Here $T$ runs over all reverse semistandard tableaux of shape $\lambda$ 
with entries in $\{ 1,\ldots,n \}$,
and $\alpha$ runs over all cells in the Young tableau $\mu$.
Moreover $T(\alpha)$ means the number in $\alpha$, 
and $c_{\alpha}$ means the content of the cell $\alpha$.

Similarly, $G_{\lambda}$ is expressed as
\begin{equation}\label{eq:expression_of_G_lambda_corresponding_to_the_column_immanant}
   G_{\lambda} = \sum_{I \in \bibinom{[n]}{p}} \frac{1}{I!} 
   \sum_{\sigma \in S_p} \rho_{\lambda}(\sigma E_{i_{\sigma(1)} i_1}(x_1) \cdots E_{i_{\sigma(p)} i_p}(x_p)),
\end{equation}
and we can calculate its eigenvalue as follows:
$$
   \pi_{\mu}(G_{\lambda}) = \sum_{T} \prod_{\alpha} (\lambda_{T(\alpha)} + c_{\alpha}).
$$
The proof is almost the same.
Here $T$ runs over all semistandard tableaux of shape $\lambda$ 
with entries in $\{ 1,\ldots,n \}$,
and $\alpha$ runs over all cells in the Young tableau $\mu$.

On one hand,
these expressions
(\ref{eq:expression_of_G^circ_lambda_corresponding_to_the_column_immanant}) 
and
(\ref{eq:expression_of_G_lambda_corresponding_to_the_column_immanant}) 
can be regarded as the counterparts of the definition of the Capelli elements 
(\ref{eq:definition_of_Capelli_el}),
and we can calculate the eigenvalues easily under these expressions.
On the other hand, the expressions of 
$G_p$, $G_{\lambda}$ and $G^{\circ}_p$, $G^{\circ}_{\lambda}$ 
given in Proposition~\ref{prop:expressions_of_G^IJ_and_G_p},
Theorem~\ref{thm:expression_of_G_lambda_corresponding_to_the_symmetrized_immanant}
and 
(\ref{eq:expression_of_G^prime_p_corresponding_to_symm_imm}),
(\ref{eq:expression_of_G^prime_lambda_corresponding_to_symm_imm}),
(\ref{eq:expression_of_G^circ_p_corresponding_to_symm_imm}),
(\ref{eq:expression_of_G^circ_lambda_corresponding_to_symm_imm}),
(\ref{eq:expression_of_G^circprime_p_corresponding_to_symm_imm}),
(\ref{eq:expression_of_G^circprime_lambda_corresponding_to_symm_imm})
are corresponding to the expression of the Capelli elements given 
in Corollary~\ref{cor:relation_between_column_det_and_symm_det}.
Under these expressions,
we can show the centrality of these elements easily.

Similar phenomena on contrastive expressions are also known
in the universal enveloping algebras $U(\mathfrak{o}_n)$ and $U(\mathfrak{sp}_n)$.
First, central elements of $U(\mathfrak{o}_n)$ expressed in terms of the column-determinant
were given in \cite{Wa},
and we can easily calculate their eigenvalues.
These elements can also be expressed in terms of the symmetrized determinant,
and we can easily show their centrality under this expression (\cite{I5}).
Similar central elements are given in $U(\mathfrak{sp}_n)$  in terms of permanents (\cite{I6}).
It is natural to expect good bases of the centers of $U(\mathfrak{o}_n)$ and $U(\mathfrak{sp}_n)$
which contains these elements 
and can be regarded as analogues of the quantum immanants.
The author hopes that our extensions of the tensor algebra are useful to construct such bases. 

%%%%%%%%%%%%%%%%%%%%%%%%%%%%%%%%%%%%%%%%%%%%%%%%%%%%%%%%%%%%%%%%%%%%%%%%%%%%%%%%%%
%
\section{Higher Capelli identity and its analogue on tensor algebras}
\label{sec:higher_Capelli}
%
%%%%%%%%%%%%%%%%%%%%%%%%%%%%%%%%%%%%%%%%%%%%%%%%%%%%%%%%%%%%%%%%%%%%%%%%%%%%%%%%%
%
The method in the previous section 
is also useful to show a Capelli type identity for the quantum immanants,
namely the ``higher Capelli identity.'' 
Several proofs are already known for this identity
(\cite{O1}, \cite{O2}, \cite{N2}, \cite{M2}),
and the proof given in \cite{M2} is particularly simple. 
Our proof is similarly simple and parallel to
the easy proof of the original Capelli identities using the exterior calculus given in \cite{I5} or \cite{I3} (cf. Section~\ref{sec:Capelli_and_FFT} in this article).

Furthermore, we give the ``higher version'' of Theorem~\ref{thm:Capelli_identity_in_T}
in $\mathcal{L}(\mathbb{C}^n \otimes \mathbb{C}^{n'})$.
We can also assemble them in terms of the quantum preimmanants.

%%%%%%%%%%%%%%%%%%%%%%%%%%%%%%%%%%%%%%%%%%%%%%%%%%%%%%%%%%%%%%%%%%%%%%%%%%%%%%%%%
\subsection{}\label{subsec:higher_Capelli_identity}
%%%%%%%%%%%%%%%%%%%%%%%%%%%%%%%%%%%%%%%%%%%%%%%%%%%%%%%%%%%%%%%%%%%%%%%%%%%%%%%%%
%
We work in the following situation.
The general linear group $GL_n(\mathbb{C})$ naturally acts on 
$\mathbb{C}^n \otimes \mathbb{C}^{n'}$ as in Section~\ref{sec:Capelli_and_FFT}
and moreover the space $\mathcal{P}(\mathbb{C}^n \otimes \mathbb{C}^{n'})$ 
of all polynomial functions on $\mathbb{C}^n \otimes \mathbb{C}^{n'}$.
The infinitesimal action $\nu$ is expressed as follows:
$$
   \nu(E_{ij}) = \sum_{k=1}^{n'} x_{ik} \partial_{jk}.
$$
Here $x_{ij}$ is  the standard coordinate of $\mathbb{C}^n \otimes \mathbb{C}^{n'}$,
and $\partial_{ij}$ means the partial differentiation $\partial_{ij} = \frac{\partial}{\partial x_{ij}}$.
Let us express this relation as
$\nu(E) = X \, {}^t\!\partial$ using the following matrices 
$E \in \operatorname{Mat}_n(U(\mathfrak{gl}_n))$ and 
$X$, $\partial \in \operatorname{Mat}_{n,n'}(\mathcal{PD}(\mathbb{C}^n \otimes \mathbb{C}^{n'}))$:
$$
   E = (E_{ij})_{1 \leq i,j \leq n}, \qquad
   X = (x_{ij})_{1 \leq i \leq n, \, 1 \leq j \leq n'}, \qquad
   \partial = (\partial_{ij})_{1 \leq i \leq n, \, 1 \leq j \leq n'}.
$$
Then we have the following relation.
This is called the ``higher Capelli identity'' \cite{O1}.

\begin{theorem}\label{thm:higher_Capelli}\sl
   We have
   $$
      \nu(G^{\circ}_{\lambda}) 
	  = \frac{\chi_{\lambda}(1)}{p!^2}
	  \sum_{I \in [n]^p, K \in [n']^p}
	  \operatorname{imm}_{\lambda} X_{IK}
	  \operatorname{imm}_{\lambda} \partial_{IK}.
   $$
\end{theorem}

This relation can be regarded as a generalization of the Capelli identity,
and has an interpretation in the dual pair theory as follows
(see also \cite{O1}, \cite{Ho}, \cite{HU}, \cite{I3}, \cite{I4}, \cite{MN} \cite{U1}).
The general linear group $GL_{n'}(\mathbb{C})$ acts on 
$\mathcal{P}(\mathbb{C}^n \otimes \mathbb{C}^{n'})$ naturally,
and these two actions of $GL_n(\mathbb{C})$ and $GL_{n'}(\mathbb{C})$ form the dual pair.
This tells us the relation
$$
   \nu(U(\mathfrak{gl}_n)^{GL_n(\mathbb{C})}) 
   = \mathcal{PD}(\mathbb{C}^n \otimes \mathbb{C}^{n'})^{GL_n(\mathbb{C}) \times GL_{n'}(\mathbb{C})}
   = \nu'(U(\mathfrak{gl}_{n'})^{GL_{n'}(\mathbb{C})}).
$$
Here $\nu'$ means the infinitesimal action of $GL_{n'}(\mathbb{C})$.
Theorem~\ref{thm:higher_Capelli} can be regarded as a beautiful description of this relation
in terms of bases as vector spaces
(from the symmetry of $\nu$ and $\nu'$, we see that the quantity in this theorem
is also equal to $\nu'(G^{\circ}_{\lambda})$).
Note that we can check that the right hand side of Theorem~\ref{thm:higher_Capelli}
is $GL_n(\mathbb{C}) \times GL_{n'}(\mathbb{C})$-invariant
using the Cauchy--Binet type formula (Proposition~\ref{prop:Cauchy_Binet_for_immanants}). 

Furthermore we can rewrite this in terms of the quantum preimmanants:

\begin{theorem}\label{thm:higher_Capelli_for_preimmanants}\sl
   We have
   $$
      \nu(G^{\circ}_p) 
	  = \frac{1}{p!^2}
	  \sum_{I \in [n]^p, K \in [n']^p}
	  \operatorname{preimm} X_{IK}
	  \operatorname{preimm}^{\circ} \partial_{IK}.
   $$
\end{theorem}

We can also regard these relations
as noncommutative analogues of the Cauchy--Binet type identities in Section~\ref{sec:immanants}
(Propositions~\ref{prop:Cauchy_Binet_for_immanants} and \ref{prop:Cauchy_Binet_for_preimmanants}).

These theorems are shown as follows.
This proof is parallel to that of the original Capelli identity
given in \cite{U5} and \cite{I3} in the exterior calculus.

\begin{proof}[Proof of Theorems~{\sl \ref{thm:higher_Capelli}} 
and {\sl \ref{thm:higher_Capelli_for_preimmanants}}]
Put $V = \mathbb{C}^n$
and consider a basis $e_1,\ldots,e_n$ of $V$
and its dual basis $e^*_1,\ldots,e^*_n$.
We consider the following elements 
in $\bar{T}^{\circ}(V^*) \otimes \mathcal{PD}(\mathbb{C}^n \otimes \mathbb{C}^{n'})$:
$$
   \eta^*_k 
   = \sum_{i=1}^n x_{ik} e^*_i, \qquad
   \gamma^*_j 
   = \sum_{i=1}^n \nu(E_{ij}) e^*_i 
   = \sum_{i=1}^n\sum_{k=1}^{n'} x_{ik} \partial_{jk} e^*_i 
   = \sum_{k=1}^{n'} \eta^*_k \partial_{jk}.
$$
Moreover we put
$$
   \gamma^*_j(u) 
   = \sum_{i=1}^n \nu(E_{ij}(u)) e^*_i 
   = \gamma^*_j + u e^*_j.
$$
A simple calculation tells us
$\nu(E_{ij}) x_{ka} = x_{ka} \nu(E_{ij}) + x_{ia} \delta_{kj}$.
From this, we see the commutation relation
$$
   \gamma^*_j(-y_{k+1}) \eta^*_k 
   = s_1 \eta^*_k \gamma^*_j(-y_k).
$$
Using this and the relation 
$\gamma^*_j(-y_1) = \gamma^*_j = \sum_{k=1}^{n'} \eta^*_k \partial_{jk}$ repeatedly, 
we have
\begin{align*}
   \gamma^*_{j_1}(-y_p) \cdots \gamma^*_{j_{p-1}}(-y_2) \gamma^*_{j_p}(-y_1) 
   & = \sum_{k_p = 1}^{n'} \gamma^*_{j_1}(-y_p) \cdots \gamma^*_{j_{p-1}}(-y_2) 
   \eta^*_{k_p} \partial_{j_p k_p} \\
   & = \sum_{k_p = 1}^{n'} \eta^*_{k_p} 
   \gamma^*_{j_1}(-y_{p-1}) \cdots \gamma^*_{j_{p-1}}(-y_1) \partial_{j_p k_p} \\
   &\qquad \vdots \\
   & = \sum_{k_1,\ldots,k_p = 1}^{n'} \eta^*_{k_p} \cdots \eta^*_{k_1} 
   \partial_{j_1 k_1} \cdots \partial_{j_p k_p}.
\end{align*}
Thus, we have
\begin{align*}
   &
   \sum_{J \in [n]^p} 
   e_{j_p} \cdots e_{j_1} 
   \gamma^*_{j_1}(-y_p) \cdots \gamma^*_{j_p}(-y_1) \\
   & \qquad
   = \sum_{I,J \in [n]^p} \sum_{K \in [n']^p} 
   e_{j_p} \cdots e_{j_1} 
   x_{i_p k_p} \cdots x_{i_1 k_1} 
   \partial_{j_1 k_1} \cdots \partial_{j_p k_p}
   e^*_{i_1} \cdots e^*_{i_p}
\end{align*}
in $\bar{T}(V,V^*) \otimes \mathcal{PD}(\mathbb{C}^n \otimes \mathbb{C}^{n'})$.
Applying $\langle \cdot \rangle$ and dividing by $p!$,
we obtain
\begin{equation}\label{eq:key_relation_for_higher_Capelli}
   \nu(G^{\circ}_p) 
   = \frac{1}{p!^2}
   \sum_{I \in [n]^p} \sum_{K \in [n']^p} \sum_{\sigma,\sigma' \in S_p} 
   \sigma'
   x_{i_{\sigma(p)} k_p} \cdots x_{i_{\sigma(1)} k_1} 
   \partial_{i_{\sigma'(1)} k_1} \cdots \partial_{i_{\sigma'(p)} k_p}
   \sigma^{-1}.
\end{equation}
The left hand side is invariant under the antiautomorphism $t \mapsto t^{\circ}$ 
given in Section~\ref{subsec:preimmanants},
because  $G^{\circ}_p \in Z\mathbb{C}S_p \otimes U(\mathfrak{gl}_n)$.
Hence, we have
$$
   \nu(G^{\circ}_p) 
   = \frac{1}{p!^2}
   \sum_{I \in [n]^p} \sum_{K \in [n']^p} \sum_{\sigma,\sigma' \in S_p} 
   \sigma
   x_{i_{\sigma(p)} k_p} \cdots x_{i_{\sigma(1)} k_1} 
   \partial_{i_{\sigma'(1)} k_1} \cdots \partial_{i_{\sigma'(p)} k_p}
   \sigma^{\prime -1}.
$$
This means Theorem~\ref{thm:higher_Capelli_for_preimmanants}.
Moreover, applying $\chi_{\lambda}$ to (\ref{eq:key_relation_for_higher_Capelli}), we have
$$
   \nu(G^{\circ}_{\lambda}) 
   = \frac{1}{p!^2}
   \sum_{I \in [n]^p} \sum_{K \in [n']^p} \sum_{\sigma,\sigma' \in S_p} 
   \chi_{\lambda}(\sigma'\sigma^{-1})
   x_{i_{\sigma(p)} k_p} \cdots x_{i_{\sigma(1)} k_1} 
   \partial_{i_{\sigma'(1)} k_1} \cdots \partial_{i_{\sigma'(p)} k_p}.
$$
As seen from the first relation of (\ref{eq:relations_for_characters}), 
the right hand side is equal to
\begin{equation*}
   \frac{\chi_{\lambda}(1)}{p!^2}
   \sum_{I \in [n]^p} \sum_{K \in [n']^p}
   \sum_{\sigma \in S_p} 
   \chi_{\lambda}(\sigma^{-1})
   x_{i_{\sigma(p)} k_p} \cdots x_{i_{\sigma(1)} k_1} 
   \sum_{\sigma' \in S_p} 
   \chi_{\lambda}(\sigma')
   \partial_{i_{\sigma'(1)} k_1} \cdots \partial_{i_{\sigma'(p)} k_p}.
\end{equation*}
This means Theorem~\ref{thm:higher_Capelli}.
\end{proof}

\begin{remark}
   The right hand side of Theorem~\ref{thm:higher_Capelli} is also equal to 
   $$
      \frac{1}{p!}
      \sum_{I \in [n]^p} \sum_{K \in [n']^p} \sum_{\sigma \in S_p} \chi_{\lambda}(\sigma)
      x_{i_{\sigma(p)} k_p} \cdots x_{i_{\sigma(1)} k_1} 
      \partial_{i_1 k_1} \cdots \partial_{i_p k_p}.
   $$
\end{remark}

%%%%%%%%%%%%%%%%%%%%%%%%%%%%%%%%%%%%%%%%%%%%%%%%%%%%%%%%%%%%%%%%%%%%%%%%%%%%%%%%%
\subsection{}\label{subsec:higher_Capelli_identity_on_T}
%%%%%%%%%%%%%%%%%%%%%%%%%%%%%%%%%%%%%%%%%%%%%%%%%%%%%%%%%%%%%%%%%%%%%%%%%%%%%%%%%
%
In the same situation as Section~\ref{sec:Capelli_and_FFT},
we have the following relation as an analogue of the higher Capelli identity:

\begin{theorem}\label{thm:higher_Capelli_on_T}\sl
   We have
   $$
      \nu(G^{\circ}_{\lambda}) 
	  = \frac{\chi(1)}{p!^2}
	  \sum_{I \in [n]^p, K \in [n']^p}
	  \operatorname{column-imm}_{\lambda} X_{I^{\circ}K^{\circ}}
	  \operatorname{column-imm}_{\lambda} X^*_{IK}.
   $$
\end{theorem}

We can regard this as the ``higher version'' of Theorem~\ref{thm:Capelli_identity_in_T}.
Namely this describes
the relation between the centers of $U(\mathfrak{gl}_n)$ and $\mathcal{Q}_2$ 
in Corollary~\ref{cor:relation_between_ZU_and_Q2}
in terms of a basis of the center of $U(\mathfrak{gl}_n)$.

Moreover, we have the following relation for the quantum preimmanants:

\begin{theorem}\label{thm:higher_Capelli_for_preimmanants_on_T}\sl
   We have
   $$
      \nu(G^{\circ}_p) 
	  = \frac{1}{p!^2}
	  \sum_{I \in [n]^p, K \in [n']^p}
	  \operatorname{column-preimm} X_{I^{\circ}K^{\circ}}
	  \operatorname{column-preimm}^{\circ} X^*_{IK}.
   $$
\end{theorem}

\begin{proof}[Proof of Theorems~{\sl\ref{thm:higher_Capelli_on_T}} 
and {\sl\ref{thm:higher_Capelli_for_preimmanants_on_T}}]
The proof is almost the same as that of Theorem~\ref{thm:higher_Capelli}.
We consider the following elements in 
$\bar{T}^{\circ}(V^*) \otimes \mathcal{L}(\mathbb{C}^n \otimes \mathbb{C}^{n'})$:
\begin{gather*}
   \eta^*_k 
   = \sum_{i=1}^n L(w_{ik}) e^*_i, \quad
   \gamma^*_j 
   = \sum_{i=1}^n \nu(E_{ij}) e^*_i 
   = \sum_{i=1}^n \sum_{k=1}^{n'} L(w_{ik}) L(w^*_{jk}) e^*_i 
   = \sum_{k=1}^{n'} \eta^*_k L(w^*_{jk}), \\
   \gamma^*_j(u)
   = \sum_{i=1}^n \nu(E_{ij}(u)) e^*_i
   = \gamma^*_j +  u e^*_j. 
\end{gather*}
By a direct calculation, we have
$$
   \nu(E_{ij}) L(w_{ka}) = L(w_{ka}) \nu(E_{ij}) + L(w_{ia}) \delta_{kj}.
$$
From this, we see the commutation relation
$$
   \gamma^*_j(-y_{k+1}) \eta^*_k 
   = s_1 \eta^*_k \gamma^*_j(-y_k).
$$
Here $s_i$ and $y_k$ belong to $\bar{T}^{\circ}(V^*)$
(these should be distinguished from $L(s_i)$ and $L(y_k)$ 
in $\mathcal{L}(\mathbb{C}^n \otimes \mathbb{C}^{n'})$).
Using this and the relation 
$\gamma^*_j(-y_1) = \gamma^*_j = \sum_{k=1}^{n'} \eta^*_k L(w^*_{jk})$
repeatedly, we have
$$
   \gamma^*_{j_1}(-y_p) \cdots \gamma^*_{j_p}(-y_1) 
   = \sum_{K \in [n']^p} \eta^*_{k_p} \cdots \eta^*_{k_1} 
   L(w^*_{j_1 k_1}) \cdots L(w^*_{j_p k_p}).
$$
Thus, we have
\begin{align*}
   &
   \sum_{I,J \in [n]^p} 
   e_{j_p} \cdots e_{j_1} 
   \gamma^*_{j_1}(-y_p) \cdots \gamma^*_{j_p}(-y_1) \\
   & \qquad
   = \sum_{I,J \in [n]^p} 
   e_{j_p} \cdots e_{j_1} 
   L(w_{i_p k_p}) \cdots L(w_{i_1 k_1}) 
   L(w^*_{j_1 k_1}) \cdots L(w^*_{j_p k_p})
   e^*_{i_1} \cdots e^*_{i_p}
\end{align*}
in $\bar{T}(V,V^*) \otimes \mathcal{L}(\mathbb{C}^n \otimes \mathbb{C}^{n'})$.
The remainder is the same as the proof of Theorems~\ref{thm:higher_Capelli}
and~\ref{thm:higher_Capelli_for_preimmanants}.
\end{proof}

%%%%%%%%%%%%%%%%%%%%%%%%%%%%%%%%%%%%%%%%%%%%%%%%%%%%%%%%%%%%%%%%%%%%%%%%%%
%
% References
%
%%%%%%%%%%%%%%%%%%%%%%%%%%%%%%%%%%%%%%%%%%%%%%%%%%%%%%%%%%%%%%%%%%%%%%%%%%
%

\end{document}